\documentclass[11pt, reqno]{amsart}
\usepackage[section]{uniqueness}

\newcommand{\corr}[1]{#1}

\title[A stable-compact method for semilinear elliptic equations]{A stable-compact method for qualitative properties of semilinear elliptic equations}

\author{Henri Berestycki}
\address{HB: Department of Mathematics, University of Maryland, 4176 Campus Dr, College Park, MD 20742, USA}
\address{Centre d'analyse et de math\'{e}matique sociales, EHESS-CNRS, 54 Boulevard Raspail, 75006, Paris, France}
\address{Senior Visiting Fellow, Institute for Advanced Study, Hong Kong University of Science and Technology, Clear Water Bay, Kowloon, Hong Kong}
\email{\href{mailto:hb@ehess.fr}{\tt hb@ehess.fr}}

\author{Cole Graham}
\address{CG: Department of Mathematics, University of Wisconsin--Madison, 480 Lincoln Dr, Madison, WI 53706, USA}
\email{\href{mailto:graham@math.wisc.edu}{\tt graham@math.wisc.edu}}

\begin{document}
\begin{abstract}
  We study the uniqueness of reaction-diffusion steady states in general domains with Dirichlet boundary data.
  Here we consider ``positive'' (monostable) reactions.
  We describe geometric conditions on the domain that ensure uniqueness and we provide complementary examples of nonuniqueness.
  Along the way, we formulate a number of open problems and conjectures. 
  To derive our results, we develop a general framework, the \emph{stable-compact method}, to study qualitative properties of nonlinear elliptic equations.
\end{abstract}

\maketitle

\section{Overview and main results}

We study the uniqueness of steady states of reaction-diffusion equations in general domains with Dirichlet boundary.
These steady states solve semilinear elliptic equations of the form
\begin{equation}
  \label{eq:main}
  \begin{cases}
    -\Delta u = f(u) & \text{in } \Omega,\\
    u = 0 & \text{on } \partial \Omega
  \end{cases}
\end{equation}
in domains $\Omega \subset \R^d$, which need not be bounded.

The classification of solutions of \eqref{eq:main} is a fundamental question in semilinear elliptic theory.
It is, moreover, prerequisite to understanding the dynamics of the parabolic form of \eqref{eq:main}, which models a host of systems in the natural sciences.
In~\cite{BG24}, we considered \eqref{eq:main} when the reaction $f$ is of strong-KPP type.
There, we found that positive bounded solutions are unique under quite general conditions on $\Omega$.
In contrast, we showed that slightly weaker assumptions on the reaction can easily lead to multiple solutions.
Thus, the classification of solutions of \eqref{eq:main} with general reactions is more complex.
Here, we take up this question.

We assume that the nonlinearity ${f \colon [0, \infty) \to \R}$ is $\m{C}^{1,\gamma}$ for some $\gamma \in (0, 1]$ and $f(0) = f(1) = 0$.
As a consequence, $0$ solves \eqref{eq:main} and $1$ is a supersolution, in the sense that it satisfies \eqref{eq:main} with $\geq$ in place of equalities.
We also assume that $f|_{(1, \infty)} < 0$ and $f'(1) < 0$, so that the reaction drives large values down toward its stable root $1$.
Then the maximum principle implies that all positive bounded solutions of \eqref{eq:main} take values between $0$ and $1$.
We are thus primarily interested in the behavior of $f$ on $(0, 1)$.
We say that a reaction $f$ is \emph{positive} if
\begin{enumerate}[label = \textup{(P)}, leftmargin = 5em, labelsep = 1em, itemsep= 1ex, topsep = 1ex]
\item
  \label{hyp:positive}
  $f|_{(0,1)} > 0$ and $f'(0^+) > 0$.
\end{enumerate}
This is sometimes termed ``monostability.''
We use distinct terminology to emphasize the positive derivative at $u = 0$, which plays a significant role in our analysis.
In many sources, monostability only denotes the first condition in \ref{hyp:positive}.

For the domain, we assume that $\Omega$ is open, nonempty, connected, and uniformly $\m{C}^{2, \gamma}$ smooth.
For a precise definition of this notion, see Definition~A.1 in~\cite{BG24}.
Here, we merely note that this hypothesis includes a uniform interior ball condition.
\medskip

We now present our main contributions.
We begin by establishing uniqueness on several classes of structured domains.
\begin{definition}
  \label{def:ES}
  A domain $\Omega$ is \emph{exterior-star} if $\Omega^\cc = \R^d \setminus \Omega$ is star-shaped.
  It is \emph{strongly exterior-star} if $\Omega^\cc$ is star-shaped about a nonempty open set of points.
\end{definition}
\begin{theorem}
  \label{thm:ES}
  Suppose $\Omega$ is strongly exterior-star and $\Omega^\cc$ is compact.
  Then \eqref{eq:main} has a unique positive bounded solution.
\end{theorem}
\noindent
This is a simplified form of Theorem~\ref{thm:ES-full} below, which also covers some domains with unbounded complements.

The uniqueness in Theorem~\ref{thm:ES} is somewhat surprising, as it is not the norm.
Indeed, on every bounded domain, there exists a positive reaction (depending on the domain) such that \eqref{eq:main} supports multiple positive solutions~\cite[Proposition~1.4]{BG24}.
On the other hand, if we hold the reaction fixed, solutions on the one-dimensional interval $(0, L)$ are unique provided $L$ is sufficiently large~\cite[Lemma~2.5]{BG22a}.
Here, we extend this result to dilation in multiple dimensions.
\begin{theorem}
  \label{thm:dilation}
  Fix a positive reaction $f$ and a domain $\Omega$.
  There exists $\ubar{\kappa}(f, \Omega) > 0$ such that for all $\kappa \geq \ubar\kappa$, \eqref{eq:main} has a unique bounded positive solution on the dilated domain $\kappa \Omega$.
\end{theorem}
We next consider uniqueness on epigraphs: domains bounded by the graph of a function.
Given $\phi \colon \R^{d-1} \to \R$, its epigraph is the open set
\begin{equation*}
  \Omega \coloneqq \big\{(x', y) \in \R^{d - 1} \times \R\mid y > \phi(x')\big\}.
\end{equation*}
In Theorem~1.2(d) of~\cite{BCN97b}, the first author, Caffarelli, and Nirenberg showed that \eqref{eq:main} has a unique positive bounded solution on $\Omega$ provided $\phi$ is \emph{uniformly Lipschitz}: $\op{Lip} \phi < \infty$.
We extend this uniqueness to a much broader class of epigraphs.
Our strongest result is somewhat technical, so we defer it to Section~\ref{sec:epigraphs}.
Here, we illustrate the general result through a particularly evocative example.
\begin{definition}
  \label{def:flat}
  We say the epigraph $\Omega$ is \emph{asymptotically flat} if for all $1 \leq i,j \leq d-1$,
  \begin{equation*}
    \partial_i\,\left(\frac{\partial_j \phi}{\sqrt{\abs{\nab \phi}^2 + 1}}\right) \to 0 \quad \text{as } \abs{x'} \to \infty.
  \end{equation*}
\end{definition}
That is, the curvature of the boundary $\partial \Omega$ vanishes at infinity.
Many epigraphs of interest are asymptotically flat but not uniformly Lipschitz; the parabola $\{y > x^2\}$ is a natural example.
This distinction arises whenever $\phi$ grows superlinearly at infinity in a consistent fashion.
We show that such superlinear growth does not impede uniqueness.
\begin{theorem}
  \label{thm:flat}
  If $\Omega$ is an asymptotically flat epigraph, then \eqref{eq:main} admits a unique bounded solution $u$.
\end{theorem}
\noindent
This is a consequence of the more general Theorem~\ref{thm:AUL} presented below.

In the plane, convex epigraphs are automatically asymptotically flat.
\begin{corollary}
  \label{cor:convex}
  If $d = 2$ and $\Omega$ is a convex epigraph, then \eqref{eq:main} has a unique positive bounded solution.
\end{corollary}
The exterior-star and epigraph properties are the central structural assumptions in Theorems~\ref{thm:ES} and \ref{thm:flat}.
However, both include additional technical conditions: the former assumes that $\Omega^\cc$ is compact, while the latter assumes that $\partial\Omega$ flattens at infinity.
It is not clear that these technical conditions are necessary.
We are led to ask: does uniqueness hold on any exterior-star domain?
On any epigraph?
We collect a number of such open problems in Section~\ref{sec:open}.

On the other hand, the fundamental structural assumptions in Theorems~\ref{thm:ES}--\ref{thm:flat} \emph{are} essential.
If one relaxes these assumptions in a bounded region, nonuniqueness can arise.
We state an informal result here; for a rigorous version, see Theorem~\ref{thm:pocket-precise} below.
\begin{theorem}
  \label{thm:pocket}
  Given a domain $\Omega_0$, if we attach a ``pocket'' to $\Omega_0$ via a sufficiently narrow bridge, then there exists a positive reaction $f$ such that \eqref{eq:main} admits multiple positive bounded solutions on the composite domain.
\end{theorem}
\noindent
We depict this operation in Figure~\ref{fig:pocket}.
This construction was introduced by Matano \cite{Matano} for bistable reactions with Neumann conditions.
It has been extended to bistable Dirichlet problems~\cite{DY,NW}, and here we apply it to positive reactions.
\begin{figure}[t]
  \centering
  \includegraphics[width = 0.5\linewidth]{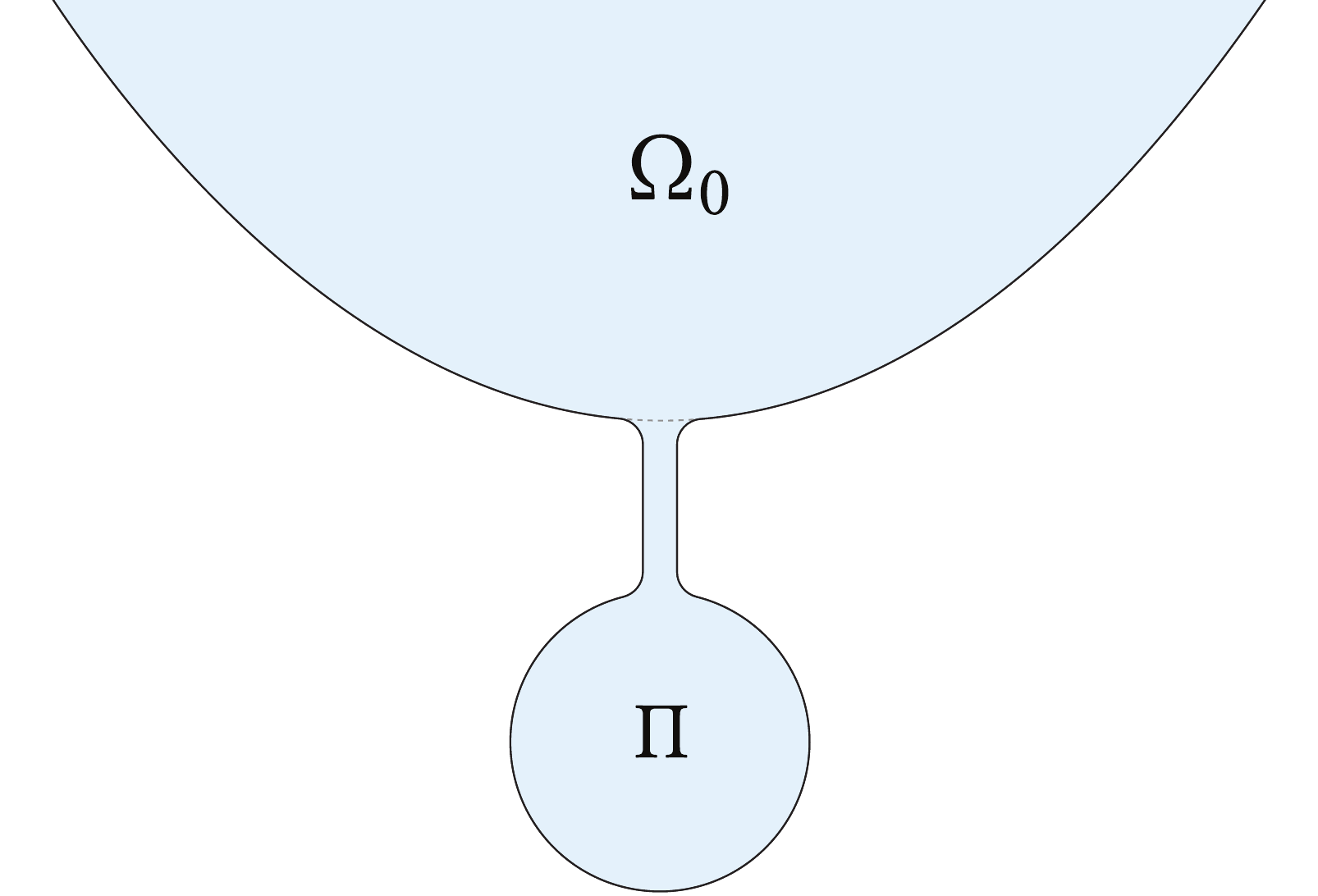}
  \caption[Domain with a pocket]{A domain $\Omega_0$ augmented by a pocket $\Pi$.}
  \label{fig:pocket}
\end{figure}

Theorem~\ref{thm:pocket} demonstrates the importance of the structural assumptions in our prior results.
For example, we can attach a pocket to an asymptotically flat epigraph to produce multiple solutions.
Thus even a compact violation of the epigraph structure suffices for multiplicity.

In our proof of Theorem~\ref{thm:pocket}, we construct two distinct solutions.
Taking a different approach, we show that in some cases \eqref{eq:main} can admit uncountably many solutions.
\begin{proposition}
  \label{prop:cylinder}
  Let $f$ be a $\m{C}^2$ positive reaction such that $f''(0^+) > 0.$
  Then there exists $L > 0$ such that \eqref{eq:main} admits a positive bounded solution on the strip $\R \times (0, L)$ that varies in the first coordinate.
  As a consequence, \eqref{eq:main} admits a continuum of distinct positive bounded solutions.
\end{proposition}
The strip enjoys a translation invariance, so we can view this proposition as a form of symmetry breaking: a symmetric domain supports asymmetric solutions.
This shows that positive reactions behave quite differently on Dirichlet strips than on the line.
After all, on $\R$, the only positive bounded solution of \eqref{eq:main} is the constant $1$.
In contrast, we show that boundary absorption on the strip can cause a positive reaction to exhibit \emph{bistable} behavior.
And indeed, bistable reactions admit oscillatory solutions on $\R$ that are not translation-invariant.
For further discussion in this direction, see Section~\ref{sec:cylinder}.

Technically, we approach Proposition~\ref{prop:cylinder} through the lens of ``spatial dynamics'' \cite{Kirchgassner,IMD}.
We think of the first coordinate as time and view \eqref{eq:main} as a second-order dynamical system.
We then seek nontrivial limit cycles.
This effort is complicated by the fact that the phase space is infinite-dimensional---it consists of functions on $(0, L)$.
The theory of spatial dynamics is well-equipped to treat this difficulty, and we are able to prove Proposition~\ref{prop:cylinder} via standard methods.
\medskip

The above results are linked by a common perspective that we term the ``stable-compact'' method.
This is a general approach to the qualitative properties of elliptic equations.
We focus on uniqueness, but other properties naturally arise in both our arguments and conclusions, including symmetry, monotonicity, and stability.

The method rests on the decomposition of the domain $\Omega$ into a ``stable'' part and a ``compact'' part.
In the former, solutions of \eqref{eq:main} are linearly stable and thus obey a maximum principle.
This is conducive to uniqueness, so we can focus on the complementary compact part.
There, solutions enjoy some form of compactness relative to a context-dependent deformation.
(The compact part may be unbounded; compactness refers to solutions, not to the domain.)
Using this deformation, compactness and the strong maximum principle yield uniqueness.
In Section~\ref{sec:stable-compact}, we examine the proofs of our main results through this stable-compact lens.

Given its generality, the stable-compact method naturally encompasses earlier work.
For example, as discussed in Section~\ref{sec:stable-compact}, it can be discerned in the method of moving planes.
We anticipate that the framework will prove of use in a variety of contexts in elliptic and parabolic theory.
\medskip

Naturally, our efforts to classify solutions of certain semilinear elliptic problems intersect an enormous body of literature.
Here, we highlight a handful of connections that strike us as particularly germane.
For a more complete view of the subject, we direct the reader to the references therein.

\subsubsection*{Structured domains}
In this paper, we tackle a rather broad class of reactions by imposing various structural conditions on the domain.
We owe much to the rich literature on the moving plane and sliding methods starting with the seminal works of Alexandrov~\cite{Alexandrov}, Serrin~\cite{Serrin}, and Gidas, Ni, and Nirenberg~\cite{GNN}.
In particular, we draw inspiration from the formulations of the moving plane and sliding methods of the first author and Nirenberg~\cite{BN91} and Dancer~\cite{Dancer92}.

Further results closely related to the present work include those of Esteban and Lions on monotonicity in coercive epigraphs~\cite{EL} and the extensive collaboration of the first author with Caffarelli and Nirenberg, who studied qualitative elliptic properties in half-spaces~\cite{BCN93,BCN97a}, cylinders~\cite{BCN96}, and epigraphs~\cite{BCN97b}.
In particular, our Theorem~\ref{thm:AUL} is a direct generalization of the results of \cite{BCN97b} on uniformly Lipschitz epigraphs.

\subsubsection*{Structured reactions}
Alternatively, one can consider structured \emph{reactions} on general domains.
Rabinowitz took this approach in~\cite{Rabinowitz}, which established uniqueness for ``strong-KPP'' reactions on all smooth bounded domains.
We introduced the terminology ``strong-KPP'' in~\cite{BG24} to distinguish a concavity-like property of $f$ from the weaker ``KPP condition''; see Definition~\ref{def:KPP} for details.
In~\cite{BG24}, we used the stable-compact method (without calling it such) to study strong-KPP uniqueness in general \emph{unbounded} domains.

\subsubsection*{Other boundary conditions}
Here we treat Dirichlet boundary conditions, but the Neumann problem is also well-motivated by applications.
Neumann steady states are much simpler: Rossi showed that $1$ is the unique positive bounded steady state on general domains~\cite[Corollary~2.8]{Rossi}.
His approach built on that of the first author, Hamel, and Nadirashvili, who established the same for (weak) KPP reactions in domains satisfying a mild geometric condition~\cite[Theorem~1.7]{BHN}.

One can also consider the intermediate Robin boundary condition.
In~\cite{BG24}, we were able to treat Dirichlet and Robin conditions in a unified framework for strong-KPP reactions.
However, many of the methods we employ in the present work do not readily extend to the Robin problem.
And indeed, most of the results we derive here are open under Robin conditions; we discuss this further in Section~\ref{sec:open}.

\subsubsection*{Stable solutions for other reactions}
We focus on positive reactions satisfying \ref{hyp:positive}, which see wide use in ecology; however, other applications call for other nonlinearities.
Sign-changing bistable reactions are essential examples arising in biology, materials science, and physics.
Even on the line, bistable reactions support infinitely many solutions, which inspires Proposition~\ref{prop:cylinder} above.
To manage this menagerie, one can consider solutions satisfying additional properties such as stability or monotonicity.
The latter is the subject of de Giorgi's celebrated conjecture that monotone bistable solutions on the whole space are one-dimensional.
This deep conjecture has spurred a great deal of work in both the positive~\mbox{\cite{GG,AC,Savin,HLSWW}} and negative directions~\cite{dPKW}.
We highlight remarkable recent results of Liu, Wang, Wei, and Wu~\cite{LWWW}, who classify stable steady states on the whole space for a very general class of ``unbalanced'' sign-changing reactions.

In a different direction, some applications call for nonnegative reactions outside the ``positive'' class discussed in the present work.
For example, the pure-power nonlinearity $f(s) = s^\al$ (which does not vanish at $1$) yields the Lane--Emden equation from mathematical physics~\cite{GS,FV,CLZ,DFP}.
As in the bistable setting, much prior work on nonnegative reactions focuses on solutions enjoying some form of stability.
For example, under dimension restrictions, Dupaigne and Farina~\cite{DupFar} have classified the stable steady states of nonnegative reactions in the whole and half-space, as well as in coercive epigraphs under additional conditions on the reaction.

A number of these works also treat solutions that are only stable outside a compact set~\cite{DF,Farina,DupFar}.
While this bears some resemblance to our stable-compact approach, the type of result we obtain here is of a different nature.
Indeed, we use ``compactness'' to refer to properties of the \emph{solution}, not the domain.
Moreover, our results do not assume a stable-compact decomposition as an abstract hypothesis.
Rather, we demonstrate that many natural problems intrinsically support a stable-compact structure.
The identification of said structure is a central step in our arguments, and we outline a general method to tackle the issue in Section~\ref{sec:stable-compact}.

\subsection*{Organization}
We devote the first three sections of the body of the paper to uniqueness on exterior-star domains (Section~\ref{sec:exterior-star}), large dilations (Section~\ref{sec:dilation}), and epigraphs (Section~\ref{sec:epigraphs}).
We then exhibit nonuniqueness on domains with pockets (Section~\ref{sec:pocket}) and on cylinders (Section~\ref{sec:cylinder}).
In Section~\ref{sec:stable-compact}, we present a general procedure to identify stable-compact structure.
We trace this procedure through the arguments of Sections~\ref{sec:exterior-star}--\ref{sec:pocket} and use it to prove a new result (Theorem~\ref{thm:irregular}) on strong-KPP reactions in nonsmooth domains.
Finally, we pose a variety of open questions in Section~\ref{sec:open}.
Appendix~\ref{app:marginal} contains supporting ODE arguments.

\section*{Acknowledgments}

We warmly thank Bj\"orn Sandstede for guiding us through the literature on spatial dynamics, which plays a central role in Section~\ref{sec:cylinder}.
\corr{We are also very grateful to Isabeau Birindelli for identifying a subtle gap in an earlier version of the proof of Theorem~\ref{thm:ES-full} below.}

CG was partially supported by the NSF Mathematical Sciences Postdoctoral Research Fellowship program under grant DMS-2103383.
\corr{HB was partially supported by the French ANR grant ReaCh 23-CE40-0023-01.}

\section{Exterior-star domains}
\label{sec:exterior-star}

We open our investigation of \eqref{eq:main} on the exterior-star domains introduced in Definition~\ref{def:ES}.
We are primarily motivated by ``exterior domains'' for which $\Omega^\cc$ is compact, as in Theorem~\ref{thm:ES}.
However, our proof applies to complements of some unbounded star-shaped domains, so we state a more general form here.
\begin{theorem}
  \label{thm:ES-full}
  Suppose $\Omega$ is strongly exterior-star and $\Omega^\cc$ coincides with a convex set outside a bounded region.
  Then \eqref{eq:main} has a unique positive bounded solution $u$.
  Moreover, if $x_*$ lies in the interior of the star centers of $\Omega^\cc$, then $u$ is strictly increasing in the radial coordinate centered at $x_*$.
\end{theorem}
\corr{%
  We permit the convex set to be empty, so that Theorem~\ref{thm:ES-full} applies to bounded obstacles $\Omega^\cc$.
  Then Theorem~\ref{thm:ES-full}
} implies Theorem~\ref{thm:ES} as well as the following.
\begin{corollary}
  If $\Omega^\cc$ is convex, then \eqref{eq:main} has a unique positive bounded solution.
\end{corollary}
Thanks to this corollary, uniqueness holds on the complements of balls, cylinders, and convex paraboloids.

\subsection{Star properties}
Before proving Theorem~\ref{thm:ES-full}, we explore Definition~\ref{def:ES} in greater detail.
To facilitate the discussion, we define some complementary terms.
\begin{definition}
  A closed set $K$ is \emph{star-shaped} about a point $x_* \in \R^d$ if for all $y \in K$, the line segment from $y$ to $x_*$ lies within $K$.
  If $K$ has $\m{C}^1$ boundary, we say it is \emph{strictly} star-shaped if $K$ is star-shaped about some $x_* \in \op{int} K$ and no ray from $x_*$ is tangent to $\partial K$ at their intersection.
  Finally, $K$ is \emph{strongly} star-shaped if it is star-shaped about a nonempty open set.
\end{definition}
We are naturally interested in the relationships between these three definitions.
\begin{figure}[t]
  \includegraphics[width = 0.5\textwidth]{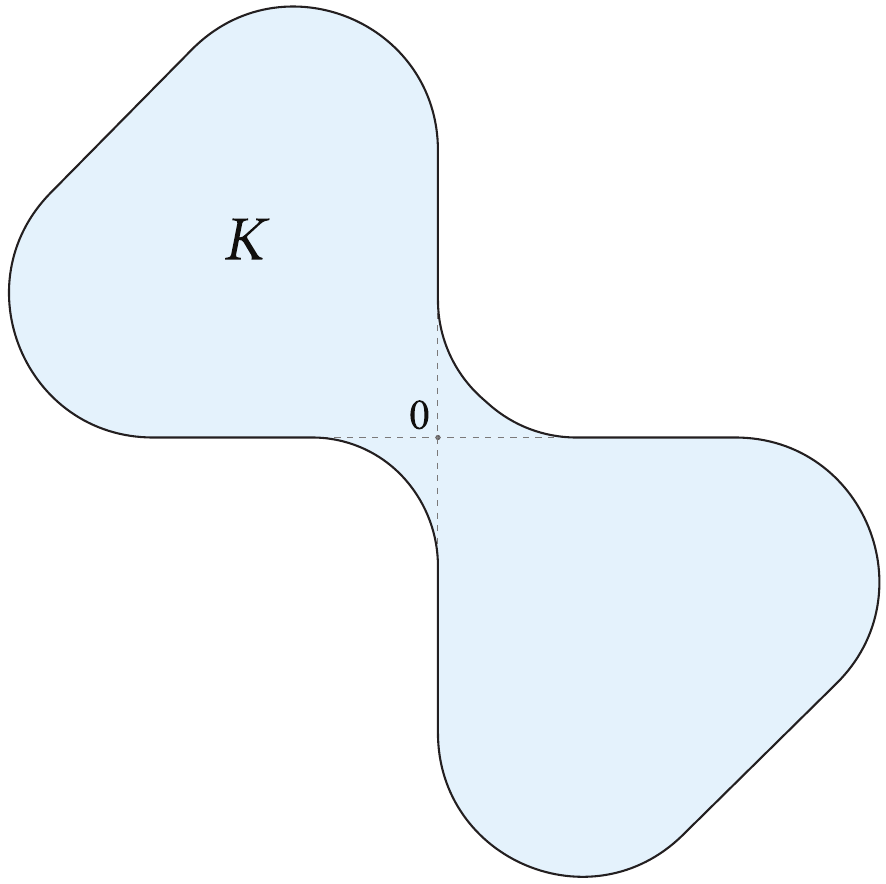}
  \centering
  \caption[Star-shaped hourglass that is neither strictly nor strongly star-shaped]{A star-shaped hourglass that is neither strictly nor strongly star-shaped.}
  \label{fig:hourglass}
\end{figure}
First, one can readily construct a smooth compact $K$ that is star-shaped but neither strictly nor strongly star-shaped.
For example, the boundary of the hourglass in Figure~\ref{fig:hourglass} contains open line segments on each signed coordinate axis.
It follows that $K$ is star-shaped precisely about the origin.
This immediately implies that $K$ is not strongly star-shaped.
Moreover, segments of the boundary are tangent to rays through the star-center $0$, so $K$ is not strictly star-shaped either.
In fact, if $K$ is compact, the strict and strong notions of star-shapedness are equivalent.
\begin{lemma}
  \label{lem:strict-strong}
  Suppose $\partial K$ is $\m{C}^1$.
  If $K$ is strongly star-shaped, then it is strictly star-shaped.
  If $K$ is compact, then the converse holds.
\end{lemma}
\begin{proof}
  First suppose $K$ is strongly star-shaped and (without loss of generality) the set of star-centers of $K$ contains an open ball $B$ about $0$.
  If $y \in \partial K$, then $K$ contains the convex hull of $B \cup \{y\}$, which includes a truncated open cone with axis through $0$ and $y$.
  Hence this axis is not tangent to $\partial K$ at $y$, and $K$ is strictly star-shaped.

  Now suppose $K$ is compact and strictly star-shaped about a point $x_* \in \op{int}K$.
  Because $\partial K$ is uniformly $\m{C}^1$, one can check that there is a continuous family $(\Gamma_y)_{y \in \partial K}$ of truncated open cones such that $\Gamma_y$ has vertex $y$ and
  \begin{equation*}
    x_* \in U \coloneqq \bigcap_{y \in \partial K} \Gamma_y.
  \end{equation*}
  Moreover, the continuity of the family $\Gamma$ and the compactness of $\partial K$ imply that $U$ is open.
  Each point in $U$ is a star-center of $K$, so $K$ is strongly star-shaped.
\end{proof}
\begin{figure}[t]
  \centering
    \includegraphics[width = 0.5\textwidth]{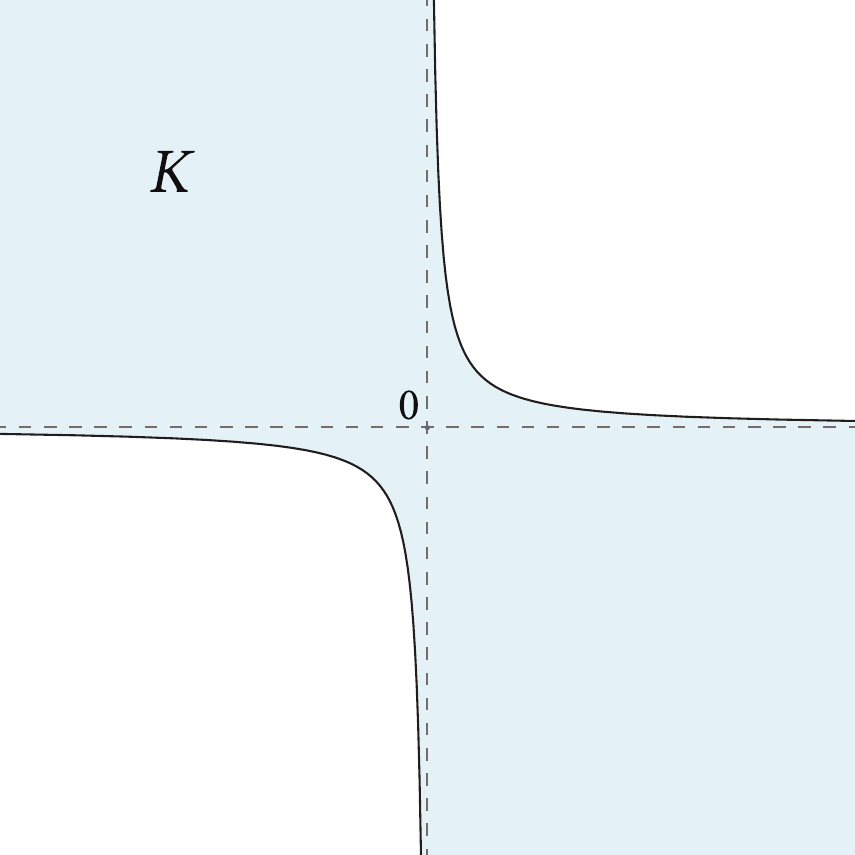}
  \caption[Hyperbolic set that is strictly but not strongly star-shaped]{A hyperbolic set that is strictly but not strongly star-shaped.}
  \label{fig:hyperbola}
\end{figure}
If $K$ is permitted to be unbounded, the strict and strong notions differ.
\begin{example}
  The hyperbolic set $K = \{xy \leq 1\} \subset \R^2$ in Figure~\ref{fig:hyperbola} is strictly but not strongly star-shaped.
  Indeed, like the hourglass in Figure~\ref{fig:hourglass}, $K$ is star-shaped precisely about the origin.
\end{example}
We will make essential use of an important geometric property of strongly exterior-star domains.
\begin{lemma}
  \label{lem:SES}
  Suppose $\Omega$ is strongly exterior-star about $0$.
  Then
  \begin{equation}
    \label{eq:SES}
    \op{dist}(\kappa \Omega, \Omega^\cc) > 0
  \end{equation}
  for all $\kappa > 1$ and
  \begin{equation}
    \label{eq:large-sep}
    \lim_{\kappa \to \infty} \op{dist}(\kappa \Omega, \Omega^\cc) \to \infty.
  \end{equation}
\end{lemma}
\begin{proof}
  Because $\Omega^\cc$ is star-shaped about $0,$ \eqref{eq:SES} is equivalent to the statement that
  \begin{equation*}
    \op{dist}(\kappa \partial \Omega, \partial \Omega) > 0.
  \end{equation*}
  Towards a contradiction, suppose there exists $\kappa > 1$ and a sequence of pairs $(y_n, z_n)_{n \in \N} \subset \partial \Omega \times \partial \Omega$ such that
  \begin{equation}
    \label{eq:approach}
    \lim_{n \to \infty} \abs{\kappa z_n - y_n} = 0.
  \end{equation}
  By Lemma~\ref{lem:strict-strong}, $\Omega^\cc$ is strictly star-shaped about $0$.
  It follows that $\kappa \partial \Omega \cap \partial \Omega = \emptyset$.
  By compactness, we must have $\abs{y_n}, \abs{z_n} \to \infty$ as $n \to \infty$.

  Because $\Omega$ is strongly exterior-star about $0$, $\Omega^\cc$ is star-shaped about a ball $B$ of radius $r > 0$ centered at $0$.
  Let $C_n$ denote the cone between $y_n$ and $B$, with axis $\ell_n \coloneqq \R y_n$.
  In the vicinity of $\kappa^{-1}y_n \in \ell_n$, the cone's radius is close to $(1 - \kappa^{-1})r > 0$.
  It follows that $C_n$ contains the open ball $B_n$ with center $\kappa^{-1}y_n$ and radius $(1 - \kappa^{-1})r/2 > 0$ once $n$ is sufficiently large.
  We emphasize that the radius of $B_n$ is independent of $n$.

  Now, \eqref{eq:approach} implies that $\abs{z_n - \kappa^{-1}y_n} \to 0$ as $n \to \infty$.
  Therefore $z_n \in B_n$ for $n \gg 1$.
  But then $z_n \in C_n \subset \op{int} \Omega^\cc$, which contradicts $z_n \in \partial \Omega$.
  This implies \eqref{eq:SES}.

  To see \eqref{eq:large-sep}, let $\delta \coloneqq \op{dist}(2 \Omega, \Omega^\cc) > 0$.
  Given $n > \delta$, we dilate by $\delta^{-1}n > 1$.
  Because $\Omega^\cc$ is star-shaped about $0$, $\Omega^\cc \subset \delta^{-1}n \Omega^\cc$.
  It follows that
  \begin{equation*}
    \op{dist}(2\delta^{-1}n \Omega, \Omega^\cc) \geq \op{dist}(2\delta^{-1}n \kappa \Omega, \delta^{-1}n \Omega^\cc) = \delta^{-1}n \op{dist}(2 \Omega, \Omega^\cc) = n.
  \end{equation*}
  Since $n$ was arbitrary, \eqref{eq:large-sep} follows.
\end{proof}
\corr{%
  In Theorem~\ref{thm:ES-full}, we assume that $\Omega^\cc$ coincides with a convex domain outside a bounded region.
  We now show that convex domains locally straighten at infinity.
  \begin{lemma}
    \label{lem:convex-infinity}
    Let $K \subset \R^d$ be a closed convex uniformly $\m{C}^{2,\gamma}$ domain, and suppose there is a sequence of points $y_n \in \partial K$ tending to infinity in the $\tbf e_1$ direction.
    Then any subsequential limit of $K - y_n$ has the form $\R \times D$ for some closed convex uniformly $\m{C}^{2,\gamma}$ domain $D \subset \R^{d-1}$ that may depend on the subsequence.
  \end{lemma}
  \begin{proof}
    Given $y \in \partial K$, let $\tbf n_y \in S^{d-1}$ denote the outward normal vector to $\partial K$ at $y$.
    For any $y \in \partial K$, convexity implies that $\tbf n_y \cdot (y_n - y) \leq 0$ for all $n \in \N$.
    Normalizing, it follows that
    \begin{equation*}
      0 \geq \tbf n_y \cdot \frac{y_n - y}{\abs{y_n - y}} \to \tbf n_y \cdot \tbf e_1 \quad \text{as } n \to \infty.
    \end{equation*}
    So $\tbf e_1 \cdot \tbf n_y \leq 0$ for all $y \in \partial K$.
    
    On the other hand, since $y_1 \in K$, convexity also yields
    \begin{equation*}
      \tbf n_{y_n} \cdot \frac{y_1 - y_n}{\abs{y_1 - y_n}} \leq 0 \ForAll n \in \N.
    \end{equation*}
    Then the boundedness of $\tbf n_{y_n}$ implies that
    \begin{equation*}
      \limsup_{n \to \infty} \tbf n_{y_n} \cdot (-\tbf e_1) \leq 0.
    \end{equation*}
    Combining these observations, we see that $\tbf e_1 \cdot \tbf n_{y_n} \to 0$ as $n \to \infty$.
    The same holds uniformly on any bounded neighborhood of $y_n$, so $\tbf{e}_1 \cdot \tbf{n} \to 0$ locally uniformly about the origin on $\partial(K - y_n)$.
    Uniform regularity allows us to extract a subsequential limit $K_\infty$ of $K - y_n$ in the sense of Definition~A.3 of \cite{BG24} (see \cite[Proposition~A.4]{BG24}).
    This convergence holds in $\m{C}^{2,\gamma-}$, so in particular the normal field converges and $\tbf{e}_1 \cdot \tbf{n} = 0$ on $\partial K_\infty$.
    That is, $\tbf{e}_1$ is everywhere tangent to $\partial K_\infty$.
    
    As the $\m{C}^{2,\gamma}$ boundary of an open set, $\partial K_\infty$ is a complete submanifold of $\R^d$.
    We can thus integrate indefinitely along the tangent vector field $\tbf{e}_1$ to conclude that $\partial K_\infty$ is a union of lines parallel to $\tbf{e}_1$.
    Let $\Sigma \coloneqq \partial K_\infty \cap \{x_1 = 0\}$ denote its cross-section.
    Then $\partial K_\infty = \R \times \Sigma$.
    The limit $K_\infty$ inherits the convexity of $K$, so $K_\infty = \R \times D$, where $D$ denotes the convex hull of $\Sigma$.
    Then $D \subset \R^{d-1}$ is a closed convex uniformly $\m{C}^{2,\gamma}$ domain, and $K_\infty = \R \times D$.
  \end{proof}
  In view of Lemma~\ref{lem:convex-infinity}, we now consider uniqueness on cylinders.
  The following result on the question is of independent interest.
  \begin{proposition}
    \label{prop:cylinder-unique}
    Suppose \eqref{eq:main} admits a unique positive bounded solution $u$ in a domain $\omega \subset \R^d$, and $\lambda(-\Delta, \omega) < f'(0)$.
    Then for any $m \in \N$, $u$ is the unique positive bounded solution in the cylinder $\R^m \times \omega$ as well.
  \end{proposition}
  \noindent
  We will apply this result to $\omega = D^\cc$, with $D$ from Lemma~\ref{lem:convex-infinity}.
  \begin{proof}
    Given $m \in \N$, let $\Gamma \coloneqq \R^m \times \omega$.
    We write coordinates as $x = (x', y)$ on $\Gamma$.
    
    Because $\lambda(-\Delta, \omega) < f'(0)$, Proposition~1.9 from \cite{BG24} states that \eqref{eq:main} admits a minimal positive solution $\ubar{v}$ in $\Gamma$ lying beneath all others.
    On the other hand, it is a quite general fact that \eqref{eq:main} admits a \emph{maximal} positive bounded solution $\bar{v}$.
    Indeed, all bounded solutions are bounded above by the supersolution $1$, so the long-time limit of the parabolic flow started from $1$ is the maximal solution.
    
    Minimality implies that $\ubar{v}$ lies beneath all its translates in $x'$, and is thus independent of $x'$.
    By the same reasoning, $\bar{v}$ is independent of $x'$.
    We can thus interpret both as positive bounded solutions on $\omega$.
    By assumption, $\ubar{v} = \bar{v}$.
    Since the minimal and maximal solutions in $\Gamma$ coincide, we have uniqueness in $\Gamma$ as well.
  \end{proof}
}

\subsection{Uniqueness on exterior-star domains}
We now return to the study of \eqref{eq:main} on exterior-star domains.
We begin by showing solutions are large far from $\partial \Omega$.
We will use this result throughout the paper.
\begin{lemma}
  \label{lem:sub}
  Given a positive reaction $f$ and $s \in (0, 1)$, there exists $R(f, s) \in \R_+$ such that if $u$ is a positive solution of \eqref{eq:main} on a domain $\Omega$, then $u(x) \geq s$ if $\op{dist}(x, \Omega^\cc) \geq R$.
\end{lemma}
\begin{proof}
  Fix $s \in (0, 1)$ and let
  \begin{equation*}
    \rho \coloneqq \inf_{r \in (0, s]} \frac{f(r)}{r}.
  \end{equation*}
  By \ref{hyp:positive}, $\rho > 0$.
  Therefore, there exists $R > 0$ such that $\rho$ is the principal Dirichlet eigenvalue of $-\Delta$ in the ball $B_R$ of radius $R$ centered at the origin.
  Let $v$ denote the corresponding positive principal eigenfunction, normalized by $\norm{v}_\infty = s$.
  We note that $v$ is radially decreasing, so $v(0) = s$.
  We extend $v$ by $0$ outside $B_R$; then $\kappa v$ is a subsolution of \eqref{eq:main} for all $\kappa \in [0, 1]$.

  Now let $u$ be a positive solution of \eqref{eq:main} on a domain $\Omega$.
  Fix $x \in \Omega$ such that $\op{dist}(x, \partial\Omega) \geq R$.
  Let $v_x \coloneqq v(\anon - x)$ denote the translate of $v$ to $x$.
  Because $u > 0$ on the compact set $B_R(x)$, there exists $\ubar{\kappa} > 0$ such that $u \geq \ubar{\kappa} v_x$.
  If we continuously raise $\kappa$ from $\ubar{\kappa}$ to $1$, the strong maximum principle prevents $\kappa v_x$ from touching $u$.
  It follows that $u > v_x$, and in particular $u(x) \geq v_x(x) = v(0) = s$.
\end{proof}
Motivated by Lemma~\ref{lem:sub}, we introduce the notation
\begin{equation}
  \label{eq:deep}
  \deep{R} \coloneqq \{x \in \Omega \mid \op{dist}(x, \partial \Omega) > R\} \For R > 0.
\end{equation}
Once $R$ is sufficiently large, Lemma~\ref{lem:sub} implies that solutions of \eqref{eq:main} are close to $1$ on $\deep{R}$.
Recalling that $f'(1) < 0$, these solutions are \emph{stable} on $\Omega[R]$ and obey a maximum principle.
The following relies on the generalized principal eigenvalue introduced in~\cite{BR}:
\begin{equation}
  \label{eq:eigenvalue}
  \lambda(-\m{L}, \Omega) \coloneqq \sup\big\{\lambda \mid \exists \psi \in W_{\text{loc}}^{2,d}(\Omega) \text{ s.t. } \psi > 0, \, (\m{L} + \lambda) \psi \leq 0\big\}.
\end{equation}
\begin{lemma}
  \label{lem:MP-deep}
  There exists $R(f) > 0$ such that the following holds.
  Let $u$ be a positive bounded solution of \eqref{eq:main} on $\Omega$ and let $v \colon \deep{R} \to [0, 1]$ be a subsolution of \eqref{eq:main}.
  If $u \geq v$ on $\partial\deep{R}$, then $u \geq v$ in $\deep{R}$.
\end{lemma}
\begin{proof}
  Because $f'(1) < 0$ and $f \in \m{C}^1$, there exists $\theta(f) \in (0, 1)$ such that
  \begin{equation*}
    f'|_{[\theta, 1]} \leq - \abs{f'(1)}/2 < 0.
  \end{equation*}
  Let $R(f, \theta) \in \R_+$ be the radius provided by Lemma~\ref{lem:sub}, which depends only on $f$.
  By Lemma~\ref{lem:sub}, $u \geq \theta$ in $\deep{R}$.
  
  We now define $w \coloneqq v - u$ and $P \coloneqq \{w > 0\} \subset \deep{R}$.
  Within $P$ we have $\theta \leq u < v \leq 1$.
  By the mean value theorem and the definition of $\theta$, we find
  \begin{equation*}
    0 = -\Delta w - \left(\frac{f(v) - f(u)}{v - u}\right)w = -\Delta w + q w
  \end{equation*}
  for some $q \geq \abs{f'(1)}/2 > 0$ in $P$.
  Set $q \equiv \abs{f'(1)}/2$ in $P^\cc$ and consider the operator $\m{L} \coloneqq \Delta - q$ on $\R^d$.
  If we extend the positive part $w_+$ by $0$ to $\deep{R}^\cc$, it remains continuous because $w_+ = 0$ on $\partial\deep{R}$ by hypothesis.
  Moreover, as the maximum of two subsolutions, $-\m{L} w_+ \leq 0$ in the sense of distributions.
  So $w_+$ is a generalized subsolution of $\m{L}$ on $\R^d$.
  
  Now, Proposition~2.1(ii) of~\cite{BG24} yields $\lambda(-\m{L}, \R^d) \geq \inf q > 0$.
  Because this eigenvalue is positive, Proposition~3.1 of~\cite{BG24} states that $\m{L}$ satisfies the maximum principle on $\R^d$ (see also Theorem~1 in~\cite{Nordmann}).
  Therefore $w_+ \leq 0$, i.e., $w = 0$.
  (Technically, we have not satisfied the hypotheses of the proposition because the potential $q$ is merely in $L^\infty$ and $w_+$ is a subsolution in the sense of distributions.
  Neither poses a challenge---the proof goes through as written.)
  Thus $v \leq u$ in $\deep{R}$, as desired.
\end{proof}
We emphasize that Lemmas~\ref{lem:sub} and \ref{lem:MP-deep} are independent of the domain, and in particular do not require the exterior-star condition.

We can now prove uniqueness on suitable exterior-star domains.
\begin{proof}[Proof of Theorem~\textup{\ref{thm:ES-full}}]
  \corr{%
    We induct on dimension.
    When $d = 1$, the only connected exterior-star domain is $\R_+$, up to translation.
    Lemma~6.1 of \cite{BG22a} establishes uniqueness in the half-line, proving the base case.
    We now inductively suppose that the statement of Theorem~\textup{\ref{thm:ES-full}} holds in dimension $d - 1$ for some $d \geq 2$.
    We wish to prove it in dimension $d$.
  }
  
  Assume without loss of generality that $\Omega$ is strongly exterior-star about $0$.
  Define
  \begin{equation*}
    \Omega_+ \coloneqq \deep{R} \And \Omega_- \coloneqq \Omega \setminus \bar{\Omega}_+,
  \end{equation*}
  where $R(f) > 0$ is provided by Lemma~\ref{lem:MP-deep}.
  By \eqref{eq:large-sep} in Lemma~\ref{lem:SES}, there exists $\bar \kappa > 0$ such that $\bar\kappa \Omega \subset \Omega_+$.

  Now fix two positive bounded solutions $u,v$ of \eqref{eq:main}.
  Given $\kappa > 1$, define the dilation
  \begin{equation*}
    v_\kappa \coloneqq v\left(\frac{\anon}{\kappa}\right),
  \end{equation*}
  which satisfies
  \begin{equation}
    \label{eq:dilate}
    \begin{cases}
      -\Delta v_\kappa = \kappa^{-2}f(v_\kappa) & \text{in } \kappa \Omega,\\
      v_\kappa = 0 & \text{on } \kappa \partial \Omega.
    \end{cases}
  \end{equation}
  Because $\kappa > 1$ and $f > 0$, we see that $v_\kappa$ is a subsolution of our original equation on the dilated domain $\kappa \Omega$.
  Extending $v_\kappa$ by $0$, it is a generalized subsolution on $\Omega$.

  Now define
  \begin{equation*}
    \kappa_* \coloneqq \inf\{\kappa > 1 \mid v_{\kappa'} \leq u \text{ in } \Omega_-\corr{\text{ for all }\kappa' \geq \kappa}\}.
  \end{equation*}
  Because $\kappa' \Omega \subset \Omega_+$ and thus $v_{\kappa'} \equiv 0$ on $\Omega_-$ \corr{for all $\kappa' \geq \bar\kappa$}, we have $\kappa_* \leq \bar\kappa$.
  We wish to show that $\kappa_* = 1$, so toward a contradiction suppose $\kappa_* > 1$.
  By pointwise continuity in $\kappa$, $v_{\kappa_*} \leq u$ in $\Omega_-$.
  Applying Lemma~\ref{lem:MP-deep}, we see that in fact $v_{\kappa_*} \leq u$ in $\Omega$.
  Moreover, $v_{\kappa_*} \neq u$ because $\kappa_* > 1$, so the strong maximum principle implies that
  \begin{equation}
    \label{eq:strict}
    v_{\kappa_*} < u \quad \text{in }\Omega.
  \end{equation}

  \corr{%
    By the definition of $\kappa_*$, there exists a sequence $\kappa_n \nearrow \kappa_*$ and $x_n \in \Omega_-$ such that
    \begin{equation}
      \label{eq:wrong-sign}
      v_{\kappa_n}(x_n) > u(x_n).
    \end{equation}
    We claim the sequence $(x_n)_n$ cannot have an accumulation point $x_\infty \in \bar\Omega_-$.
    Otherwise \eqref{eq:SES} would ensure that $x_\infty \in \Omega_-$, but then pointwise continuity would imply that $v_{\kappa_*}(x_\infty) \geq u(x_\infty)$, contradicting \eqref{eq:strict}.
    
    It follows that $(x_n)_n$ tends to infinity.
    By the compactness of $S^{d-1}$, we can extract a subsequence (also called $x_n$) that tends to infinity in a particular direction in the sense that $\frac{x_n}{\abs{x_n}}$ converges.
    After a rotation, we may assume that this limiting direction is $\tbf{e}_1$.
    Recall that $\Omega^\cc$ coincides with a convex domain $K$ outside some bounded region $A$.
    Since $(x_n)_n$ tends to infinity, we are free to assume that $\op{dist}(x_n,A) \geq 2R$ for each $n$.
    Let $y_n$ denote the projection of $x_n$ onto the convex set $K$.
    Because $x_n \in \Omega_-$, $\abs{y_n - x_n} \leq R$.
    It follows that $y_n$ tends to infinity in direction $\tbf{e}_1$ as well.

    The regularity of $\partial K$ (which coincides with $\partial \Omega$ outside $A$) allows us to extract a subsequential limit $K_\infty$ of $K - y_n$, and hence a complementary limit $\Omega_\infty = K_\infty^\cc$ of $\Omega - y_n$.
    By Lemma~\ref{lem:convex-infinity}, $K_\infty = \R \times D$ for some convex uniformly $\m{C}^{2,\gamma}$ set $D \subset \R^{d-1}$.
    We likewise wish to extract subsequential limits $z,u_\infty,$ and $\ti v$ of $x_n - y_n$, $u(\anon + y_n)$, and $v_{\kappa_n}(\anon + y_n)$.
    Each has its own subtlety, so we discuss these in turn below.

    Because $\abs{x_n - y_n} \leq R$, we can readily extract a subsequential limit $z \in \bar{\Omega}_\infty$.
    We claim that in fact $z \in \Omega_\infty$.
    Indeed, by \eqref{eq:SES}, there exists $\delta > 0$ such that for all $n \in \N$, $v_{\kappa_n}(x) = 0$ where $\op{dist}(x, \partial\Omega) \leq \delta$.
    Thus \eqref{eq:wrong-sign} implies that $\op{dist}(x_n,\partial\Omega) > \delta$.
    Passing to the limit, we see that $\op{dist}(z,\partial\Omega_\infty) \geq \delta > 0$, and in particular $z \in \Omega_\infty$.
    As a union of half-spaces, $\Omega_\infty$ is connected unless it is exactly the union of two opposite half-spaces separated by a flat slab.
    In this case we replace $\Omega_\infty$ by its connected component (a half-space) containing $z$.

    We next consider the sequence $u(\anon + y_n)$.
    Uniform Schauder estimates allow us to extract a subsequential limit $u_\infty$ solving \eqref{eq:main}  in $\Omega_\infty$.
    We claim that $u_\infty > 0$.
    By the connectedness of $\Omega_\infty$ and the strong maximum principle, it suffices to show that $u_\infty$ is positive at some point in $\Omega_\infty$.
    To see this, fix $n \in \N$ and recall that $y_n$ is the projection of $x_n$ on $K$.
    Parametrizing the ray from $y_n$ to $x_n$ by $p_n(t) \coloneqq \frac{x_n - y_n}{\abs{x_n - y_n}}t + y_n$, convexity implies that $\op{dist}(p_n(t),K) = t$.
    Recalling $\delta$ from above, let $q_n \coloneqq p_n(R + \delta/2)$.
    Then $\abs{q_n- x_n} < R - \delta/2$ and $\op{dist}(q_n,K) = R + \delta/2$.
    Now $\op{dist}(x_n,A) \geq 2R$, so $\op{dist}(q_n,A) > R + \delta/2$.
    In particular, $q_n \in \Omega[R]$.
    By the construction of $R$ in Lemma~\ref{lem:MP-deep}, $\inf_{\deep{R}} u > 0$.
    Thus if $q_\infty \in \Omega_\infty$ denotes a subsequential limit of the bounded sequence $q_n - y_n$, locally uniform convergence implies that $u_\infty(q_\infty) \geq \inf_{\deep{R}} u > 0$.
    So indeed $u_\infty > 0$ in $\Omega_\infty$.

    Finally, consider the sequence $v_{\kappa_n}(\anon + y_n)$.
    The dilated domains $\kappa_n \Omega$ are uniformly $\m{C}^{2,\gamma}$, so we can extract a uniformly $\m{C}^{2,\gamma}$ subsequential limit $\Omega_\infty'$ of $(\kappa_n \Omega - y_n)_n$.
    Note that the sets $\kappa_n \Omega - y_n$ cannot withdraw to infinity locally, leaving $\Omega_\infty'$ empty, because \eqref{eq:wrong-sign} ensures that $x_n \in \kappa_n \Omega$ and $\abs{x_n - y_n} < R$.
    So $\Omega_\infty' \neq \emptyset$, but on the other hand \eqref{eq:SES} implies that $\op{dist}(\Omega_\infty',K_\infty) \geq \delta$.
    That is, $\Omega_\infty'$ is a proper subset of $\Omega_\infty$.
    Next, Schauder estimates ensure that $v_{\kappa_n}$ is uniformly $\m{C}^{2,\gamma}$ in $\kappa_n \Omega$, so we can extract a $\m{C}^{2,\gamma}$ subsequential limit $\ti v$ of $v_{\kappa_n}(\anon + y_n)$ in $\Omega_\infty'$.
    Using local $\m{C}^2$ convergence in \eqref{eq:dilate}, we see that $\ti v$ satisfies $-\Delta \ti v = \kappa_*^{-2}f(\ti v)$ in $\Omega_\infty'$.
    In particular, it is a subsolution of \eqref{eq:main} there.
    Extending $\ti v$ by zero outside $\Omega_\infty'$, $\ti v$ is the locally uniform limit of $v_{\kappa_n}(\anon + y_n)$ and a generalized subsolution.
    Note that there does not seem to be a straightforward link between $\ti v$ and $v_{\kappa_*}$, because dilation acts non-uniformly at large distances.
    Here, we only use the fact that the limit $\ti v$ is a subsolution.

    We have now shown that $u_\infty$ is a positive bounded solution of \eqref{eq:main} in $\Omega_\infty$, while $\ti v$ is a generalized subsolution that vanishes on the nonempty set $\Omega_\infty \setminus \Omega_\infty'$.
    Moreover, passing to the limit in \eqref{eq:wrong-sign},
    \begin{equation}
      \label{eq:wrong-sign-limit}
      \ti v(z) \geq u_\infty(z) > 0 \quad \text{at } z \in \Omega_\infty.
    \end{equation}
    Recall that $v \leq 1$, and hence $\ti v \leq 1$.
    Because $\ti v$ is a subsolution and $1$ a supersolution, there is a solution $\ti u$ between them, which can be obtained as the long-time limit of the parabolic flow starting from $1$.
    By \eqref{eq:wrong-sign-limit}, $\ti{u} \not\equiv 0$, so $\ti{u} > 0$ by the strong maximum principle.
    That is, both $u_\infty$ and $\ti{u}$ are positive bounded solutions of \eqref{eq:main} in $\Omega_\infty$.

    Recall that $K_\infty$ is a cylinder: $K_\infty = \R \times D$ for some uniformly $\m{C}^{2,\gamma}$ convex set $D \subset \R^{d-1}$.
    We now write $\Omega_\infty = (\R \times D)^\cc = \R \times \omega$ for $\omega \coloneqq D^\cc \subset \R^{d-1}$.
    As a $\m{C}^{2,\gamma}$ convex set, $D$ is strongly star-shaped and of course coincides with a convex set.
    Thus $\omega$ satisfies the hypotheses of Theorem~\ref{thm:ES-full} in dimension $d-1$.
    By our inductive hypothesis, \eqref{eq:main} admits a unique positive bounded solution in $\omega$.
    As a complement of a convex set, $\omega$ contains a half-space.
    By Proposition~2.1(ii) of \cite{BG24}, $\lambda(-\Delta,\omega) = 0 < f'(0)$.
    Applying Proposition~\ref{prop:cylinder-unique} to $\Omega_\infty$, we see that $u_\infty = \ti{u}$, and in particular $u_\infty \geq \ti v$.
    On the other hand, the subsolution $\ti v$ does not coincide with $u_\infty$, because it vanishes in $\Omega_\infty \setminus \Omega_\infty'$.
    Thus by the strong maximum principle, $\ti v < u_\infty$ in $\Omega_\infty$, contradicting \eqref{eq:wrong-sign-limit}.
  }  

  From this contradiction, we conclude that in fact $\kappa_* = 1$.
  That is, $v \leq u$.
  By symmetry, $u = v$ as desired.
  Moreover, this argument shows that $u_\kappa \leq u$ for all $\kappa \geq 1$, so $u$ is increasing in the radial coordinate centered at zero.
  Applying the strong maximum principle to $\partial_r u$, it is \emph{strictly} increasing.
  Shifting, this holds about any point $x_*$ in the interior of the star centers of $\Omega$.
  \corr{Having shown the theorem statement in dimension $d$, the result now follows from induction.}
\end{proof}
\begin{figure}[t]
  \centering
  \includegraphics[width=0.7\linewidth]{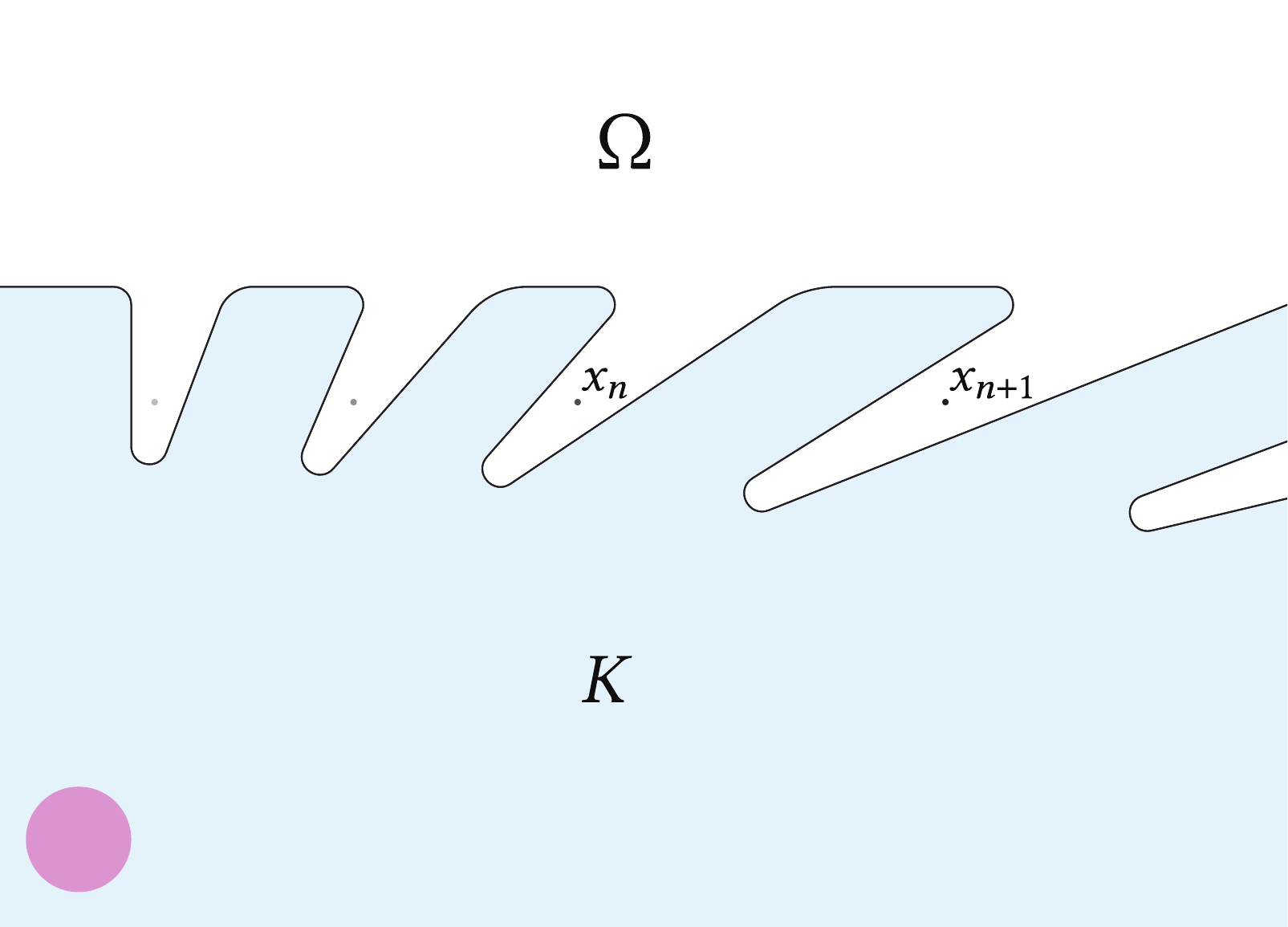}
  \caption[Strongly exterior-star ``slanted comb'' that presents certain difficulties]{A ``slanted comb'' $\Omega$ whose complement $K$ is strongly star-shaped about the magenta disk.
    If the channels enclosing the sequence $(x_n)_{n \in \N}$ are sufficiently narrow, all solutions of \eqref{eq:main} vanish locally uniformly in the vicinity of $x_n$ as $n \to \infty$.
  }
  \label{fig:comb}
\end{figure}
In the above proof, we use far-field convexity in two ways: to show solutions cannot vanish locally uniformly along any subsequence, \corr{and to reduce the effective dimension in the inductive step.}
A proof in the absence of such convexity seems delicate.
For example, one can construct strongly exterior-star domains in which solutions \emph{do} vanish locally uniformly along a subsequence; see Figure~\ref{fig:comb}.
Thus to fully remove far-field convexity from Theorem~\ref{thm:ES-full}, one would need to develop an argument that permits such decay, perhaps by including additional regions in $\Omega_+$.
\corr{Additionally, one would need a substitute for the local straightening used to induct on dimension.}
Because Theorem~\ref{thm:ES-full} already covers a host of interesting domains, we leave this question to future work.

If $\Omega^\cc$ is bounded, one may naturally wonder whether the exterior-star hypothesis is necessary.
Perhaps all ``exterior domains'' (complements of compact sets) enjoy uniqueness?
This is not the case; for an example, see Figure~\ref{fig:nonunique} below.

\section{Dilated domains}
\label{sec:dilation}

In~\cite{BG22a}, we studied the behavior of \eqref{eq:main} on one-dimensional bounded intervals as a stepping stone toward results on the half-space.
As noted in the introduction, uniqueness cannot hold in general on a bounded interval.
Indeed, to every interval $[0, L]$ one can associate a positive reaction $f_L$ such that \eqref{eq:main} admits multiple positive solutions on $[0, L]$ with reaction $f_L$.
However, in Lemma~2.5 of \cite{BG22a}, we show that for a \emph{fixed} reaction, uniqueness does hold on sufficiently large intervals.
That is, given a positive reaction $f$, there exists a length $L_f > 0$ such that \eqref{eq:main} has a unique positive solution on $[0, L]$ whenever $L > L_f$.
In this section, we investigate this phenomenon in higher directions.

Throughout the remainder of the section, let us fix a positive reaction $f$ and a domain $\Omega \subset \R^d$ satisfying our standing assumptions.
Given a dilation factor $\kappa \in \R_+$, we investigate uniqueness for \eqref{eq:main} on the dilation $\kappa \Omega$.
That is, we consider
\begin{equation}
  \label{eq:dilation}
  \begin{cases}
    -\Delta u = f(u) & \text{in } \kappa \Omega,\\
    u = 0 & \text{on } \kappa \partial \Omega.
  \end{cases}
\end{equation}
The aim of the section is Theorem~\ref{thm:dilation}, which ensures uniqueness provided $\kappa$ exceeds some threshold depending on $f$ and $\Omega$.
In contrast to other uniqueness results in this paper, we make no structural assumptions on $\Omega$.
This is due to the fact that in the stable-compact framework, the problem \eqref{eq:dilation} becomes \emph{purely stable} once $\kappa$ is sufficiently large.
There is no need for a compact part with its attendant deformation.
The pure stability of \eqref{eq:dilation} (at large $\kappa$) is thus responsible for the generality of Theorem~\ref{thm:dilation}.
We note that $\Omega$ need not be bounded, although the bounded case seems of particular interest.

Our approach to Theorem~\ref{thm:dilation} rests on the following observation.
As $\kappa \to \infty$, the domain $\kappa \Omega$ locally looks like a whole space or half-space modulo isometry.
Precisely, using the notion of locally uniform limit introduced in~\cite[Definition~A.2]{BG24}, the sequence $(\kappa \Omega)_{\kappa \geq 1}$ has exactly two limits as $\kappa \to \infty$: the whole space $\R^d$ and the half-space $\H^d \coloneqq \R^{d-1} \times \R_+$ (up to isometry).
The constant $1$ is the unique positive bounded solution of \eqref{eq:main} on $\R^d$.
Likewise, \eqref{eq:main} has a unique positive bounded solution $\varphi$ on $\H^d$, which is a function of the distance to $\partial\H^d$~\cite[Theorem~1.1(A)]{BG22a}.
We will argue below that both $1$ and $\varphi$ are strictly \emph{linearly stable}.
This is the source of the claimed stability of the problem \eqref{eq:dilation} when $\kappa \gg 1$.

To exploit this stability, we must relate solutions on $\kappa \Omega$ to those on $\R^d$ and $\H^d$.
To this end, define
\begin{equation}
  \label{eq:dilated-profile}
  \Phi_\kappa(x) \coloneqq \varphi\big(\op{dist}(x, \kappa\partial\Omega)\big).
\end{equation}
Here we have used the fact that $\varphi$ is essentially one-dimensional: writing coordinates as $x = (x', y) \in \R^{d-1} \times \R_+ = \H^d$, $\varphi$ depends $y$ alone.
The profile $\varphi$ satisfies $\varphi(+\infty) = 1$, so $\Phi_\kappa$ is close to $1$ deep in the interior of $\kappa \Omega$.
(In the extreme case that $\Omega = \R^d$, we set $\Phi_\kappa \equiv 1$.)
Moreover, when $\kappa \gg 1$, the curvature of $\kappa \partial\Omega$ is very slight, so $\Phi_\kappa$ locally resembles an isometry of $\varphi$ itself.
Thus when $\kappa \gg 1$, $\Phi_\kappa$ approximately unifies the limiting solutions $1$ and $\varphi$ in one object.
We show that solutions of \eqref{eq:dilation} coalesce around $\Phi_\kappa$.
\begin{lemma}
  \label{lem:dilation-convergence}
  Let $\m{U}_\kappa$ denote the set of positive bounded solutions of \eqref{eq:dilation}.
  Then
  \begin{equation*}
    \lim_{\kappa \to \infty} \sup_{u \in \m{U}_\kappa} \norm{u - \Phi_\kappa}_{\m{C}(\kappa\Omega)} = 0.
  \end{equation*}
\end{lemma}
\begin{proof}
  Let $R$ denote the radius corresponding to $f$ and $s = \tfrac{1}{2}$ in Lemma~\ref{lem:sub}.
  Because $\Omega$ is uniformly smooth, there exists $\kappa_0(f, \Omega) \in \R_+$ such that $\kappa\Omega$ contains a ball of radius $R$ for every $\kappa \geq \kappa_0$.
  Recalling the notation \eqref{eq:deep}, define
  \begin{equation}
    \label{eq:faraway}
    \Omega_\kappa^\ell \coloneqq (\kappa \Omega)[\ell] = \{x \in \kappa \Omega \mid \op{dist}(x, \kappa\partial\Omega) > \ell\} \For \ell > 0.
  \end{equation}
  Then $\Omega_\kappa^R$ is nonempty when $\kappa > \kappa_0$.
  By Lemma~\ref{lem:sub},
  \begin{equation}
    \label{eq:lower-inner}
    \inf_{(u, x) \in \m{U}_\kappa \times \Omega_\kappa^R} u(x) \geq \frac{1}{2}.
  \end{equation}
  
  Recall that $\kappa\Omega$ locally (and smoothly) converges to the half-space near its boundary.
  It follows that the intrinsic distance from any point in $\kappa\Omega$ to $\Omega_\kappa^R$ is uniformly bounded provided $\kappa \gg 1$.
  Precisely, there exists $\kappa_1(f, \Omega) \in \R_+$ such that for all $\kappa \geq \kappa_1$,
  \begin{equation*}
    \sup_{x \in \kappa\Omega} \op{dist}_{\kappa\Omega}(x, \Omega_\kappa^R) \leq 2R.
  \end{equation*}
  Thus \eqref{eq:lower-inner}, the Harnack inequality, and uniform smoothness imply that solutions of \eqref{eq:dilation} are uniformly positive away from $\kappa\partial \Omega$.
  That is, for all $\delta > 0$ and $\kappa \geq \kappa_1$,
  \begin{equation}
    \label{eq:uniformly-positive}
    \inf_{(u, x) \in \m{U}_\kappa \times \Omega_\kappa^\delta} u(x) > 0.
  \end{equation}
  
  Towards a contradiction, suppose there exists $\eps > 0$ and ${(u_\kappa, x_\kappa) \in \m{U}_\kappa \times \kappa \Omega}$ such that $\abs{u_\kappa(x_\kappa) - \Phi_\kappa(x_\kappa)} > \eps$ along some sequence of $\kappa$ tending to infinity.
  We restrict to this sequence.
  If $\op{dist}(x_\kappa, \kappa\partial\Omega) \to \infty$ along a subsequence, then we can center around $x_\kappa$ and extract a locally-uniform subsequential limit $u^*$ that solves \eqref{eq:main} on $\R^d$.
  Along the same sequence, $\Phi_\kappa$ tends to $\varphi(+\infty) = 1$, so $u^*(0) \leq 1 - \eps$.
  Moreover, \eqref{eq:lower-inner} implies that $u^*(0) \geq \tfrac{1}{2}$.  
  However, the only bounded solutions of \eqref{eq:main} on $\R^d$ are $0$ and $1$, a contradiction.
  
  It therefore follows that the sequence $(x_\kappa)$ remains a bounded distance from $\kappa\partial\Omega$.
  We again center around $x_\kappa$ and extract a locally-uniform subsequential limit $(u^*, \Omega^*)$ of $(u_\kappa, \kappa \Omega)$.
  Note first that \eqref{eq:uniformly-positive} implies that $u^* > 0$; that is, $u^*$ is a positive bounded solution of \eqref{eq:main} on $\Omega^*$.
  Moreover, as $\kappa \to \infty$, the boundary of $\kappa\Omega$ flattens, so $\Omega^* = g^{-1} \H^d$ for some isometry $g$ of $\R^d$.
  Thus by Theorem~1.1(A) of~\cite{BG22a}, $u^* = \varphi \circ g$.
  Moreover, $\Phi_\kappa$ also converges to the limit $\Phi^* \coloneqq \varphi \circ g$.
  But by the definition of $(u_\kappa, x_\kappa)$, $\abs{u^*(0) - \Phi^*(0)} \geq \eps > 0$, a contradiction.
  The lemma follows.
\end{proof}
Due to this lemma, \eqref{eq:dilation} becomes perturbative around $\Phi_\kappa$ when $\kappa \to \infty$.
We are therefore interested in the linearization of \eqref{eq:dilation} about $\Phi_\kappa$.
We begin by showing that the half-space solution $\varphi$ underlying $\Phi_\kappa$ is linearly stable.
We tackle this in two steps.
We first show that the principal eigenvalue on the half-space $\H^d$ coincides with that on the half-line.
We state the result in a rather general form, as it may be of independent interest.
\begin{lemma}
  \label{lem:product}
  Let $\m{L}$ be a self-adjoint elliptic operator on a uniformly smooth domain $\omega$.
  For any $d \geq 1$,
  \begin{equation*}
    \lambda(-(\m{L} + \Delta_{\R^d}), \omega \times \R^d) = \lambda(-\m{L}, \omega).
  \end{equation*}
  In particular, $\lambda(-\Delta - f'(\varphi), \H^d) = \lambda(-\partial_x^2 - f'(\varphi), \R_+)$.
\end{lemma}
\begin{proof}
  We draw on the generalized principal eigenvalues $\lambda$ and $\lambda'$ defined in~\cite{BR}.
  We defined the former in \eqref{eq:eigenvalue}; the latter is given by
  \begin{equation*}
    \lambda'(-\m{L}, \Omega) \coloneqq \inf\big\{\lambda \mid \exists \psi \in W_{\text{loc}}^{2,d}(\Omega) \cap L^\infty(\Omega) \text{ s.t. } \psi > 0, \psi|_{\partial\Omega} = 0, (\m{L} + \lambda) \psi \geq 0\big\}.
  \end{equation*}
  We can view any function $\psi \colon \omega \to \R$ as a function on $\omega \times \R^d$ that is constant in the second factor.
  It follows that any supersolution $\psi$ satisfying the definition \eqref{eq:eigenvalue} of $\lambda(-\m{L}, \omega)$ yields a supersolution for $\lambda(-(\m{L} + \Delta_{\R^d}), \omega \times \R^d)$.
  The same is true for the subsolutions defining $\lambda'$.
  Therefore
  \begin{equation}
    \label{eq:extension}
    \begin{aligned}
      \lambda(-(\m{L} + \Delta_{\R^d}), \omega \times \R^d) \geq \lambda(-\m{L}, \omega)& \quad \text{and}\\
                                               &\hspace{5pt}\lambda'(-(\m{L} + \Delta_{\R^d}), \omega \times \R^d) \leq \lambda'(-\m{L}, \omega).
    \end{aligned}
  \end{equation}
  On the other hand, $\m{L}$ and $\m{L} + \Delta_{\R^d}$ are self-adjoint, so \cite[Theorem~1.7(i)]{BR} implies that
  \begin{equation}
    \label{eq:lambdas}
    \lambda(\m{L}, \omega) = \lambda'(\m{L}, \omega) \And \lambda(-(\m{L} + \Delta_{\R^d}), \omega \times \R^d) = \lambda'(-(\m{L} + \Delta_{\R^d}), \omega \times \R^d).
  \end{equation}
  The lemma follows from \eqref{eq:extension} and \eqref{eq:lambdas}.
\end{proof}
This short argument demonstrates the utility of expressing $\lambda$ as both a supremum and an infimum.
We next show stability on the half-line.
\begin{lemma}
  \label{lem:half-line-stable}
  We have $0 < \lambda(-\partial_x^2 - f'(\varphi), \R_+) \leq \abs{f'(1)}$.
\end{lemma}
\begin{proof}
  Let $\m{H} \coloneqq -\partial_x^2 - f'(\varphi)$, which we view as a Schr\"odinger operator on the Dirichlet half-line.
  For the sake of brevity, let $\lambda_{\m{H}} \coloneqq \lambda(\m{H}, \R_+)$ denote the generalized principal eigenvalue.
  By Proposition~2.3(vi) of \cite{BR},
  \begin{equation*}
    \lambda_{\m{H}} = \inf_{\psi \in H_0^1(\R_+) \setminus \{0\}} \frac{\int_{\R_+} \psi \m{H} \psi}{\int_{\R_+} \psi^2}.
  \end{equation*}
  This is the classical Rayleigh quotient formula for the principal eigenvalue.
  By standard functional analysis for symmetric operators,
  \begin{equation*}
    \lambda_{\m{H}} = \min \Spec(\m{H}, \R_+),
  \end{equation*}
  where $\Spec(\m{H}, \R_+)$ denotes the Dirichlet spectrum of $\m{H}$ on $L^2(\R_+)$.

  We now recall that $\varphi$ satisfies the ODE
  \begin{equation}
    \label{eq:ODE}
    - \varphi'' = f(\varphi), \quad \varphi(0) = 0.
  \end{equation}
  Moreover, $\varphi'>0$ and $\varphi(+\infty) = 1$.
  Hence $-f'(\varphi) \to \abs{f'(1)} > 0$ at infinity.
  An elementary calculation then implies that $\varphi \to 1$ exponentially quickly at infinity.
  Because $f' \in \m{C}^\gamma$, the same holds for the limit $-f'(\varphi) \to \abs{f'(1)}$.
  In particular, $f'(\varphi) + \abs{f'(1)} \in L^1(\R_+)$.
  It follows from Theorem~9.38 of~\cite{Teschl} that $\m{H}$ has no singular continuous spectrum and its absolutely continuous spectrum is $[\abs{f'(1)}, \infty)$.
  (While that theorem is stated for operators on $\R$, it also holds on $\R_+$, as the author notes below the proof.)
  Hence $\lambda_{\m{H}} \leq \abs{f'(1)}$.
  
  If $\lambda_{\m{H}} = \abs{f'(1)} > 0$, we are done.
  Suppose otherwise, so that $\m{H}$ has spectrum below $\abs{f'(1)}$.
  By the above results, $\m{H}$ has only pure point spectrum below $\abs{f'(1)}$. 
  Thus there exists a principal eigenfunction $\psi > 0$ in $H^1(\R_+)$ solving
  \begin{equation*}
    \begin{cases}
      -\psi'' - f'(\varphi) \psi = \lambda_{\m{H}} \psi & \text{on } \R_+,\\
      \psi(0) = 0.
    \end{cases}
  \end{equation*}
  Also, differentiating \eqref{eq:ODE}, we find
  \begin{equation*}
    \begin{cases}
      -(\varphi')'' - f'(\varphi) \varphi' = 0 & \text{on } \R_+,\\
      \varphi''(0) = 0.
    \end{cases}
  \end{equation*}
  So $\m{H} \varphi' = 0$.

  It is straightforward to show that both $\psi$ and $\varphi'$ decay exponentially at infinity.
  Hence $\psi$ and $\varphi'$ are exponentially-localized positive eigenfunctions of $\m{H}$, albeit with different boundary data.
  The Hopf lemma (or elementary ODE uniqueness theory) yields $\psi'(0) , \varphi'(0) > 0.$
  Integrating over $\R_+$, we find:
  \begin{align*}
    \lambda_{\m{H}} \int_{\R_+} \psi \varphi' &= - \int_{\R_+} [\psi'' + f'(\varphi)\psi] \varphi'\\
                                              &= -\int_{\R_+} [(\varphi')'' + f'(\varphi)\varphi'] \psi + (\psi \varphi'' - \psi'\varphi')\big|_0^\infty\\
                                              &= \psi'(0) \varphi'(0) > 0.
  \end{align*}
  Now $\psi, \varphi' > 0$ in $\R_+$, so we must have $\lambda_{\m H} > 0$, as desired.
\end{proof}
We now turn to the stability of the function $\Phi_\kappa$ defined in \eqref{eq:dilated-profile}.
\begin{proposition}
  \label{prop:dilation-asymp-stable}
  We have
  \begin{equation*}
    \lim_{\kappa \to \infty} \lambda(-\Delta - f'(\Phi_\kappa), \kappa\Omega) = \lambda(-\partial_x^2  - f'(\varphi),\R_+) > 0.
  \end{equation*}
\end{proposition}
As a general principle, the principal eigenvalue is upper semicontinuous in the operator and domain; see Lemma~2.2 of~\cite{BG24} for an example.
One can thus conclude from ``soft'' analysis (and Lemma~\ref{lem:product}) that
\begin{equation*}
  \limsup_{\kappa \to \infty} \lambda(-\Delta - f'(\Phi_\kappa), \kappa\Omega) \leq \lambda(-\Delta - f'(\varphi), \H^d) = \lambda(-\partial_x^2 - f'(\varphi), \R_+).
\end{equation*}
It is more challenging, however, to obtain a matching lower bound as $\kappa \to \infty$.
For this purpose, we make essential use of a beautiful result of Lieb~\cite{Lieb}.
\begin{theorem}
  \label{thm:potential-Lieb}
  Let $V$ be a uniformly $\m{C}^\gamma(\Omega)$ potential.
  Then for all $R > 0,$
  \begin{equation}
    \label{eq:potential-Lieb}
    \inf_{x \in \R^d} \lambda\big(\!-\Delta + V, \Omega \cap B_R(x)\big) \leq \lambda(-\Delta + V, \Omega) + \lambda(-\Delta, B_1) R^{-2}.
  \end{equation}
\end{theorem}
In a certain sense, this bound states that the principal eigenvalue is ``local:'' $\lambda$ can be approximated to accuracy $\eps$ by examining the eigenvalue problem at spatial scale $\eps^{-1/2}$.
This is crucial for our purposes, as $\Omega$ only resembles the half-space locally.
In~\cite{Lieb}, Lieb stated Lemma~\ref{thm:potential-Lieb} for $V \equiv 0$.
However, his argument readily extends to other potentials.
For the sake of completeness, we include a proof here.
\begin{proof}
  Fix $\eps > 0$.
  By the Rayleigh quotient formula for the principal eigenvalue of self-adjoint operators, there exist $f \in \m{C}_0^\infty(\Omega)$ and $g \in \m{C}_0^\infty(B_R)$ with unit $L^2$ norms such that
  \begin{equation}
    \label{eq:almost-eigen}
    \int_\Omega \big(\abs{\nab f}^2 + V f^2\big) < \lambda(-\Delta + V, \Omega) + \frac{\eps}{2} \And \int_{B_R} \abs{\nab g}^2 < \lambda(-\Delta, B_R) + \frac{\eps}{2}.
  \end{equation}
  We extend $g$ by $0$ to $\R^d$ and define
  \begin{equation*}
    h_x(y) \coloneqq f(y)g(y - x)
  \end{equation*}
  for $x \in \R^d$ and $y \in \Omega$.
  Next, we define,
  \begin{equation*}
    T(x) \coloneqq \int_\Omega \big(\abs{\nab h_x}^2 + V h_x^2\big) \And D(x) \coloneqq \int_\Omega h_x^2.
  \end{equation*}
  Fubini and our $L^2$-normalization of $f$ and $g$ imply that $\int_{\R^d} D = 1$ and
  \begin{equation}
    \label{eq:potential}
    \int_{\R^d \times \Omega} V(y) h_x(y)^2 \d x \ds y = \int_{\R^d \times \Omega} V(y) f(y)^2 g(y - x)^2 \d x \ds y = \int_\Omega V f^2.
  \end{equation}
  For the gradient term of $T$, we compute
  \begin{equation*}
    \abs{\nab h_x(y)}^2 = \abs{\nab f(y)}^2 g(y - x)^2 + f(y)^2 \abs{\nab g(y - x)}^2 + \frac{1}{2} \nab (f^2)(y) \cdot \nab(g^2)(y - x).
  \end{equation*}
  Writing $ \nab(g^2)(y - x)$ as $-\nab_x[g(y - x)^2]$, we see that the final term vanishes when we integrate in $x$.
  Thus by Fubini,
  \begin{equation}
    \label{eq:deriv-int}
    \begin{aligned}
      \int_{\R^d \times \Omega} \abs{\nab h_x}^2 &= \int_{\R^d \times \Omega} \left[\abs{\nab f(y)}^2 g(y - x)^2 + f(y)^2 \abs{\nab g(y - x)}^2\right]\\
      &= \int_\Omega \abs{\nab f}^2 + \int_{B_R} \abs{\nab g}^2.
    \end{aligned}
  \end{equation}
  Combining \eqref{eq:potential} and \eqref{eq:almost-eigen} with \eqref{eq:deriv-int}, we see that
  \begin{equation*}
    \int_{\R^d} T(x) < \lambda(-\Delta + V, \Omega) + \lambda(-\Delta, B_R) + \eps.
  \end{equation*}
  Since $\int D = 1$, we have
  \begin{equation*}
    \int_{\R^d} \big[T(x) - [\lambda(-\Delta + V, \Omega) + \lambda(-\Delta, B_R) + \eps] D(x)\big] < 0.
  \end{equation*}
  Therefore
  \begin{equation}
    \label{eq:positive-set}
    [\lambda(-\Delta + V, \Omega) + \lambda(-\Delta, B_R) + \eps] D(x) > T(x)
  \end{equation}
  on a set of positive measure.
  In particular, there exists $x \in \R^d$ satisfying \eqref{eq:positive-set}.
  Now, if we substitute $h_x$ into the Rayleigh quotient for $\lambda\big(\!-\Delta + V, \Omega \cap B_R(x)\big)$, we obtain $T(x)/D(x)$.
  Thus by our choice of $x$, we see that
  \begin{equation*}
    \lambda\big(\!-\Delta + V, \Omega \cap B_R(x)\big) \leq \lambda(-\Delta + V, \Omega) + \lambda(-\Delta, B_R) + \eps.
  \end{equation*}
  Since $\eps > 0$ was arbitrary, \eqref{eq:potential-Lieb} follows from the identity $\lambda(-\Delta, B_R) = \lambda(-\Delta, B_1) R^{-2}$.
\end{proof}
With this tool in hand, we can tackle Proposition~\ref{prop:dilation-asymp-stable}.
\begin{proof}[Proof of Proposition~\textup{\ref{prop:dilation-asymp-stable}}]
  The result is trivial when $\Omega = \R^d$, so suppose otherwise.

  We first tackle the hard direction: we use our extension of Lieb's theorem to show the lower bound
  \begin{equation}
    \label{eq:liminf}
    \liminf_{\kappa \to \infty} \lambda(-\Delta - f'(\Phi_\kappa), \kappa \Omega) \geq \lambda(-\partial_x^2 - f'(\varphi), \R_+).
  \end{equation}
  Fix $\eps > 0$ and choose $R > 0$ such that $\lambda(-\Delta, B_1) R^{-2} < \eps/2$.
  Then Theorem~\ref{thm:potential-Lieb} yields
  \begin{equation}
    \label{eq:dilation-Lieb}
    \lambda(-\Delta - f'(\Phi_\kappa), \kappa \Omega) \geq \inf_{x \in \R^d} \lambda\big(\!-\Delta - f'(\Phi_\kappa), \kappa \Omega \cap B_R(x)\big) - \frac{\eps}{2}.
  \end{equation}
  We are free to assume that $x \in S_\kappa \coloneqq \{x \in \R^d \mid \op{dist}(x, \kappa \Omega) < R\}$, for otherwise $\kappa\Omega \cap B_R(x) = \emptyset$.

  Since $\varphi(\infty) = 1$ and $f'$ is continuous, there exists $Y > R$ such that
  \begin{equation*}
    -f' \circ \varphi(y) \geq \abs{f'(1)} - \tfrac{\eps}{2} \ForAll y \geq Y - R.
  \end{equation*}
  Recalling the set $\Omega_\kappa^\ell$ from \eqref{eq:faraway}, Lemma~\ref{lem:half-line-stable} yields
  \begin{equation}
    \label{eq:eigenvalue-deep}
    \inf_{x \in \Omega_\kappa^Y} \lambda\big(\!-\Delta - f'(\Phi_\kappa), \kappa \Omega \cap B_R(x)\big) \geq \abs{f'(1)} - \frac{\eps}{2} \geq \lambda(-\partial_x^2 - f'(\varphi), \R_+) - \frac{\eps}{2}.
  \end{equation}
  Now consider centers $x$ in the ``collar'' $\Sigma_\kappa \coloneqq S_\kappa \setminus \Omega_\kappa^Y$, which lies a bounded distance from $\kappa \partial \Omega$.

  As $\kappa \to \infty$, the principal radii of curvature of $\kappa \partial\Omega$ grow without bound.
  Thus for sufficiently large $\kappa$, every point $x \in \Sigma_\kappa$ has a unique nearest point $p(x) \in \kappa \partial \Omega$ and $\Phi_\kappa$ is $\m{C}^{2,\gamma}$ on $\Sigma_\kappa$.
  Let $n(x)$ denote the inward unit normal vector to $\kappa \partial \Omega$ at $p(x)$.
  As the boundary flattens, the distance function $z \mapsto \op{dist}(z, \kappa\partial \Omega)$ comes to resemble the linear coordinate
  \begin{equation*}
    y_x(z) \coloneqq n(x) \cdot [z - p(x)]
  \end{equation*}
  on $\kappa \Omega \cap B_R(x)$.
  That is,
  \begin{equation*}
    \lim_{\kappa \to \infty} \sup_{x \in \Sigma_\kappa} \norm{\op{dist}(\anon, \kappa \partial \Omega) - y_x}_{\m{C}^{2, \gamma}(\kappa \Omega \cap B_R(x))} = 0.
  \end{equation*}
  Hence $f'(\Phi_\kappa) \to f'(\varphi \circ y_x)$ on $\kappa\Omega \cap B_R(x)$ uniformly in $x \in \Sigma_\kappa$ as $\kappa \to \infty$.
  In the same manner, the domain $\Omega \cap B_R(x)$ converges to $\H_x \cap B_R(x)$, where $\H_x$ denotes the half-space defined by $y_x > 0$.
  Since $B_R$ has uniformly bounded radius, the principal eigenvalue is continuous in the potential and the domain within $B_R$; see, for instance,~\cite{Buttazzo}.
  It follows that
  \begin{equation}
    \label{eq:local-eigenvalue-limit}
    \begin{aligned}
      \lim_{\kappa \to \infty} \sup_{x \in \Sigma_\kappa} \big|\lambda\big(\!-\Delta - f'(\Phi_\kappa), \kappa &\Omega \cap B_R(x)\big)\\
      &- \lambda\big(\!-\Delta - f'(\varphi \circ y_x), \H_x \cap B_R(x)\big)\big| = 0.
    \end{aligned}
  \end{equation}
  Using monotonicity in the domain (see \cite[Proposition~2.1(i)]{BG22a}),
  \begin{equation*}
    \lambda\big(\!-\Delta - f'(\varphi \circ y_x), \H_x \cap B_R(x)\big) \geq \lambda\big(\!-\Delta - f'(\varphi \circ y_x), \H_x\big).
  \end{equation*}
  Now, the principal eigenvalue is invariant under isometry.
  Rotating, translating, and using Lemma~\ref{lem:product}, we find
  \begin{equation*}
    \lambda\big(\!-\Delta - f'(\varphi \circ y_x), \H_x \cap B_R(x)\big) \geq \lambda\big(\!-\Delta - f'(\varphi), \H^d\big) = \lambda(-\partial_x^2 - f'(\varphi), \R_+).
  \end{equation*}
  In light of \eqref{eq:local-eigenvalue-limit}, we have
  \begin{equation*}
    \liminf_{\kappa \to \infty} \inf_{x \in \Sigma_\kappa} \lambda\big(\!-\Delta - f'(\Phi_\kappa), \kappa \Omega \cap B_R(x)\big) \geq \lambda(-\partial_x^2 - f'(\varphi), \R_+).
  \end{equation*}
  Combining this with \eqref{eq:eigenvalue-deep}, we see that
  \begin{equation*}
    \liminf_{\kappa \to \infty} \inf_{x \in S_\kappa} \lambda\big(\!-\Delta - f'(\Phi_\kappa), \kappa \Omega \cap B_R(x)\big) \geq \lambda(-\partial_x^2 - f'(\varphi), \R_+) - \frac{\eps}{2}.
  \end{equation*}
  Then \eqref{eq:dilation-Lieb} yields
  \begin{equation*}
    \liminf_{\kappa \to \infty} \lambda(-\Delta - f'(\Phi_\kappa), \kappa \Omega) \geq \lambda(-\partial_x^2 - f'(\varphi), \R_+) - \eps.
  \end{equation*}
  Since $\eps > 0$ was arbitrary, \eqref{eq:liminf} follows.

  We now show a matching upper bound.
  For each $\kappa > 0$, let $g_\kappa$ be an isometry of $\R^d$ such that if $\Omega_\kappa \coloneqq g_\kappa^{-1}(\kappa \Omega)$, then $0 \in \partial \Omega_\kappa$ and $e_n$ is the inward unit normal vector to $\partial\Omega_\kappa$ at $0$.
  We showed above that $\Omega_\kappa \to \H^d$ locally uniformly in a sense made precise in Definition~A.2 in~\cite{BG22a}.
  Similarly, we showed that $\Phi_\kappa \circ g_\kappa \to \varphi$ locally uniformly in $\m{C}^{2,\gamma}$.
  Hence by isometry-invariance, Lemma~2.2 of~\cite{BG22a}, and Lemma~\ref{lem:product},
  \begin{equation}
    \label{eq:limsup}
    \limsup_{\kappa \to \infty} \lambda(-\Delta - f'(\Phi_\kappa), \kappa\Omega) \leq \lambda(-\Delta - f'(\varphi), \H^d) = \lambda(-\partial_x^2 - f'(\varphi), \R_+).
  \end{equation}
  Strictly speaking, \cite[Lemma~2.2]{BG22a} only treats the operator $-\Delta$.
  However, the inclusion of a locally convergent potential like $-f'(\Phi_\kappa \circ g_\kappa)$ does not change the proof; we do not repeat it here.
  The proposition now follows from \eqref{eq:liminf}, \eqref{eq:limsup}, and Lemma~\ref{lem:half-line-stable}.
\end{proof}
With this spectral estimate in place, we can prove uniqueness in \eqref{eq:dilation}.
\begin{proof}[Proof of Theorem~\textup{\ref{thm:dilation}}]
  Let $\mu \coloneqq \lambda(-\partial_x^2 - f'(\varphi), \R_+)$, which is positive by Lemma~\ref{lem:half-line-stable}.
  Because $f'$ is uniformly continuous, there exists $\delta > 0$ such that
  \begin{equation}
    \label{eq:deriv-cont}
    \abs{f'(s_1) - f'(s_2)} \leq \frac{\mu}{3} \quad \text{when } \abs{s_1 - s_2} \leq \delta.
  \end{equation}
  Recall that $\m{U}_\kappa$ denotes the set of positive bounded solutions of \eqref{eq:dilation}.
  By Lemma~\ref{lem:dilation-convergence} and Proposition~\ref{prop:dilation-asymp-stable}, there exists $\ubar{\kappa}(\Omega, f) > 0$ such that for all $\kappa \geq \ubar{\kappa}$, we have
  \begin{equation}
    \label{eq:convergence-large}
    \sup_{u \in \m{U}_\kappa} \norm{u - \Phi_\kappa}_{\m{C}(\kappa \Omega)} \leq \delta
  \end{equation}
  and
  \begin{equation}
    \label{eq:eigenvalue-large}
    \lambda(-\Delta - f'(\Phi_\kappa), \kappa\Omega) \geq \frac{2\mu}{3}.
  \end{equation}

  Now fix $\kappa \geq \ubar{\kappa}$ and two solutions $u,v \in \m{U}_\kappa$.
  By the mean value theorem and \eqref{eq:convergence-large}, there exists $r \colon \Omega \to [0, 1]$ such that $\abs{r - \Phi_\kappa} \leq \delta$ and
  \begin{equation*}
    f(u) - f(v) = f'(r) (u - v).
  \end{equation*}
  Let $w \coloneqq u - v$ and $\m{L} \coloneqq \Delta + f'(r)$.
  Then \eqref{eq:dilation} yields
  \begin{equation}
    \label{eq:diff}
    \begin{cases}
      \m{L} w = 0 & \text{in } \kappa \Omega,\\
      w = 0& \text{on } \kappa \partial\Omega.
    \end{cases}
  \end{equation}
  Because $\abs{r - \Phi_\kappa} < \delta$, \eqref{eq:deriv-cont} implies that $f'(r) = f'(\Phi_\kappa) + q$ for some remainder $q$ satisfying $\abs{q} \leq \mu/3$.
  So $-\m{L} \geq -\Delta - f'(\Phi_\kappa) - \mu /3$.
  Using \eqref{eq:eigenvalue-large}, we find
  \begin{equation*}
    \lambda(-\m{L}, \kappa\Omega) \geq \lambda(-\Delta - f'(\Phi_\kappa), \kappa\Omega) - \frac{\mu}{3} \geq \frac{\mu}{3} > 0.
  \end{equation*}
  By Theorem~1 of~\cite{Nordmann} (a form of the maximum principle), \eqref{eq:diff} implies that $w = 0$.
  That is, $u = v$ on $\kappa \Omega$.
\end{proof}
\begin{figure}[t]
  \centering
  \includegraphics[width = 0.5\textwidth]{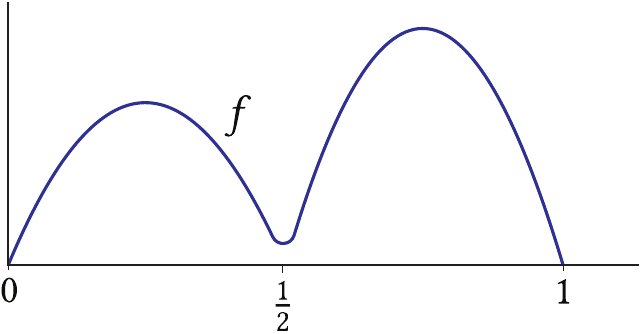}
  \caption[``Double-humped'' reaction exhibiting nonuniqueness]{A ``double-humped'' reaction exhibiting nonuniqueness.}
  \label{fig:camel}
\end{figure}
Having proved uniqueness on ``strongly dilated'' domains, we return to a point raised at the beginning of the section.
Fix a bounded domain $\Omega \subset \R^d$.
By Proposition~1.4 of~\cite{BG24} there exists a positive reaction $f$ such that \eqref{eq:main} admits multiple positive bounded solutions.
Qualitatively, the reaction we construct is ``double-humped'' as in Figure~\ref{fig:camel}.
It admits a solution $u^-$ whose range lies in the first hump and a larger solution $u^+$ whose range spans both humps.

We now consider how these solutions vary as we dilate $\Omega$ by a factor $\kappa \geq 1$.
The proof of \cite[Proposition~1.4]{BG24} shows that we can arrange $f'(0) > \lambda(-\Delta, \Omega)$, which implies $f'(0) > \lambda(-\Delta,\kappa\Omega)$ for $\kappa \geq 1$.
Then Proposition~1.8 of~\cite{BG24} implies that \eqref{eq:dilation} admits a minimal solution $u_\kappa^-$.
Moreover, the proof of the same proposition shows the existence of a maximal solution $u_\kappa^+$.
A comparison argument readily shows that the family $(u_\kappa^\pm)_{\kappa \geq 1}$ is increasing in $\kappa$.
We consider the behavior of the pair $u_\kappa^\pm$ as $\kappa$ grows from $1$.

When $\kappa = 1$, we have $u_1^- < u_1^+$.
On the other hand, Theorem~\ref{thm:dilation} provides $\ubar{\kappa}(\Omega, f) > 1$ such that $u_\kappa^- = u_\kappa^+$ once $\kappa \geq \ubar{\kappa}$.
We informally describe the manner in which the branch $u_\kappa^-$ might merge with $u_\kappa^+$.

Initially, we have constructed our pair of solutions so that $\sup u_1^- < \frac{1}{2}$.
That is, $u_1^-$ is confined to the first hump of $f$.
However, Lemma~\ref{lem:dilation-convergence} implies that $u_\kappa^\pm$ must coalesce around $\Phi_\kappa$ as $\kappa \to \infty$.
Since $\sup \Phi_\kappa \to \varphi(\infty) = 1$ in this limit, $\sup u_\kappa^-$ must eventually cross the threshold $\tfrac{1}{2}$.
That is, as $\kappa$ grows, the minimal solution $u_\kappa^-$ must eventually grow into the second hump of the reaction $f$.
Once this occurs, our proof of nonuniqueness in \cite{BG24} breaks down, and there is nothing preventing $u_\kappa^-$ from merging with $u_\kappa^+$.

Of course, the branches $u_\kappa^\pm$ may exhibit more complicated behavior between $\kappa = 1$ and $\kappa = \ubar{\kappa}$.
Nonetheless, the invasion of the second hump by $u_\kappa^-$ is one pathway by which multiple solution branches might merge as $\kappa \to \infty$.
It would be interesting to rigorously confirm this picture on a simple domain like $\Omega = (0, 1)$.

\section{Epigraphs}
\label{sec:epigraphs}

We now turn to uniqueness on epigraphs: domains bounded by the graph of a function.
Given a $\m{C}_{\text{loc}}^{2, \gamma}$ function $\phi \colon \R^{d - 1} \to \R$, we study its epigraph
\begin{equation*}
  \Omega \coloneqq \big\{(x', y) \in \R^{d - 1} \times \R\mid y > \phi(x')\big\}.
\end{equation*}
We assume that $\Omega$ is uniformly $\m{C}^{2,\gamma}$ as a subset of $\R^d$.
Notably, this is much weaker than assuming that $\phi$ is a uniformly $\m{C}^{2,\gamma}$ \emph{function}.
This discrepancy is due to the fact that the smoothness of a domain can be measured with respect to different coordinate frames.
If the gradient $\nab\phi$ is very large but slowly-varying, the epigraph $\Omega$ can still be quite smooth, because its local $\m{C}^{2,\gamma}$ norm can be measured with respect to a frame oriented normal to the graph of $\phi$.
\begin{example}
  \label{ex:parabola}
  The parabolic epigraph $\Omega$ of the quadratic $\phi(x) = x^2$ is uniformly $\m{C}^{2,\gamma}$, but $\phi$ and $\phi'$ diverge at infinity, so $\phi$ is not uniformly $\m{C}^{2,\gamma}$.
\end{example}
We are interested in the uniqueness of positive bounded solutions of \eqref{eq:main} on the epigraph $\Omega$.

\subsection{Epigraphs of uniformly Lipschitz functions}

This question has already been resolved for an important subclass of epigraphs: those for which $\op{Lip}\phi < \infty$, where $\op{Lip} \phi$ denotes the global Lipschitz constant of $\phi$.
For convenience, we refer to these as uniformly Lipschitz epigraphs.
In~\cite{BCN97b}, Caffarelli, Nirenberg, and the first author studied the qualitative properties of \eqref{eq:main} in this uniformly Lipschitz setting.
For convenience, we restate a form of their main result here.
\begin{theorem}[\cite{BCN97b}]
  \label{thm:Lipschitz}
  Let $\Omega$ be the epigraph of a function $\phi$ such that $\op{Lip} \phi < \infty$.
  Then \eqref{eq:main} admits a unique positive bounded solution $u$, and $\partial_y u > 0$.
\end{theorem}
In this section, we expand this result to a much broader class of epigraphs.
The condition $\op{Lip}\phi < \infty$ ensures that the graph of $\phi$ (the boundary $\partial\Omega$) lies between two cones oriented along the $y$-axis.
Many epigraphs of interest, such as the parabola in Example~\ref{ex:parabola}, do not satisfy this condition---they correspond to functions $\phi$ that grow superlinearly at infinity.
We are thus led to a natural question: does the conclusion of Theorem~\ref{thm:Lipschitz} hold for \emph{all} (uniformly $\m{C}^{2,\gamma}$) epigraphs?
While we do not fully resolve this question, we make significant progress.

\subsection{Asymptotically uniformly Lipschitz epigraphs}

We study epigraphs that locally resemble uniformly Lipschitz epigraphs at infinity, perhaps after rotation.
\begin{definition}
  \label{def:AUL}
  The epigraph $\Omega$ is \emph{asymptotically uniformly-Lipschitz} (AUL) if there exists $M \in \R_+$ such that every locally uniform limit of $\Omega$ at infinity is either $\R^d$ or a rotation of an epigraph with global Lipschitz constant at most $M$.
\end{definition}
That is, if we examine $\Omega$ near a sequence of points tending to $\infty$, it locally resembles either the whole space or a rotation of a uniformly Lipschitz epigraph of the form discussed above.
In particular, if the curvature of $\partial\Omega$ vanishes at infinity, the only limit domains at infinity are $\R^d$ and an isometry of the half-space $\H^d$, which is evidently uniformly Lipschitz.
It follows that asymptotically flat epigraphs in the sense of Definition~\ref{def:flat} are AUL in the sense of Definition~\ref{def:AUL}.
\begin{figure}[t]
  \centering
  \includegraphics[width = 0.9\linewidth]{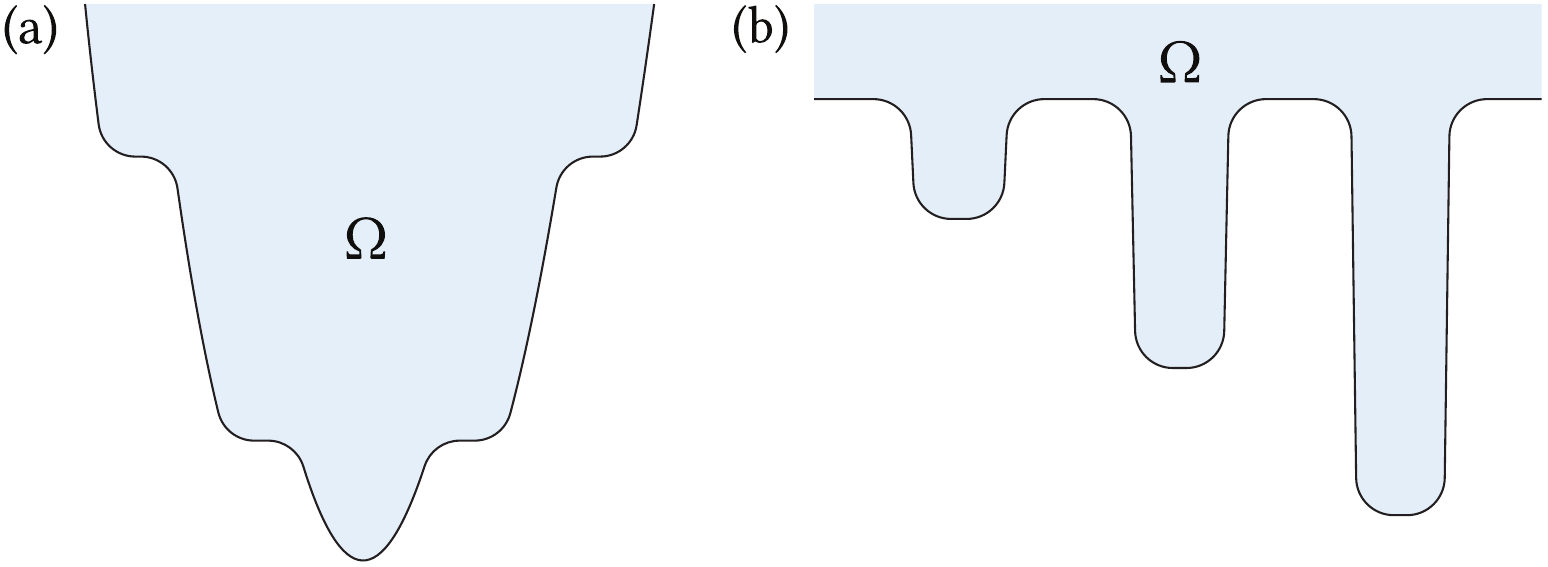}
  \caption[Epigraph ``counterexamples'']{
    (a)~An AUL epigraph that is neither uniformly Lipschitz nor asymptotically flat.
    (b)~A uniformly $\m{C}^{2,\gamma}$ epigraph with arbitrarily deep ``wells'' that is not AUL.
    The well walls are not quite vertical, so $\Omega$ is still the epigraph of a continuous function.
  }
  \label{fig:counter}
\end{figure}
On the other hand, the parabola with ``steps'' in Figure~\ref{fig:counter}(a) is an AUL epigraph that is neither uniformly Lipschitz nor asymptotically flat.

We show uniqueness on AUL epigraphs.
\begin{theorem}
  \label{thm:AUL}
  If $\Omega$ is AUL, then \eqref{eq:main} has a unique positive bounded solution $u$.
  Moreover, $\partial_y u > 0$.
\end{theorem}
This extends and significantly strengthens the main result of~\cite{BCN97b}.
Given that bounded domains do not exhibit uniqueness in this generality, Theorem~\ref{thm:AUL} is quite striking.
The epigraph structure seems conducive to uniqueness in~\eqref{eq:main}.

Definition~\ref{def:AUL} encompasses a great variety of ``natural'' epigraphs.
It is, however, not comprehensive: there are uniformly $\m{C}^{2,\gamma}$ epigraphs that are not AUL.
For example, an epigraph with a sequence of ever deeper ``wells'' of bounded width is not AUL; see Figure~\ref{fig:counter}(b) for an example.
Such domains present obstacles to our uniqueness argument; we discuss these difficulties in greater detail in Section~\ref{sec:wells}.

Before proceeding, we use our main Theorem~\ref{thm:AUL} to show the simpler results in the introduction.
\begin{proof}[Proof of Theorem~\textup{\ref{thm:flat}}]
  As noted above, asymptotically flat epigraphs are AUL.
  So Theorem~\ref{thm:flat} is a special case of Theorem~\ref{thm:AUL}.
\end{proof}
\begin{proof}[Proof of Corollary~\textup{\ref{cor:convex}}]
  If $\phi \colon \R \to \R$ in $\m{C}_{\textrm{loc}}^{2,\gamma}$ is convex, then the limits $\phi(\pm \infty)$ exist (though they may be infinite).
  It follows that the only limits of $\Omega$ at infinity are either $\R^2$ or isometries of $\H^2$.
  Hence $\Omega$ is AUL, and the corollary follows from Theorem~\ref{thm:AUL}.
\end{proof}

\subsection{Stability on uniformly Lipschitz epigraphs}

To prove Theorem~\ref{thm:AUL}, we use uniformly Lipschitz epigraphs as the ``building blocks'' of AUL domains.
The centerpiece of the argument is a new qualitative property of uniformly Lipschitz epigraphs: the unique solution $u$ in Theorem~\ref{thm:Lipschitz} is \emph{strictly stable}.
This stability neatly complements the other qualitative properties shown in~\cite{BCN97b}.

In fact, we will require this strict stability to be uniform over a broad class of epigraphs.
To state the result, we recall the notion of a $(\gamma, r, K)$-smooth domain from Definition~A.1 of~\cite{BG24}.
We do not restate the full (rather technical) definition, but informally, the boundary of a $(\gamma, r, K)$-smooth domain has $\m{C}^{2,\gamma}$ norm no larger than $K$ at spatial scale $r$.
Given $\gamma \in (0, 1)$ and $r, K, M \in \R_+$, let $\s{G}(r, \gamma, K, M)$ denote the set of all $(r, \gamma, K)$-smooth epigraphs with global Lipschitz constant at most $M$.
Given $G \in \s{G}(r, \gamma, K, M)$, let $u^G$ denote the unique positive bounded solution of \eqref{eq:main} on $G$ provided by Theorem~\ref{thm:Lipschitz}.
\begin{proposition}
  \label{prop:uniform-stability}
  For all $\gamma \in (0, 1)$ and $r, K ,M \in \R_+$,
  \begin{equation}
    \label{eq:uniform-stability}
    \abs{f'(1)} \geq \inf_{G \in \s{G}(\gamma, r, K, M)} \lambda(-\Delta - f'(u^G), G) > 0.
  \end{equation}
\end{proposition}
The half-space is the simplest epigraph, and we have already shown strict stability in that setting in Lemmas~\ref{lem:product} and \ref{lem:half-line-stable}.
So Proposition~\ref{prop:uniform-stability} can be viewed as a vast generalization of those lemmas.
Its proof, naturally, is significantly more complex.

The uniformity in Proposition~\ref{prop:uniform-stability} leads to a curious question: which epigraph is ``most unstable''?
It seems possible that the smoothness parameters in Proposition~\ref{prop:uniform-stability} are superfluous---the infimum in \eqref{eq:uniform-stability} may depend on $M$ alone.
If this is the case, the only free parameter is the Lipschitz bound $M$, and we are left with a family of clean geometric optimization problems.
For example, which $1$-Lipschitz epigraph $G$ minimizes the eigenvalue in \eqref{eq:uniform-stability}?
The quarter-plane $\{y > \abs{x}\}$ is a natural guess, but we leave this line of inquiry to future investigation.
\begin{proof}[Proof of Proposition~\textup{\ref{prop:uniform-stability}}]
  Fix an epigraph $G \in \s{G}(\gamma, r, K, M)$ corresponding to a function $\phi$ with $\op{Lip} \phi \leq M$ and let $u \coloneqq u^G$.
  Let $y' \coloneqq y - \phi(x')$, so $G = \{y' > 0\}$.
  Let $x = (x', y)$.
  Because $\op{Lip} \phi \leq M$,
  \begin{equation}
    \label{eq:trigonometry}
    \op{dist}(x, \partial G) \geq \frac{y'}{\sqrt{M^2 + 1}}.
  \end{equation}
  Now $f'$ is continuous, so there exists $\eta \in (0, 1)$ such that $f' \leq \frac{2}{3} f'(1) < 0$ on $(\eta, 1)$.
  Combining Lemma~\ref{lem:sub} with \eqref{eq:trigonometry},  we see that there exists $H > 1$ such that $u > \eta$ where $y' > H$.
  In particular,
  \begin{equation}
    \label{eq:stable-height}
    -f'(u) \geq \frac{2}{3} \abs{f'(1)} \quad \text{where } y' > H.
  \end{equation}
  Let $\m{L} \coloneqq \Delta + f'(u)$ and $u' \coloneqq \partial_y u > 0$, noting that $\m{L}u' = 0$.

  Given $R > H + 1$, let $B_R^{d-1}$ demote the $R$-ball in $\R^{d-1}$ and let $\Gamma_R \coloneqq B_R^{d-1} \times \R$ denote the cylinder of radius $R$.
  We work on the truncation ${G_R \coloneqq \{0 < y' < R\} \cap \Gamma_R}$ of $G$.
  Let
  \begin{equation*}
    \lambda_R \coloneqq \lambda(-\m{L}, G_R).
  \end{equation*}
  Note that $G_R$ is not smooth: it has corners.
  This will not pose difficulties.
  
  Because $G_R$ is bounded, $\lambda_R$ corresponds to a positive principal eigenfunction $\psi_R \in H_0^1(G_R)$.
  The family $(G_R)_{R > 0}$ exhausts $G$, so by \cite[Proposition~2.3(iv)]{BR},
  \begin{equation}
    \label{eq:limit}
    \lambda(-\m{L}, G) = \lim_{R \to \infty} \lambda_R.
  \end{equation}
  We wish to show that the eigenvalues $(\lambda_R)_{R > 0}$ are uniformly positive.
  We note that they are non-increasing in $R$ because the domains $G_R$ are nested.

  If $\lambda_R \geq \frac{1}{3} \abs{f'(1)}$ for all $R > 0$, we are done.
  Suppose otherwise, so there exists $\ubar{R} > H + 1$ such that $\lambda_R < \frac{1}{3}\abs{f'(1)}$ for all $R > \ubar{R}$.
  In the remainder of the proof, we assume $R > \ubar{R}$.
  In the following, we let $\nu$ denote the outward unit normal vector field on the relevant domain of integration.
  Using $-\m{L} \psi_R = \lambda_R \psi_R$ and $\m{L}u' = 0$ and integrating by parts, we find
  \begin{equation*}
    \lambda_R \int_{G_R} u'\psi_R = -\int_{\partial G_R} u' \partial_\nu \psi_R = \int_{\partial G_R} u'\abs{\partial_\nu \psi_R}.
  \end{equation*}
  In the final equality we have used the fact that $\psi_R$ is positive, so $\partial_\nu \psi_R < 0$ on $\partial G_R$.
  Let $\partial_- G_R \coloneqq \partial G \cap \Gamma_R$ denote the bottom boundary of $G_R$.
  Then we have
  \begin{equation}
    \label{eq:ratio}
    \lambda_R > \frac{\int_{\partial_-G_R} u'\abs{\partial_\nu \psi_R}}{\int_{G_R} u' \psi_R}.
  \end{equation}
  We show that this ratio is uniformly positive.

  We will argue that a significant fraction of the mass in $\int_{G_R} u' \psi_R$ is concentrated near the bottom of the domain.
  To this end, let $G_R^+ \coloneqq \{H < y' < R\} \cap \Gamma_R$ denote the portion of $G_R$ that is at least height $H$ above $\partial G$.
  Using \eqref{eq:stable-height} and our assumption that $\lambda_R < \tfrac{1}{3} \abs{f'(1)}$, we compute
  \begin{equation}
    \label{eq:spectral-gap}
    \begin{aligned}
      \int_{G_R^+} \psi_R &\leq \frac{3}{\abs{f'(1)}} \int_{G_R^+} [-f'(u) - \lambda_R]\psi_R\\
                          &\hspace{3cm}= \frac{3}{\abs{f'(1)}} \int_{G_R^+} \Delta \psi_R = \frac{3}{\abs{f'(1)}} \int_{\partial G_R^+} \partial_\nu \psi_R.      
    \end{aligned}
  \end{equation}
  Now $\partial_\nu \psi_R < 0$ on the lateral and top boundaries $\partial G_R \cap \{y' > H\}$, so
  \begin{equation}
    \label{eq:drop-lateral}
    \int_{\partial G_R^+} \partial_\nu \psi_R \leq \int_{\partial_- G_R^+} \abs{\partial_\nu \psi_R},
  \end{equation}
  where $\partial_-G_R^+ \coloneqq \{y' = H\} \cap \Gamma_R$ denotes the bottom boundary of $G_R^+$.
  Combining \eqref{eq:spectral-gap} and \eqref{eq:drop-lateral}, we see that
  \begin{equation}
    \label{eq:bottom}
    \int_{G_R^+} \psi_R \leq \frac{3}{\abs{f'(1)}} \int_{\partial_- G_R^+} \abs{\partial_\nu \psi_R}.
  \end{equation}

  Now let $G_R^- \coloneqq \{0 < y' < H\} \cap \Gamma_R$ denote the portion of $G_R$ within height $H$ of $\partial G$.
  We claim that
  \begin{equation}
    \label{eq:flip}
    \int_{\partial_- G_R^+} \abs{\partial_\nu \psi_R} \lesssim \int_{G_R^-} \psi_R,
  \end{equation}
  where $a \lesssim b$ indicates that $a \leq C b$ for some constant $C$ that does not depend on $R$ but may depend on $r, \gamma, K,$ $M$, $d$, and $f$.
  In the following, we denote coordinates on $G$ by $(x', y')$ for $x' \in \R^d$ and $y' > 0$ defined above.
  We first focus on the ``interior'' portion $\partial_- G_{R-1}^+$ lying at least distance $1$ from $\partial G_R$.
  There, interior Schauder estimates imply that
  \begin{equation*}
    \abs{\nab \psi_R(x', H)} \leq C_{\text{S}} \sup_{B_{1/2}(x',H)} \psi_R,
  \end{equation*}
  where $C_{\text{S}}$ depends only on $d$ and the H\"older norm of $f'(u)$.
  Moreover, Harnack's inequality yields
  \begin{equation*}
    \sup_{B_{1/2}(x',H)} \psi_R \leq C_{\text{H}} \inf_{B_{1/2}(x',H)} \psi_R
  \end{equation*}
  for a constant $C_{\text{H}}$ depending on the same.
  Combining these bounds and integrating in $y'$, we obtain
  \begin{equation*}
    \abs{\nab \psi_R(x', H)} \lesssim \int_{H-1/2}^{H} \psi_R(x', y') \d y' \ForAll x' \in B_{R-1}^{d-1}.
  \end{equation*}
  Extending the region of integration and integrating over $B_{R-1}^{d-1}$, we have
  \begin{equation}
    \label{eq:interior}
    \int_{B_{R - 1}^{d-1}} \abs{\nab \psi_R(x', H)} \d x' \lesssim \int_{B_{R-1}^{d-1}} \int_0^H \psi_R(x', y') \d y' \d x'.
  \end{equation}

  Now consider the portion near the lateral boundary of $G_R$.
  Schauder estimates up to the boundary yield
  \begin{equation}
    \label{eq:Schauder-boundary}
    \abs{\nab \psi_R(x', H)} \leq C_{\text{S}}' \sup_{B_{1/2}(x',H) \cap G_R} \psi_R.
  \end{equation}
  Now define the contraction
  \begin{equation*}
    \iota x' \coloneqq \frac{\abs{x'} - 1}{\abs{x'}} x',
  \end{equation*}
  which moves points in $B_R^{d-1}$ toward the origin by distance $1$.
  Then Carleson's inequality implies that
  \begin{equation}
    \label{eq:Carleson}
    \psi_R(x', y') \leq C_{\text{C}} \psi_R(\iota x', y') \ForAll x' \in B_R^{d-1}\setminus B_{R-1}^{d-1}, \, y' \in (H-1/2, H+1/2).
  \end{equation}
  That is, the value of $\psi_R$ near the boundary is controlled by its values deeper in the interior.
  Combining \eqref{eq:Schauder-boundary} and \eqref{eq:Carleson} with the interior Harnack inequality, we can write
  \begin{equation*}
    \abs{\nab \psi_R(x', H)} \lesssim \int_{H-1/2}^H \psi_R(\iota x', y') \d y' \ForAll x' \in B_R^{d-1} \setminus B_{R-1}^{d-1}.
  \end{equation*}
  Now $\iota$ has bounded Jacobian once $R \geq 2$.
  It follows that
  \begin{equation*}
    \int_{B_R^{d-1} \setminus B_{R-1}^{d-1}} \abs{\nab \psi_R(x', H)} \d x' \lesssim \int_{B_{R-1}^{d-1}} \int_0^H \psi_R(x', y') \d y' \ds x'.
  \end{equation*}
  In conjunction with \eqref{eq:interior}, we find
  \begin{equation*}
    \int_{B_R^{d-1}} \abs{\nab \psi_R(x', H)} \d x' \lesssim \int_{B_R^{d-1}} \int_0^H \psi_R(x', y') \d y' \ds x',
  \end{equation*}
  which is \eqref{eq:flip}.

  We next observe that $u'$ is uniformly positive on $G_R^-$ and uniformly bounded on $G_R$.
  Hence \eqref{eq:bottom} and \eqref{eq:flip} yield
  \begin{equation*}
    \int_{G_R^+} u' \psi_R \lesssim \int_{G_R^+} \psi_R \lesssim \int_{G_R^-} \psi_R \lesssim \int_{G_R^-} u' \psi_R.
  \end{equation*}
  These constants are uniform in $G$, for otherwise we could extract a subsequential limit domain $G^* \in \s{G}$ such that $\partial_y u^{G^*}$ violates the strong maximum principle or Hopf lemma.
  Therefore
  \begin{equation}
    \label{eq:concentration}
    \int_{G_R} u' \psi_R \lesssim \int_{G_R^-} u' \psi_R,
  \end{equation}
  as claimed.
  Moreover, similar reasoning based on \eqref{eq:Carleson} yields
  \begin{equation}
    \label{eq:equivalent}
    \int_{G_R^-} u' \psi_R \lesssim \int_{G_{R-1}^-} u' \psi_R.
  \end{equation}
  Here we apply Carleson's inequality in the vicinity of the corner of $G_R$; this is possible because the inequality holds on Lipschitz domains~\cite{HW68,HW70}.

  For the remainder of the analysis, we work on $\Gamma_{R-1}$ and thus avoid the corners of $G_R$.
  We claim that
  \begin{equation}
    \label{eq:linear}
    \psi_R(x', y') \lesssim \partial_y \psi_R(x', 0) y' \quad \text{in } G_{R - 1}^-.
  \end{equation}
  To see this, suppose to the contrary that there exists a sequence of graphs $G^n \in \s{G},$ radii $R_n \nearrow \infty$, and points $x_n = (x_n',y_n') \in (G_{R_n - 1}^n)^-$ such that
  \begin{equation*}
    \partial_y \psi_{R_n}(x_n', 0) \leq \frac{1}{n y_n'} \psi_{R_n}(x'_n, y'_n) \ForAll n \in \N.
  \end{equation*}
  Let $A_n$ denote the affine transformation
  \begin{equation*}
    A_n x \coloneqq y_n'x + x_n
  \end{equation*}
  and let
  \begin{equation*}
    \Psi_n \coloneqq \frac{\psi_{R_n} \circ A_n}{\psi_{R_n}(x_n)}.
  \end{equation*}
  This satisfies
  \begin{equation}
    \label{eq:operator-sequence}
    -\Delta \Psi_n = (y_n')^2\big(f'\circ u \circ A_n + \lambda_R\big) \Psi_n \quad \text{on } A_n^{-1}G_{R_n}^n
  \end{equation}
  and
  \begin{equation*}
    \partial_y \Psi_n(0,-1) \leq \frac{1}{n},
  \end{equation*}
  noting that $(0, -1) \in \partial A_n^{-1} G_{R_n}^n$.
  Because $\op{Lip} G^n \leq M$, $A_n^{-1}G_{R_n}^n$ contains a ball $B_\rho$ for some $\rho > 0$ independent of $n$.
  Also, $\Psi_n(0) = 1$ and $\m{L}_n\Psi_n = 0$ for the linear operator $\m{L}_n$ appearing in \eqref{eq:operator-sequence}, whose coefficients are uniformly bounded.
  Thus Schauder estimates allow us to extract a subsequential limit of $(\Psi_n)$ as $n \to \infty$.
  We obtain a nonnegative Dirichlet solution $\Psi_\infty$ of $-\m{L}_\infty\Psi_\infty = 0$ on a uniformly Lipschitz epigraph $G_\infty \supset B_\rho$ with $(0,-1) \in \partial G_\infty$ and
  \begin{equation}
    \label{eq:zero-slope}
    \partial_y \Psi_\infty(0,-1) = 0.
  \end{equation}
  The limit operator $\m{L}_\infty$ is either the Laplacian or a rescaling of $\Delta + f'(u^{G_\infty}) - \lambda_R$, depending on whether $(y_n')$ tends to $0$.
  In either case, it satisfies the strong maximum principle and the Hopf lemma.
  However, $\Psi_\infty(0) = 1$, so by the strong maximum principle $\Psi_\infty > 0$.
  Then \eqref{eq:zero-slope} contradicts the Hopf lemma; this contradiction proves \eqref{eq:linear}.
  
  Integrating \eqref{eq:linear}, we see that
  \begin{equation*}
    \int_{G_{R-1}^-} u'\psi_R \lesssim \int_{\partial_- G_{R-1}} u' \partial_y \psi_R.
  \end{equation*}
  In light of \eqref{eq:equivalent}, we find
  \begin{equation*}
    \int_{G_R^-} u'\psi_R \lesssim \int_{\partial_- G_{R-1}} u' \partial_y \psi_R.
  \end{equation*}
  Finally, we observe that the tangential derivatives of $\psi_R$ vanish on $\partial G$, so
  \begin{equation*}
    \partial_y \psi_R = (y \cdot \nu) \partial_\nu \psi_R.
  \end{equation*}
  Because $G$ is uniformly Lipschitz, $-y \cdot \nu$ is uniformly positive.
  So
  \begin{equation}
    \label{eq:boundary}
    \int_{G_R^-} u'\psi_R \lesssim \int_{\partial_- G_{R-1}} u' \partial_y \psi_R \lesssim - \int_{\partial_- G_{R-1}} u' \partial_\nu \psi_R = \int_{\partial_- G_{R-1}} u' \abs{\partial_\nu \psi_R}.
  \end{equation}
  In light of \eqref{eq:concentration} and \eqref{eq:boundary}, \eqref{eq:ratio} implies that $\lambda_R \gtrsim 1$.
  We emphasize that the implied constant is independent of $R$, so $\inf_{R > \ubar{R}} \lambda_R > 0.$
  The proposition follows from \eqref{eq:limit}.
\end{proof}

\subsection{Uniqueness on AUL epigraphs}

Throughout this subsection, fix an AUL epigraph $\Omega$.
Then $\Omega$ is $(\gamma,r,K)$-smooth for some constants $r, K \in \R_+$.
It follows from Definition~\ref{def:AUL} and \cite[Proposition~A.1]{BG24} that there exists $M \in \R_+$ such that the local limits of $\Omega$ at infinity are either $\R^d$ or rotations of epigraphs in $\s{G}(\gamma,r,K,M)$.
Define
\begin{equation}
  \label{eq:inf-eigenvalue}
  \ubar{\lambda} \coloneqq \inf_{G \in \s{G}(\gamma,r,K,M)} \lambda(-\Delta - f'(u^G), G) > 0,
\end{equation}
which is positive by Proposition~\ref{prop:uniform-stability}.
Using this uniform stability, we show that $\Omega$ is ``stable at infinity.''
Given $R \geq 1$, let $\Omega_R$ be a uniformly smooth domain satisfying
\begin{equation}
  \label{eq:truncation}
  \Omega \setminus \bar{B}_{R + 1} \subset \Omega_R \subset \Omega \setminus \bar{B}_R.
\end{equation}
We think of $\Omega_R$ as a smooth approximation of $\Omega \setminus \bar{B}_R$.
\begin{proposition}
  \label{prop:far-stable}
  Let $u$ be a positive bounded solution of \eqref{eq:main}.
  Then
  \begin{equation}
    \label{eq:asymptotic-stability}
    \lim_{R \to \infty} \lambda(-\Delta - f'(u), \Omega_R) \geq \ubar\lambda > 0.
  \end{equation}
\end{proposition}
The proof is similar to that of Proposition~\ref{prop:dilation-asymp-stable}
\begin{proof}
  We argue by contradiction.
  Suppose there exists $\delta > 0$ such that
  \begin{equation}
    \label{eq:small-eigenval}
    \lambda(-\Delta - f'(u), \Omega_R) \leq \ubar\lambda - 3 \delta
  \end{equation}
  for all $R \geq 1$.
  Fix $\rho > 0$ such that $\lambda(-\Delta, B_1) \rho^{-2} \leq \delta$.
  Theorem~\ref{thm:potential-Lieb} (due to Lieb) implies that
  \begin{equation}
    \label{eq:Lieb}
    \lambda(-\Delta - f'(u), \Omega_R) \geq \inf_{x \in \R^d} \lambda\big(\!-\Delta - f'(u), \Omega_R \cap B_\rho(x)\big) - \delta.
  \end{equation}
  Taking $R = n \in \N$ in \eqref{eq:small-eigenval} and using \eqref{eq:Lieb}, we see that there exists a sequence $(x_n)_{n \in \N}$ such that
  \begin{equation}
    \label{eq:Lieb-contradiction-pre}
    \lambda_n \coloneqq \lambda\big(\!-\Delta - f'(u), \Omega_n \cap B_\rho(x_n)\big) \leq \ubar\lambda - \delta.
  \end{equation}
  The domain $\Omega_n \cap B_\rho(x_n)$ must be nonempty (otherwise the eigenvalue is infinite by convention), so by \eqref{eq:truncation} we have $\abs{x_n} \geq n - \rho$.
  In particular, $\abs{x_n} \to \infty$ as $n \to \infty$.
  By \eqref{eq:truncation}, $\Omega_n \subset \Omega$.
  Since increasing the domain decreases the eigenvalue, \eqref{eq:Lieb-contradiction-pre} yields
    \begin{equation}
    \label{eq:Lieb-contradiction}
    \lambda_n \coloneqq \lambda\big(\!-\Delta - f'(u), \Omega \cap B_\rho(x_n)\big) \leq \ubar\lambda - \delta.
  \end{equation}
  
  Again noting that $\op{dist}(x_n, \Omega) \leq \rho$, we can extract a nonempty limit $\Omega^*$ of $\Omega$ along a subsequence that we rename $(x_n)_{n \in \N}$.
  By Definition~\ref{def:AUL}, $\Omega^* = \R^d$ or there exists a rotation $\sigma$ such that $\Omega^* = \sigma G$ for some $G \in \s{G}(\gamma,r,K,M)$.
  First suppose $\Omega^* = \R^d$.
  This implies $\op{dist}(x_n, \Omega^\cc) \to \infty$, and by Lemma~\ref{lem:sub} we have $u \to 1$ uniformly on $B_\rho(x_n)$.
  Then by the continuity of $\lambda$ on bounded domains (see, e.g., \cite{Buttazzo}), we obtain
  \begin{equation*}
    \abs{f'(1)} \leq \lim_{n \to \infty} \lambda_n \leq \ubar\lambda - \delta.
  \end{equation*}
  This contradicts \eqref{eq:uniform-stability}, so we must have $\Omega^* = \sigma G$.
  In the vicinity of $x_n$, $u$ must converge to the unique positive bounded solution $u^G$ on $G$ (after rotation), for Lemma~\ref{lem:sub} prevents $u$ from degenerating to $0$.
  Again, the continuity of $\lambda$ on bounded domains yields
  \begin{equation*}
    \lim_{n \to \infty} \lambda_n = \lambda(-\Delta - f'(u^G), G \cap B_\rho) \geq \lambda(-\Delta - f'(u^G), G) \geq \ubar{\lambda}
  \end{equation*}
  by the definition~\eqref{eq:inf-eigenvalue} of $\ubar\lambda$.
  This contradicts \eqref{eq:Lieb-contradiction}.
  Thus \eqref{eq:small-eigenval} is false and \eqref{eq:asymptotic-stability} follows because $\delta > 0$ is arbitrary.
\end{proof}
Using this asymptotic stability, we prove a maximum principle outside a large ball.
Given $h \geq 0$, define the shift $v_h \coloneqq v(\anon - h \tbf{e}_y)$ on $\Omega + h \tbf{e}_y$, which we extend by $0$ to $\Omega$.
\begin{proposition}
  \label{prop:asymptotic-MP}
  Let $u$ and $v$ be positive bounded solutions of \eqref{eq:main}.
  Then there exists $R > 0$ such that for all $h \geq 0$, if $u \geq v_h$ on $\Omega \cap \partial \Omega_R$, then $u \geq v_h$ in $\Omega_R$.
\end{proposition}
\begin{proof}
  Because $f'$ is uniformly continuous on $(0, 1)$, there exists $\delta > 0$ such that
  \begin{equation}
    \label{eq:uniform-cont}
    \abs{f'(r) - f'(s)} \leq \ubar{\lambda}/3 \ForAll r, s \in (0, 1) \text{ such that } \abs{r - s} < \delta.
  \end{equation}
  We claim that there exists $R > 0$ such that $\lambda(-\Delta - f'(u), \Omega_R) \geq 2\ubar{\lambda}/3$ and
  \begin{equation}
    \label{eq:small-diff}
    \sup_{h \geq 0} \sup_{(\Omega + h \tbf{e}_y) \setminus \bar B_R} (v_h - u) \leq \delta.
  \end{equation}
  By Proposition~\ref{prop:far-stable}, it suffices to show \eqref{eq:small-diff} for sufficiently large $R$.
  Suppose for the sake of contradiction that there exists a sequence $(h_n, x_n)_{n \in \N}$ such that $h_n \geq 0,$ $x_n \in (\Omega + h_n \tbf{e}_y) \setminus \bar B_n,$ and
  \begin{equation}
    \label{eq:diff-contradiction}
    v_{h_n}(x_n) > u(x_n) + \delta.
  \end{equation}
  By Lemma~\ref{lem:sub}, we must have $\sup_{n \in \N} \op{dist}(x_n, \partial\Omega) < \infty$.
  Thus we can extract a subsequence that we rename $(x_n)$ along which $\Omega - x_n$, $u(\anon + x_n)$, and $v_{h_n}(\anon + x_n)$ have limits $\Omega^*,u^*,$ and $v^*$.
  By hypothesis, $\Omega^* = \sigma G$ for some rotation $\sigma$ and an epigraph $G \in \s{G}(\gamma, r, K, M)$.
  By uniqueness on $G$, $u^* \circ \sigma = u^G$, for Lemma~\ref{lem:sub} prevents $u^* = 0$.
  Also, $v^* \circ \sigma$ is a nonnegative subsolution on $G$.
  
  Let $\m{P}$ denote the parabolic semigroup on $G$, so that $w(t, x) \coloneqq (\m{P}_tq)(x)$ solves
  \begin{equation}
    \label{eq:parabolic}
    \begin{cases}
      \partial_t w = \Delta w + f(w) & \text{in } G,\\
      w = 0 & \text{on } \partial G,\\
      w(t=0, \anon) = q.
    \end{cases}
  \end{equation}
  Because $v^* \circ \sigma$ is a subsolution, $\m{P}_t(v^* \circ \sigma)$ is increasing in $t$ and thus has a limit $\m{P}_\infty(v^* \circ \sigma) \geq v^* \circ \sigma$ solving \eqref{eq:main} on $G$.
  By Theorem~\ref{thm:Lipschitz}, the only two solutions are $0$ and $u^G$, so $v^* \circ \sigma \leq u^G$.
  Thus
  \begin{equation*}
    \liminf_{n \to \infty} v_{h_n}(x_n) \leq \limsup_{n \to \infty} u(x_n),
  \end{equation*}
  contradicting \eqref{eq:diff-contradiction}.
  This proves \eqref{eq:small-diff}.
  We fix the corresponding value of $R$ in the remainder of the proof.

  Now fix $h \geq 0$ and suppose $w \coloneqq v_h - u \leq 0$ on $\Omega \cap \partial \Omega_R$.
  Because $u = v_h = 0$ on $\partial\Omega$, we also have $w = 0$ on $\partial\Omega \cap \partial\Omega_R$.
  Hence $w \leq 0$ on $\partial \Omega_R$.
  
  Next, using the mean value theorem, we write
  \begin{equation*}
    -\Delta w - f'(r) w = 0
  \end{equation*}
  for some $r$ between $u$ and $v_h$.
  Let $P \coloneqq \{w > 0\} \cap \Omega_R$.
  Then \eqref{eq:small-diff} and the definition \eqref{eq:uniform-cont} of $\delta$ imply that
  \begin{equation*}
    -\Delta w - f'(u) w - qw = 0 \quad \text{on } P
  \end{equation*}
  for some $q \in \m{C}^\gamma(P)$ satisfying $\abs{q} \leq \ubar\lambda/3$.
  Set $q = 0$ on $P^\cc$ and let ${\m{L} \coloneqq \Delta + f'(u) + q}$.
  Then our choice of $R$ implies that
  \begin{equation*}
    \lambda(-\m{L}, \Omega_R) \geq \lambda(-\Delta - f'(u), \Omega_R) - \ubar{\lambda}/3 \geq \ubar{\lambda}/3 > 0.
  \end{equation*}
  By Theorem~1 of~\cite{Nordmann} or Proposition~3.1 of~\cite{BG24}, $\m{L}$ satisfies the maximum principle on $\Omega_R$.
  Now $\m{L} w|_P = 0$, $\m{L}0 = 0$, and $w \leq 0$ on $P^\cc$.
  It follows that $-\m{L} w_+ \leq 0$ in the sense of distributions.
  Since we have $w \leq 0$ on $\partial\Omega_R$ by hypothesis, the maximum principle implies that $w_+ \leq 0$.
  That is, $w = 0$ and $v_h \leq u$ on $\Omega \setminus B_R$, as desired.

  As in the proof of Lemma~\ref{lem:MP-deep}, we have not quite satisfied the hypotheses of Proposition~3.1 of~\cite{BG24}, but the proof goes through nonetheless.
\end{proof}
\begin{corollary}
  \label{cor:order}
  There exists $H \geq 0$ such that $u \geq v_h$ \corr{for all $h \geq H$}.
\end{corollary}
\begin{proof}
  Take $R > 0$ as in Proposition~\ref{prop:asymptotic-MP} and let
  \begin{equation*}
    h_R \coloneqq \sup_{B_{R}^{d-1}}(-\phi).
  \end{equation*}
  Then if $H \coloneqq h_{R+1} + R+1$, $(\Omega + h \tbf{e}_y) \cap B_{R+1} = \emptyset$ \corr{for all $h \geq H$}.
  By \eqref{eq:truncation}, $u \geq 0 = v_h$ on $\Omega \cap \partial \Omega_R$.
  Then the corollary follows from Proposition~\ref{prop:asymptotic-MP}.
\end{proof}
We can finally apply a sliding argument to prove uniqueness on $\Omega$.
\begin{proof}[Proof of Theorem~\textup{\ref{thm:AUL}}]
  Take $R > 0$ as in Proposition~\ref{prop:asymptotic-MP} and define
  \begin{equation*}
    h_* \coloneqq \inf\{h \geq 0 \mid u \geq v_{h'} \corr{\text{ for all } h' \geq h}\}.
  \end{equation*}
  By Corollary~\ref{cor:order}, $0 \leq h_* \leq H$.
  We wish to show that $h_* = 0.$
  Suppose for the sake of contradiction that $h_* > 0$.
  By continuity, we have $u \geq v_{h_*}$.
  Since $h_* > 0$, $u \neq v_{h_*}$.
  Thus by the strong maximum principle and compactness,
  \begin{equation*}
    \inf_{\bar\Omega_{h_*/2} \cap \bar B_{R+1}} (u - v_{h_*}) > 0.
  \end{equation*}
  Hence by continuity, there exists $\eps \in (0, h_*/2)$ such that $u \geq v_{h_* - \eps}$ on $\Omega \cap \bar B_{R+1}$.
  Then by \eqref{eq:truncation} and Proposition~\ref{prop:asymptotic-MP}, we have $u \geq v_{h_* - \eps}$ in $\Omega_R$, and hence in $\Omega$.
  This contradicts the definition of $h_*$, so in fact $h_* = 0$.
  That is, $u \geq v$.
  By symmetry, $u = v$.
  Finally, since $u \geq u_h$ for all $h \geq 0$, we have $\partial_y u \geq 0$.
  Then the strong maximum principle implies that $\partial_y u > 0$.
\end{proof}

\subsection{A marginally stable epigraph}
\label{sec:wells}

The previous proof points to a broad principle: an epigraph $\Omega$ enjoys uniqueness whenever its far-field limits are (isometries of) epigraphs with unique, strictly stable solutions.
This suggests a possible induction on ever larger classes of epigraphs.
At step $n \in \N$, we might prove the uniqueness and strict stability of positive bounded solutions of \eqref{eq:main} on epigraphs in some class $\m{C}_n$.
By the reasoning in the preceding subsection, this would imply uniqueness on the class $\m{C}_{n+1}$ of all epigraphs whose far-field limits are isometries of epigraphs in $\m{C}_n$.
For instance, we could take the set of uniformly Lipschitz epigraphs as the base case $\m{C}_1$; then $\m{C}_2$ would be the set of AUL epigraphs.
We might hope that such a procedure could lead to a proof of uniqueness on \emph{all} epigraphs.

We record two difficulties with this program.
First, we have only demonstrated half of the inductive step: we have shown uniqueness but \emph{not} strict stability for solutions on AUL epigraphs (though stability seems plausible).
Second, there exist epigraphs whose solutions are at best marginally stable:
\begin{proposition}
  \label{prop:marginal}
  There exists a positive reaction $f$ and an epigraph $\Omega \subset \R^2$ such that for any positive bounded solution $u$ of \eqref{eq:main} on $\Omega$,
  \begin{equation*}
    \lambda(-\Delta - f'(u), \Omega) \leq 0.
  \end{equation*}
\end{proposition}
This proposition asserts nothing about uniqueness on $\Omega$, and uniqueness may well hold on all epigraphs.
Nonetheless, this marginally stable example shows that the putative induction described would break down once $\Omega \in \m{C}_n$.

The epigraph we construct in the proof of Proposition~\ref{prop:marginal} has an infinite sequence of ever deeper ``wells'' of some width $L$, as in Figure~\ref{fig:counter}(b).
As a result, $\Omega$ has the infinite strip $(0, L) \times \R$ as a local limit.
By choosing $f$ and $L$ carefully, we can ensure that any solution $u$ has vanishing eigenvalue on this limit strip.
We emphasize that $\Omega$ is not asymptotically uniformly Lipschitz, so Proposition~\ref{prop:marginal} does contradict Theorem~\ref{thm:AUL}.
The proof of Proposition~\ref{prop:marginal} has a distinct character from the rest of the paper, so we relegate it to Appendix~\ref{app:marginal}.

\section{Non-uniqueness in domains with pockets}
\label{sec:pocket}

Thus far, we have largely focused on proving uniqueness in domains satisfying various structural assumptions.
In contrast, in this section we describe a robust method to produce examples of nonuniqueness.

The idea is so attach a bounded ``pocket'' $\Pi$ to a given domain $\Omega_0$ via a narrow ``bridge'' $\Gamma$ to form a ``composite domain'' $\Omega = \Pi \cup \Gamma \cup \Omega_0$; see Figure~\ref{fig:pocket-precise} for an illustration.
\begin{figure}[t]
  \centering
  \includegraphics[width = 0.7\linewidth]{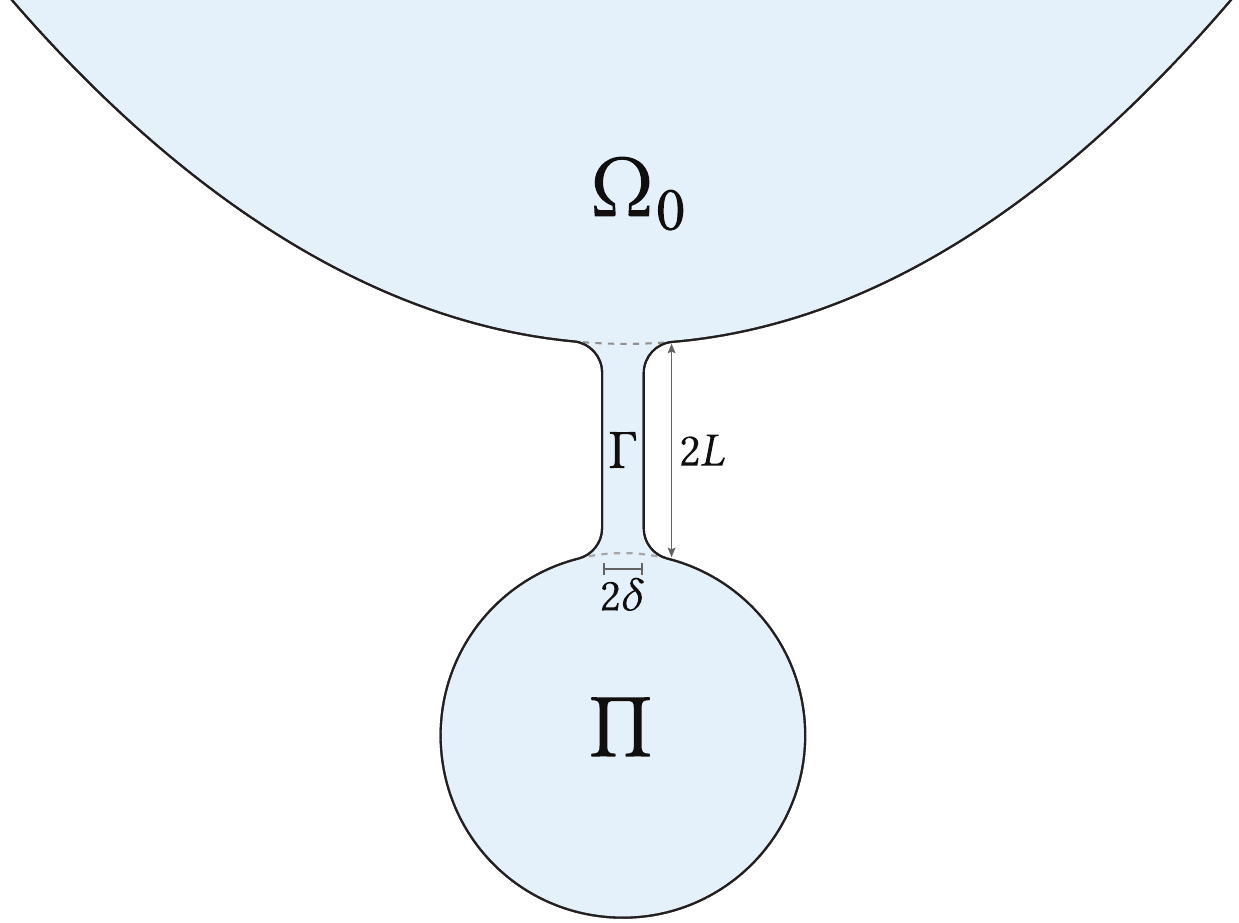}
  \caption[Detailed domain with pocket]{A pocket $\Pi$ attached to $\Omega_0$ via a cylindrical bridge $\Gamma$ of length $2L$ and diameter $2\delta$.}
  \label{fig:pocket-precise}
\end{figure}
Given a pocket $\Pi$, we choose a positive reaction $f$ such that \eqref{eq:main} admits multiple positive solutions on $\Pi$.
We show that this nonuniqueness extends to the entire composite domain $\Omega$ provided the bridge $\Gamma$ is sufficiently narrow.

With this motivation, we study the behavior of solutions on the narrow bridge $\Gamma = B_\delta^{d-1} \times (-L, L)$, where $\delta, L > 0$ and $B_\delta^{d-1}$ denotes the $\delta$-ball in $\R^{d-1}$.
Technically, we must augment $\Gamma$ slightly at either end to smoothly join $\Pi$ and $\Omega_0$.
We elide this point, as it poses no problems for our argument.
We denote coordinates on $\Gamma$ by $(x', y)$.
Given $\mu > 0$, let $\m{F}_\mu$ denote the set of Lipschitz functions $f \colon [0, 1] \to \R$ such that $f(0) = f(1) = 0$ and $\op{Lip} f \leq \mu$.
In an exception to our standing assumptions, the reactions in $\m{F}_\mu$ need not be smoother than Lipschitz.
\begin{lemma}
  \label{lem:exp-decay}
  Suppose $\Gamma \subset \Omega$ and $\partial B_\delta^{d-1} \times (-L, L) \subset \partial \Omega$ for some $L > 0$.
  For every $\mu > 0$, there exists $\delta(d, \mu, L) > 0$ such that for all $f \in \s{F}_\mu$ and every solution $0\leq u \leq 1$ of \eqref{eq:main},
  \begin{equation*}
     u(0, \anon) \leq \frac{1}{4} \phi,
   \end{equation*}
   where $\phi$ denotes the positive Dirichlet principal eigenfunction of $-\Delta$ on $B_\delta^{d-1}$ such that $\norm{\phi}_\infty = 1$.
\end{lemma}
\noindent
Thus solutions of \eqref{eq:main} on $\Omega$ are small at the midpoint of a narrow bridge.
\begin{proof}
  Let $\phi$ denote the positive principal eigenfunction of $-\Delta$ on $B_\delta^{d-1}$ normalized by $\norm{\phi}_\infty = 1$.
  Let $\al \coloneqq \lambda(-\Delta, B^{d-1})$, so that $\lambda(-\Delta, B_\delta^{d-1}) = \al \delta^{-2}$.
  We assume $\al \delta^{-2} > \mu$.
  Take $f \in \s{F}_\mu$, and a solution $0 \leq u \leq 1$ of \eqref{eq:main}.
  
  Because $f$ is $\m{C}^{0, 1}$, Schauder estimates imply that $\abs{\nab u} \lesssim_{d, \mu} \delta^{-1},$ where the subscript $\lesssim_{d,\mu}$ indicates that the implied constant can depend on $d$ and $\mu$.
  After all, if we dilate $\Omega$ by the factor $\delta^{-1}$, we obtain a uniformly smooth domain on which the gradient is order $1$.
  Shrinking to the original size, the gradient grows by a factor of $\delta^{-1}$.
  Integrating from the boundary, we find $u \lesssim_{d, \mu} \delta^{-1} \op{dist}(x, \partial \Omega)$.
  On the other hand, $\phi(x') \gtrsim_d \delta^{-1} \op{dist}(x', \partial B_\delta^{d-1})$.
  It follows that
  \begin{equation}
    \label{eq:initial-bd}
    u|_{\Gamma} \leq A \phi
  \end{equation}
  for some $A(d, \mu) > 0$ independent of $\delta$.
  
  By \eqref{eq:initial-bd} and the parabolic comparison principle, $u|_\Gamma$ is bounded above by the solution $z$ of the linear evolution equation
  \begin{equation*}
    \begin{cases}
      \partial_t z = (\Delta + \mu) z & \text{in } \Gamma,\\
      z = 0 & \text{on } \partial B_\delta^{d-1} \times (-L, L),\\
      z = A\phi & \text{on } B_\delta^{d-1} \times \{\pm L\},\\
      z(t = 0, \anon) = A \phi & \text{in } \Gamma.
    \end{cases}
  \end{equation*}
  That is, $z \geq u$ in $\Gamma$.
  Decomposing $z$ in the cross-sectional eigenbasis of the Laplacian, we see that $z = Z(t, y) \phi(x')$ for a function $Z$ solving the one-dimensional equation
  \begin{equation*}
    \begin{cases}
      \partial_t Z = (\partial_y^2 + \mu - \al \delta^{-2}) Z & \text{in } (-L, L),\\
      z(t, \pm L) = A,\\
      z(0, \anon) = A & \text{in } (-L, L).
    \end{cases}
  \end{equation*}
  Because $\al \delta^{-2} > \mu$, $Z$ converges exponentially in time to the unique steady state
  \begin{equation*}
    Z(\infty, y) = A \frac{\cosh\left(y\sqrt{\al \delta^{-2} - \mu}\right)}{\cosh\left(L\sqrt{\al \delta^{-2} - \mu}\right)}.
  \end{equation*}
  Because $\cosh \xi \geq \tfrac{1}{2} \e^\xi$, we obtain
  \begin{equation*}
    u(x', 0) \leq z(\infty, x', 0) = Z(\infty, 0) \phi(x') \leq 2 A(d, \mu) \e^{-L\sqrt{\al \delta^{-2} - \mu}} \phi(x').
  \end{equation*}
  We choose $\delta(d, \mu, L) > 0$ sufficiently small that $2A(d, \mu) \e^{-L\sqrt{\al \delta^{-2} - \mu}} \leq 1/4$.
\end{proof}
We now state our main nonuniqueness result.
\begin{theorem}
  \label{thm:pocket-precise}
  Given a bounded domain $\Pi$ and $L > 0$, there exist a positive reaction $f$ and $\delta > 0$ such that \eqref{eq:main} admits multiple positive bounded solutions on $\Omega = \Pi \cup \Gamma \cup \Omega_0$.
\end{theorem}
\begin{proof}
  In~\cite[Proposition~1.4]{BG24}, we showed that there exists $f_0 \geq 0$ of the form in Figure~\ref{fig:reaction}(a) such that \eqref{eq:main} admits multiple positive solutions on $\Pi$ with reaction $f_0$.
  \begin{figure}[t]
    \centering
    \includegraphics[width = 0.9\linewidth]{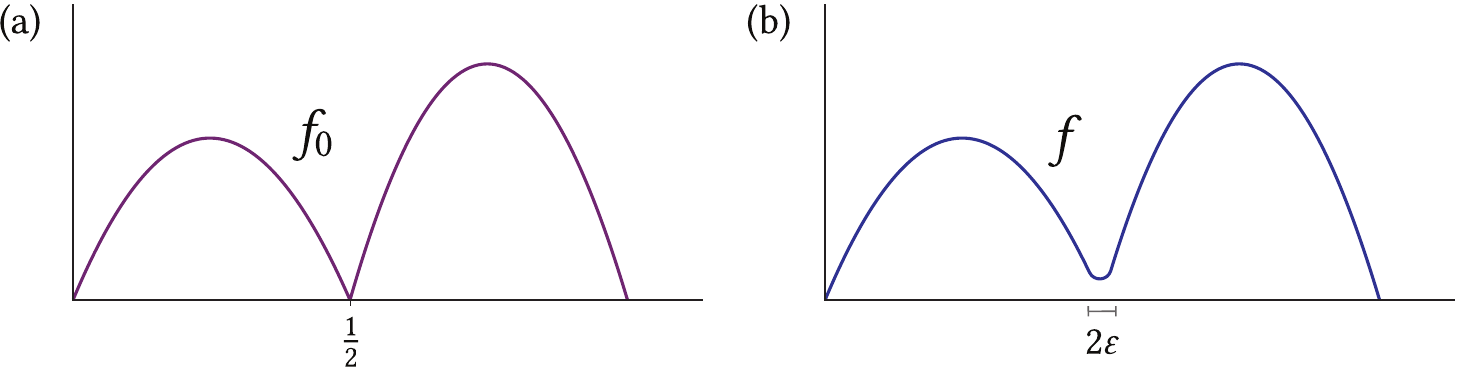}
    \caption[Reactions used to construct nonuniqueness]{
      Reactions used in the proof of nonuniqueness.
      (a) A preliminary reaction that is positive save for a zero at $s = 1/2$.
      (b) A true positive reaction.
    }
    \label{fig:reaction}
  \end{figure}
  In particular, there exist positive solutions $\ubar u_0, \bar{u}_0$ such that
  \begin{equation}
    \label{eq:distinct}
    \ubar{u}_0 < \frac{1}{2} \And \sup \bar{u}_0 > \frac{1}{2}.
  \end{equation}
  
  Let $\mu \coloneqq \op{Lip} f_0$ and take $\delta(d, \mu, L) > 0$ as in Lemma~\ref{lem:exp-decay}.
  Let $\phi$ denote the positive Dirichlet principal eigenfunction of $-\Delta$ on $B_\delta^{d-1}$ such that $\norm{\phi}_\infty = 1$.
  Let $\Pi_+ \coloneqq \Pi \cup \big(B_\delta^{d-1} \times [-L, 0)\big)$ denote the union of the pocket $\Pi$ and the bottom half of the bridge $\Gamma$.
  Then $\partial \Pi_+ \cap \Omega = B_\delta^{d-1} \times \{0\}$.
  By Lemma~\ref{lem:exp-decay}, every solution $0 \leq u \leq 1$ of \eqref{eq:main} on $\Omega$ with a reaction $f \in \m{F}_\mu$ satisfies
  \begin{equation}
    \label{eq:quarter}
    u \leq \frac{1}{4} \phi \quad \text{on } B_\delta^{d-1} \times \{0\}.
  \end{equation}

  Recall the solution $\ubar{u}_0 < 1/2$ of \eqref{eq:main} on $\Pi$ with reaction $f_0$ from \eqref{eq:distinct}.
  Extend $\ubar{u}_0$ by $0$ to $\Omega$ and let $v$ solve the parabolic equation
  \begin{equation}
    \label{eq:parabolic-half}
    \begin{cases}
      \partial_t v = \Delta v + f_0(v) & \text{in } \Pi_+,\\
      v = 0 & \text{on } \partial \Pi_+ \cap \partial \Omega,\\
      v = \tfrac{1}{4} \phi & \text{on } B_\delta^{d-1} \times \{0\},\\
      v(t=0, \anon) = \ubar{u}_0 & \text{in } \Pi_+.
    \end{cases}
  \end{equation}
  Then the initial condition $\ubar{u}_0$ is a subsolution of \eqref{eq:parabolic-half}, so $v$ is increasing in $t$ and has a long-time limit $v(\infty, \anon)$.
  On the other hand, $1/2$ is a supersolution of \eqref{eq:parabolic-half} because $f_0(1/2) = 0$, $\phi \leq 1$, and $\ubar{u}_0 < 1/2$.
  Thus by the strong maximum principle, $v(\infty, \anon) < 1/2$ on $\bar{\Pi}_+$.
  Because $\Pi_+$ is bounded, we have
  \begin{equation*}
    \eps \coloneqq \frac{1}{2} - \sup_{\Pi_+} v(\infty, \anon) > 0.
  \end{equation*}
  We now increase $f_0$ on the region $(1/2-\eps, 1/2+\eps)$ to be positive and smooth while remaining in $\s{F}_\mu$.
  Let $f$ denote the new reaction, as depicted in Figure~\ref{fig:reaction}(b).

  Recall the parabolic semigroup $\m{P}$ from \eqref{eq:parabolic}, defined now on the composite domain $\Omega$.
  Because $\ubar u_0$ is a subsolution on $\Omega,$ $\m{P}_t \ubar{u}_0$ is increasing in $t$ and has a long-time limit $\ubar u \coloneqq \m{P}_\infty \ubar{u}_0$ solving \eqref{eq:main}.
  By \eqref{eq:quarter}, $\m{P}_t \ubar{u}_0 \leq \ubar u \leq \phi/4$ on $\partial\Pi_+\cap \Omega$.
  It follows that $v$ is a supersolution for $\m{P}|_{\Pi_+}$.
  After all, $v \leq 1/2 - \eps$ for all time, so $f_0(v) = f(v)$.
  Therefore $(\m{P} \ubar{u}_0)|_{\Pi_+} \leq v$.
  In the long-time limit, we have $\ubar u|_{\Pi_+} \leq v(\infty, \anon)$.
  In particular,
  \begin{equation}
    \label{eq:lower}
    \sup_{\Pi} \ubar{u} < \frac{1}{2}.
  \end{equation}

  Finally, recall the larger solution $\bar{u}_0 > \ubar{u}_0$ of \eqref{eq:main} on $\Pi$ with reaction $f_0$, which satisfies $\sup_\Pi \bar{u}_0 > 1/2$.
  We extend $\bar{u}_0$ by $0$ to $\Omega$.
  This is a subsolution for \eqref{eq:main} (because $f \geq f_0$), so the limit $\bar{u} \coloneqq \m{P}_\infty \bar{u}_0 \geq \bar{u}_0$ exists and solves \eqref{eq:main}.
  It follows from \eqref{eq:distinct} that
  \begin{equation*}
    \sup_{\Pi} \bar u \geq \sup_\Pi \bar{u}_0 > \frac{1}{2}.
  \end{equation*}
  Comparing with \eqref{eq:lower}, we see that $\bar{u} \neq \ubar{u}$.
  That is, \eqref{eq:main} admits multiple bounded positive solutions on $\Omega$ with the positive reaction $f$.
\end{proof}
In Sections~\ref{sec:exterior-star}--\ref{sec:epigraphs}, we described three classes of structured domains that enjoy uniqueness (under certain conditions): exterior-star, large dilations, and epigraphs.
The nonuniqueness constructed above demonstrates the importance of these structural assumptions: if they are violated in a suitable compact set, uniqueness is lost.
We present this phenomenon in Figure~\ref{fig:nonunique}.
Domains in the first row satisfy our structural hypotheses and have unique positive bounded solutions, while domains in the second have pockets that support multiple positive bounded solutions.
\begin{figure}[t]
  \centering
  \includegraphics[width = \linewidth]{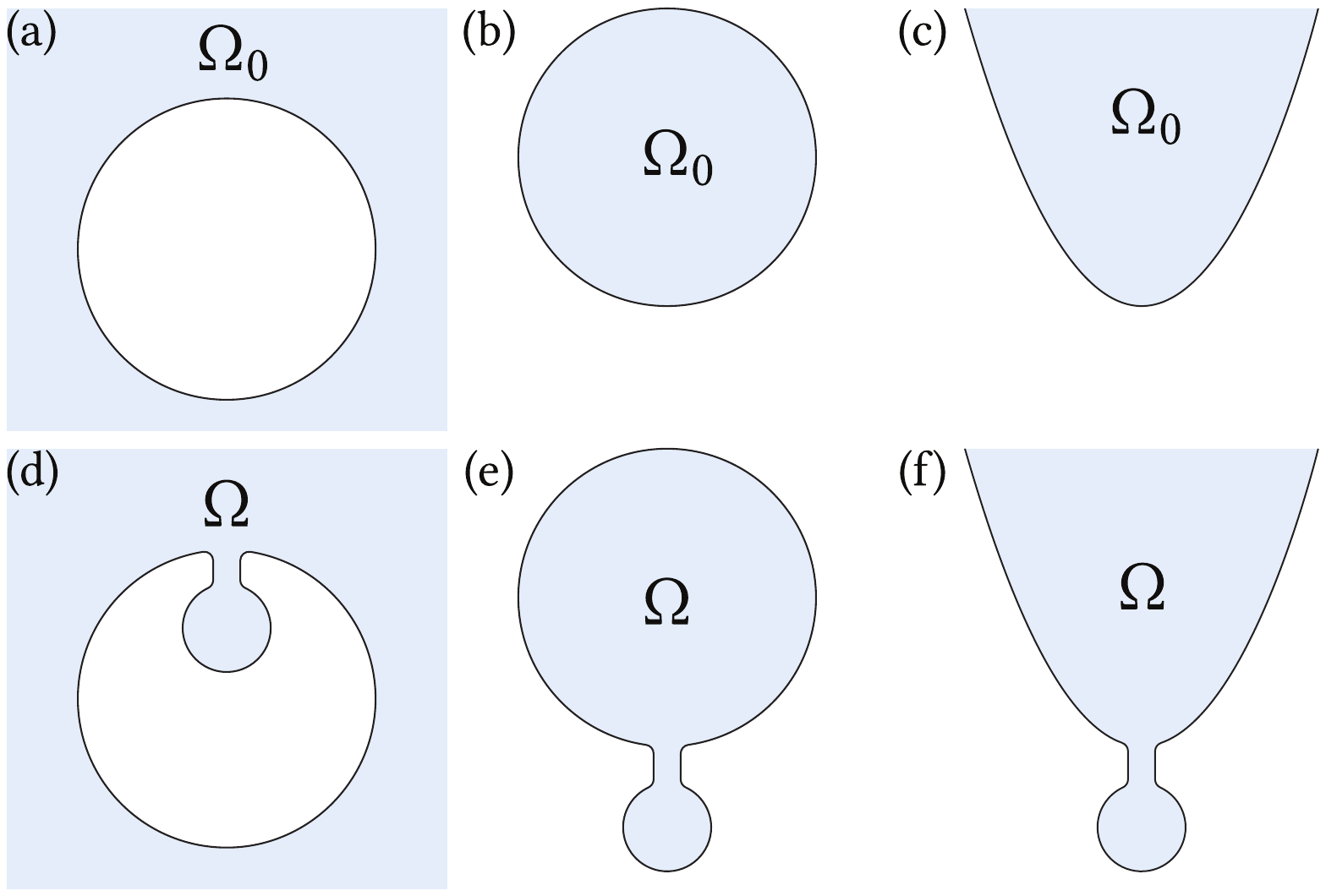}
  \caption[Examples of pockets causing nonuniqueness]{
    Theorems~\ref{thm:ES}--\ref{thm:flat} ensure that \eqref{eq:main} admits a unique positive bounded solution on the following domains $\Omega_0$:
    (a) The exterior of a ball.
    (b) A large ball $B_R$ for $R \gg 1$ depending on the reaction.
    (c) A parabola.
    However, Theorem~\ref{thm:pocket} provides a pocket $\Pi$ and a reaction $f$ such that \eqref{eq:main} admits multiple positive bounded solutions on the augmented domains $\Omega = \Pi \cup \Gamma \cup \Omega_0$ in (d)--(f).
  }
  \label{fig:nonunique}
\end{figure}

\section{Multiple solutions in a cylinder}
\label{sec:cylinder}

In this section, we make a detailed study of \eqref{eq:main} on cylinders with bounded smooth cross-section $\omega \subset \R^{d-1}$.
Our motivation is twofold.
First, as Figure~\ref{fig:counter}(b) demonstrates, cylinders can arise as limits of epigraphs.
Indeed, we can loosely think of a cylinder as an infinitely deep and infinitely steep epigraph.
Thus any form of nonuniqueness on cylinders contrasts strikingly with uniqueness on AUL epigraphs (Theorem~\ref{thm:AUL}).
Second, in the previous section we demonstrated how to robustly construct domains with multiple solutions (Theorem~\ref{thm:pocket-precise}).
There, we are only guaranteed the existence of two solutions.
In this section, we show that a cylinder can support uncountably many solutions.

As noted in the introduction, positive reactions on the line $\R$ admit only one positive bounded solution: their stable root $1$.
We contrast this with bistable reactions, namely $f$ for which $f < 0$ on $(0, \theta)$ and $f > 0$ on $(\theta, 1)$ for some $\theta \in (0, 1)$.
On $\R$, one can perturb around the unstable intermediate root $\theta$ to produce (unstable) oscillatory solutions on $\R$.
These vary in $x$ and are thus not translation-invariant.

We use this phenomenon as a guide to Proposition~\ref{prop:cylinder}.
In a precise sense, positive reactions can appear ``bistable'' on the interval: they may admit multiple stable solutions separated by intermediate unstable solutions.
Moving up in dimension from the interval to the strip, we can perturb around the unstable cross-section to produce oscillatory solutions in the strip.

\subsection{Instability on the interval}
\label{sec:unstable-interval}

As a first step, we show that reactions with $f''(0) > 0$ admit unstable solutions with particularly simple spectra.
\begin{lemma}
  \label{lem:simple-unstable}
  Let $f$ be a $\m{C}^2$ positive reaction with $f''(0) > 0$.
  Then there exists a length $L > 0$ and a solution $\phi$ of \eqref{eq:main} on $(0, L)$ such that the Dirichlet spectrum of the operator $-\partial_x^2 - f'(\phi)$ consists of one negative (principal) eigenvalue and infinitely many positive eigenvalues.
\end{lemma}
\noindent
Informally, $\phi$ is strictly unstable in one direction but strictly stable in all others.

To prove Lemma~\ref{lem:simple-unstable}, we employ the shooting method.
Given $\al \in \R_+$, let $\phi_\al$ solve the initial-value ODE
\begin{equation}
  \label{eq:shooting}
  -\phi_\al'' = f(\phi_\al), \quad \phi_\al(0) = 0, \enspace \phi_\al'(0) = \al.
\end{equation}
We are interested in values of $\al$ for which $\phi_\al$ bends back down to zero at some positive location $L_\al$.
Recalling the $\R_+$-steady state $\varphi$ from~\cite[Theorem~1.1(A)]{BG22a}, let $\al^* \coloneqq \varphi'(0)$.
In the proof of \cite[Lemma~2.4]{BG22a}, we showed that $\phi_\al$ bends back to zero if and only if $\al \in (0, \al^*)$.
(The lemma is stated for other reaction classes, but this fact applies to positive reactions, as noted in the proof of Lemma~2.5 in the same paper.)
So for all $\al \in (0, \al^*)$, $\phi_\al$ has a first positive zero $L_\al$ and is thus a positive solution of \eqref{eq:main} on the interval $(0, L_\al)$.

Let $s_\al \coloneqq \phi_\al(L_\al/2)$ denote the maximum value of $\phi_\al$ and let $F(s) \coloneqq \int_0^s f(r) \d r$ denote the antiderivative of $f$.
Multiplying \eqref{eq:shooting} by $\phi_\al'$ and integrating, we find
\begin{equation}
  \label{eq:first-integral}
  (\phi_\al')^2 - \al^2 = - 2 F(\phi_\al).
\end{equation}
Since $\phi_\al' = 0$ at its maximum, this yields
\begin{equation}
  \label{eq:shooting-max}
  \al^2 = 2 F(s_\al).
\end{equation}
On the other hand, we can rearrange \eqref{eq:first-integral} and integrate again to obtain
\begin{equation}
  \label{eq:shooting-length}
  L_\al = 2 \int_0^{s_\al} \frac{\dn s}{\sqrt{\al^2 - 2F(s)}}.
\end{equation}
(These calculations are presented in greater detail in the proof of \cite[Lemma~2.4]{BG22a}.)
If a function $g$ depends on $\al$, we use the notation $\dot{g}$ to denote the derivative $\partial_\al g$. 
To prove Lemma~\ref{lem:simple-unstable}, we first show that $\dot{L} < 0$ implies instability.
\begin{lemma}
  \label{lem:instability}
  If $\dot{L}_\beta < 0$ for some $\beta \in (0, \al^*)$, then $\phi_\beta$ is strictly unstable:
  \begin{equation*}
    \lambda\big(\!-\partial_x^2 - f'(\phi_\beta), (0, L_\beta)\big) < 0.
  \end{equation*}
\end{lemma}
In \cite[Lemma~2.4]{BG22a}, we showed that $L_\al \to \infty$ as $\al \to \infty$.
As a result, if $\dot{L}_\beta < 0$, then there exists $\al > \beta$ such that $L_\al = L_\beta$.
That is, we have nonuniqueness.
This demonstrates a link between instability and nonuniqueness.
For a partial converse, see Lemma~\ref{lem:unique-stable} below.
\begin{proof}[Proof of Lemma~\textup{\ref{lem:instability}}]
  Assume $\dot{L}_\beta < 0$ for some $\beta \in (0, \al^*)$.
  By symmetry,
  \begin{equation*}
    \phi_\beta'(L_\beta) = -\phi_\beta'(0) = -\beta.
  \end{equation*}
  It follows that
  \begin{equation*}
    0 = \der{}{\al}[\phi_\al(L_\al)]\Big|_{\al = \beta} = \dot{\phi}_\beta(L_\beta) + \phi_\beta'(L_\beta) \dot{L}_\beta = \dot{\phi}_\beta(L_\beta) - \beta \dot{L}_\beta.
  \end{equation*}
  Rearranging, we see that $\dot{\phi}_\beta(L_\beta) = \beta \dot{L}_\beta < 0$.
  On the other hand, $\dot{\phi}_\beta(0) = 0$ and
  \begin{equation*}
   \dot\phi_\beta'(0) = \pa{^2}{x \partial \al} \phi_\al \Big|_{\al = \beta, x = 0} = \pa{^2}{\al \partial x} \phi_\al \Big|_{\al = \beta, x = 0} = \der{}{\al}[\phi_\al'(0)]\Big|_{\al = \beta} = \der{}{\al} \al \Big|_{\al = \beta} = 1.
  \end{equation*}
  So $\psi \coloneqq \dot{\phi}_\beta$ satisfies $\psi(0) = 0$, $\psi'(0) = 1$, and $\psi(L_\beta) < 0$.
  It follows that $\psi$ vanishes somewhere between $0$ and $L_\al$.
  Let $\ell$ denotes its earliest intermediate zero:
  \begin{equation*}
    \ell \coloneqq \min\{x \in (0, L_\al) \mid \psi(x) = 0\} \in (0, L_\beta).
  \end{equation*}
  Then $\psi > 0$ on $(0, \ell)$ and $\psi(0) = \psi(\ell) = 0$.
  Moreover, if we let $\m{L} \coloneqq \partial_x^2 + f'(\phi_\beta)$ and differentiate \eqref{eq:shooting} with respect to $\al$ at $\al = \beta$, we find $\m{L} \psi = 0.$
  It follows that $\psi$ is the principal eigenfunction of $\m{L}$ on $(0, \ell)$ and
  \begin{equation*}
    \lambda\big(\!-\m{L}, (0, \ell)\big) = 0.
  \end{equation*}
  The eigenvalue can only fall when we increase the domain, so $\lambda\big(\!-\m{L}, (0, L_\beta)\big) \leq 0.$
  Moreover, this eigenvalue cannot be $0$, for otherwise its principal eigenfunction would coincide with $\psi$ (by ODE uniqueness), and $\psi$ is not positive on the entirety of $(0, L_\al)$.
  So in fact $\lambda\big(\!-\m{L}, (0, L_\beta)\big) < 0,$ as desired.
\end{proof}
We use this lemma to construct the simply unstable solution $\phi$ in Lemma~\ref{lem:simple-unstable}.
\begin{proof}[Proof of Lemma~\textup{\ref{lem:simple-unstable}}]
  Let $m \coloneqq f'(0) > 0$.
  The function $\al \mapsto L_\al$ is continuous on $(0, \al^*)$ by ODE well-posedness.
  We consider the behavior of $L_\al$ when $\al \ll 1$.
  This is an almost-linear regime in which $f(s)$ is well-approximated by $ms$.
  Hence the maximal value $s_\al$ tends to $0$ and the length $L_\al$ tends to the value $L_*$ for which $\lambda\big(\!-\partial_x^2 - m, (0, L_*)\big) = 0$.
  The principal eigenfunction is sinusoidal, and one can explicitly compute $L_* = \pi m^{-1/2}$.
  So $s_\al \to 0$ and $L_\al \to L_*$ as $\al \searrow 0$.
  
  Using the substitution $s = s_\al r$ in the first integral in \eqref{eq:shooting-length} as well as \eqref{eq:shooting-max}, we find
  \begin{equation*}
    L_\al = \sqrt{2} \int_0^1 \left[s_\al^{-2}F(s_\al) - s_\al^{-2} F(s_\al r)\right]^{-1/2} \d r.
  \end{equation*}
  Let
  \begin{equation*}
    D(\al, r) \coloneqq s_\al^{-2}F(s_\al) - s_\al^{-2} F(s_\al r)
  \end{equation*}
  denote the denominator.
  Using the hypothesis $f''(0) > 0$, we show that $\dot{D} > 0$ for all $r \in (0, 1)$ when $\al \ll 1$, which implies that $\dot{L}_\al < 0$.

  Differentiating, we can write $\dot{D} = -\dot{s}_\al s_\al^{-3} E$ for
  \begin{equation*}
    E(\al, r) \coloneqq 2F(s_\al) - 2F(s_\al r) -[s_\al f(s_\al) - s_\al r f(s_\al r)].
  \end{equation*}
  Using \eqref{eq:shooting-max}, we have
  \begin{equation}
    \label{eq:max-deriv}
    \dot{s}_\al = \frac{\al}{f(s_\al)} > 0.
  \end{equation}
  It thus suffices to show that $E < 0$ when $r \in (0, 1)$ and $\al \ll 1$.
  Using the definition of $F$, we can write
  \begin{equation}
    \label{eq:E}
    E = 2 \int_{r s_\al}^{s_\al} f(t) \d t - tf(t) \Big|_{rs_\al}^{s_\al}.
  \end{equation}
  Integrating by parts, we have
  \begin{equation*}
    \int_{r s_\al}^{s_\al} f(t) \d t = tf(t) \Big|_{rs_\al}^{s_\al} - \int_{r s_\al}^{s_\al} t f'(t) \d t.
  \end{equation*}
  Using this in \eqref{eq:E}, we find
  \begin{equation}
    \label{eq:E-int}
    E = \int_{r s_\al}^{s_\al} [f(t) - t f'(t)] \d t.
  \end{equation}
  Because $f''(0) > 0$ and $f \in \m{C}^2$, $f$ is strictly convex on some interval $[0, \eps]$ with $\eps \in (0, 1)$.
  By the mean value theorem and strict convexity,
  \begin{equation*}
    \frac{f(t)}{t} = \frac{f(t) - f(0)}{t - 0} < f'(t) \ForAll t \in (0, \eps].
  \end{equation*}
  That is, the integrand in \eqref{eq:E-int} is negative.
  Fix $\ubar{\al} \in (0, \al^*)$ such that
  \begin{equation*}
    \sup_{\al \in (0, \ubar{\al}]} s_\al \leq \eps.
  \end{equation*}
  For all $\al \in (0, \ubar{\al}]$ and $r \in [0, 1)$, we have shown that $E < 0$.
  As shown above, this implies that $\dot{L}_\al < 0$ for all $\al \in (0, \ubar{\al}]$.

  By Proposition~2.3(vii) of~\cite{BR}, the principal eigenvalue is continuous in the potential.
  By scaling, it is also continuous in the length $L$.
  Therefore
  \begin{equation*}
    \lambda\big(\!-\partial_x^2 - f'(\phi_\al), (0, L_\al)\big) \to \lambda\big(\!-\partial_x^2 - m, (0, L_*)\big) = 0 \quad \text{as } \al \searrow 0.
  \end{equation*}
  The principal eigenvalue is simple and thus has a spectral gap.
  That is, there exists $\delta > 0$ such that all non-principal eigenvalues of $-\partial_x^2 - m$ on $(0, L_*)$ exceed $2\delta$.
  Using the minimax formula for eigenvalues, one can readily check that the second-smallest Dirichlet eigenvalue is continuous in $\al$ as well.
  Thus there exists $\al \in (0, \ubar{\al}]$ such that all non-principal eigenvalues of $-\partial_x^2 - f'(\phi_\al)$ on $(0, L_\al)$ exceed $\delta$.
  On the other hand, because $\dot{L}_\al < 0$, Lemma~\ref{lem:instability} shows that
  \begin{equation*}
    \lambda\big(\!-\partial_x^2 - f'(\phi_\al), (0, L_\al)\big) < 0.
  \end{equation*}
  This completes the proof with $\phi = \phi_\al$ and $L = L_\al$.
\end{proof}

\subsection{Spatial dynamics}
We now deploy the theory of spatial dynamics to prove Proposition~\ref{prop:cylinder}.
Throughout this subsection, we fix the length $L$ and solution $\phi$ from Lemma~\ref{lem:simple-unstable}.
Let $\m{L} \coloneqq \partial_y^2 + f'(\phi)$ denote the linearization of \eqref{eq:main} about $\phi$ on $(0, L)$.
We are interested in solutions near $\phi$ of \eqref{eq:main} on the strip $\R \times (0, L)$.
We will think of the first coordinate as a time variable, so we write coordinates as $x = (\tau, y)$ on $\R \times (0, L)$.
Defining $v = u - \phi$ in \eqref{eq:main}, we can write
\begin{equation}
  \label{eq:recentered}
  0 = -\Delta v - [f(v + \phi) - f(\phi)] = -\partial_\tau^2 v - \m{L} v + g(y, v)
\end{equation}
for a nonlinear part $g(y, v)$ satisfying $g(y, 0) = \partial_v g(y, 0) = 0.$
We view \eqref{eq:recentered} as a second-order ODE in an infinite dimensional Banach space.
To make it first order, let $w \coloneqq \partial_\tau v$, so that
\begin{equation}
  \label{eq:vector-equation}
  \partial_\tau
  \begin{pmatrix}
    v\\
    w
  \end{pmatrix} =
  \begin{pmatrix}
    0 & 1\\
    -\m{L} & 0
  \end{pmatrix}
  \begin{pmatrix}
    v\\
    w
  \end{pmatrix}
  +
  \begin{pmatrix}
    0\\
    g(y, v)
  \end{pmatrix}.
\end{equation}
Let $z \coloneqq (v, w)^\top$, let $B$ denote the square linear operator above, and define the nonlinear operator $\m{N}(z) \coloneqq \big(0, g(y, v)\big)^\top$.
Then \eqref{eq:vector-equation} becomes
\begin{equation}
  \label{eq:first-order}
  \partial_\tau z = B z + \m{N}(z).
\end{equation}
By Lemma~\ref{lem:simple-unstable}, $\Spec(-\m{L}) = \{-\lambda\} \cup \m{P}$ for some $\lambda > 0$ and $\m{P} \subset \R_+$ satisfying $\inf \m{P} > 0$.
A short computation shows that the spectrum of $B$ consists of the ``square-root'' of $\Spec(-\m{L})$ in $\C$:
\begin{equation*}
  \Spec B = \{\pm \iu \sqrt{\lambda}\} \cup (\pm \sqrt{\m{P}}).
\end{equation*}
Thus $B$ has precisely two (conjugate) imaginary eigenvalues and the remainder of its spectrum lies in the complement of a strip about the imaginary axis.
Fischer studied dynamical systems satisfying these hypotheses in~\cite{Fischer}; his results more or less immediately imply Proposition~\ref{prop:cylinder}.
\begin{proof}[Proof of Proposition~\textup{\ref{prop:cylinder}}]
  By Theorem~5.1 of~\cite{Fischer}, there exists $\delta_0 > 0$ such that the solutions of \eqref{eq:first-order} satisfying $\norm{z}_{L_\tau^\infty H_y^1} \leq \delta_0$ constitute a two-dimensional manifold.
  (The dimension equals the number of imaginary eigenvalues of $B$ with multiplicity.)
  Moreover, by Theorem~6.1 of~\cite{Fischer}, there exists $\delta_1 \in (0, \delta_0]$ such that every solution satisfying $\norm{z}_{L_\tau^\infty H_y^1} \leq \delta_1$ is periodic in $\tau$.

  It only remains to show that such solutions are not, in fact, constant in $\tau$.
  To see this, recall that in the proof of Lemma~\ref{lem:simple-unstable}, there exists $\ubar{\al} \in (0, \al^*)$ such that $\dot{L}_\al < 0$ when $\al \in (0, \bar{\al}]$.
  Because $s$ is increasing in $\al,$ it follows that $\phi$ is the unique positive solution of \eqref{eq:main} on $(0, L)$ satisfying $\phi \leq s_{\bar{\al}}$.
  That is, \emph{small} solutions are unique.
  Thus for $\delta_1 \ll 1$, the only constant-in-$\tau$ solution $z$ of \eqref{eq:first-order} satisfying $\norm{z}_{L_\tau^\infty H_y^1} \leq \delta_1$ is $0.$
  Because there is an entire two-dimensional manifold of other small, periodic solutions, we see that \eqref{eq:main} admits a positive solution $u = \phi + v$ on the strip $\R \times (0, L)$ that is periodic but not constant in $\tau$.
\end{proof}
The above results show that instability in the cross-sectional problem can lead to a profusion of solutions in the cylinder.
We are naturally interested in complementary results establishing uniqueness in cylinders under suitable conditions.
Uniqueness in the cross-sectional problem is of course a necessary condition.
If we further assume that $0$ is strictly unstable, Proposition~\ref{prop:cylinder-unique} yields uniqueness in the corresponding cylinder.
It is unclear whether this additional hypothesis is necessary.
If $0$ is strictly \emph{stable} and the positive bounded cross-sectional solution $u$ is unique, index arguments imply that $u$ is marginally stable.
This is thus a delicate nongeneric situation, and we are not presently able to treat it.

\section{The stable-compact method}
\label{sec:stable-compact}

As noted in the introduction, many of the above proofs follow what we term the ``stable-compact method.''
Here, we clarify this approach by systematically examining our arguments through this lens.

The method relies on a decomposition of the domain into two parts, stable and compact.
These terms refer to solutions, not to the parts themselves.
In particular, the compact part is not necessarily (pre)compact as a subset of $\R^d$.
Rather, solutions on that part enjoy certain compactness properties, and similarly for the stable part.%
\footnote{\,``Compactiferous'' might be a more precise term for the compact part, in the sense that it is compact-bearing, not necessarily compact itself.}
We note that the decomposition is far from unique; for example, one can modify the division in a region compactly contained in $\Omega$ without disrupting the argument.

The precise forms of stability and compactness vary with the application.
On the stable part, solutions and certain transformations thereof lie near a linearly stable profile.
It follows that associated linear operators obey a maximum principle, which allows us to compare solutions and their transformations.
On the compact part, we deform one solution via a one-parameter family of transformations and compare it with another.
Compactness, which depends on the deformation, allows us to contradict the strong maximum principle unless uniqueness holds.
As in the moving plane method, this approach yields other qualitative properties as well.
  
\subsection{Examples}

To apply the method to a particular problem, we identify a deformation, an associated notion of compactness, and a source of stability in regions where compactness fails.
We describe this process in each of our main uniqueness arguments.

\subsubsection*{Exterior-star domains}
An exterior-star domain $\Omega$ is monotone with respect to dilation about the star center.
Moreover, if we compose a solution with a dilation, it becomes a subsolution.
This monotonicity allows us to use dilation as the deformation.
Dilation by a bounded factor provides a corresponding notion of compactness.
Given a large constant $\kappa$, the compact region becomes $\Omega \setminus (\kappa \Omega)$.

The corresponding stable part is $\kappa\Omega$; we choose $\kappa \gg 1$ to ensure stability.
Indeed, when $\kappa$ is large, the stable part is far from the original boundary $\partial \Omega$.
As a consequence, all positive bounded solutions are close to $1$ on $\kappa \Omega.$
Because $1$ is the stable root of the reaction $f$, solutions then stable on $\kappa \Omega$ and obey a maximum principle.

\subsubsection*{Large dilations}
Our argument on large dilations is a ``degenerate'' application of the stable-compact method, as we make no use of a deformation or compactness.
Rather, for a fixed reaction and $\kappa \gg 1$, solutions on $\kappa \Omega$ locally resemble the unique half-space solution.
This solution is stable, so we can linearize around it to prove a maximum principle (and thus uniqueness) for solutions of \eqref{eq:main} on $\kappa \Omega$.

\subsubsection*{AUL epigraphs}
Epigraphs are monotone with respect to vertical translation, so we use sliding as our deformation.
Then, points within a bounded height of $\partial \Omega$ can serve as the compact region.
(In fact, we are able to use a smaller set in the proof, which simplifies the argument.)

We are left searching for stability.
By design, asymptotically uniformly Lipschitz epigraphs resemble (rotations of) uniformly Lipschitz (UL) epigraphs at infinity.
We show that the unique solution on a UL epigraph is strictly stable.
Linearizing about these profiles, we can conclude that solutions on AUL epigraphs satisfy a maximum principle outside a bounded set.
In short, the far-field is stable and the near-field is compact and equipped with the sliding deformation.

We highlight one subtlety: our proof of strict stability on UL epigraphs (Proposition~\ref{prop:uniform-stability}) is itself an example of the stable-compact method.
We decompose a UL epigraph into regions far from and near to the boundary, and we use stability on the former and compactness on the latter.
There is no need to employ a deformation when proving stability (as opposed to uniqueness).
This demonstrates the flexibility of the stable-compact framework---it sheds light on multiple qualitative properties of elliptic equations.

\subsubsection*{Pockets}
In the opposite direction, Theorem~\ref{thm:pocket} states that we can disrupt uniqueness by attaching a suitable ``pocket'' to a domain.
This demonstrates the importance of the deformation.
After all, one can view the pocket as a bounded modification of the compact set, so that the ``stable-compact'' structure remains intact.
However, the pocket does disrupt the continuous deformation (dilation or sliding in the examples above), so we cannot prove uniqueness.

\subsection{The method of moving planes}
As noted in the introduction, we take inspiration from the method of moving planes developed in~\cite{Alexandrov,Serrin,GNN}.
And indeed, the formulation of this method by the first author and Nirenberg~\cite{BN91} can be viewed as an early exemplar of the stable-compact framework.
We describe this correspondence through the lens of \cite[Theorem~1.3]{BN91}, which loosely states that solutions of \eqref{eq:main} inherit the reflection symmetries of their domains.

The proof of that result uses reflection to reduce the problem to the following question, whose technical details we elide.
Let $(D_\tau, w_\tau)_{\tau}$ be a continuous family of bounded domains $D_\tau$ and elliptic solutions $w_\tau \colon \bar D_\tau \to \R$ that are nonnegative on the boundary.
Does $w_{\tau_0}|_{D_{\tau_0}} > 0$ imply the same for nearby $\tau$?

The answer is affirmative.
Indeed, the first author and Nirenberg observe that for $\tau$ near $\tau_0$, there is a fixed compact set $K \subset D_\tau$ such that $\abs{D_\tau \setminus K} \ll 1$.
It follows from the Alexandrov--Bakelman--Pucci inequality that $D_\tau \setminus K$ supports a maximum principle (see~\cite[Theorem~1.1]{BNV}).
This allow one to transfer positivity on $K$ (a consequence of continuity and compactness) to positivity on $D_\tau$.
In our language, $D_\tau \setminus K$ is the stable part (for it supports a maximum principle) and $K$ is the compact part.
We employ a similar strategy in the proof of Theorem~\ref{thm:irregular} below.

\subsection{Strong-KPP reactions}
\label{sec:strong-KPP-Lip}

To further illustrate the stable-compact method, we interpret our earlier work \cite{BG24} through this new lens.
In~\cite{BG24}, we introduced the following terminology:
\begin{definition}
  \label{def:KPP}
  A positive reaction $f$ is \emph{weak-KPP} if $f'(s) \leq f'(0) s$ for all $s \in [0, 1]$.
  It is \emph{strong-KPP} if $f(s)/s$ is strictly decreasing in $s \in (0, 1]$.
\end{definition}
Rabinowitz showed that if $\Omega$ is smooth and bounded and $f$ is strong-KPP, then \eqref{eq:main} admits at most one positive solution~\cite{Rabinowitz}.
In~\cite{BG24}, we extended this result to very general unbounded domains.
Although we did not use the terminology, our proof is a clear example of the stable-compact method.
We decomposed the domain $\Omega$ into a \emph{stable} part $\Omega_+$ obeying a maximum principle and a \emph{compact} part $\Omega_-$ on which positive solutions remain comparable to one another (a form of compactness).

Here, we deploy an idea from~\cite{BN91} to extend Rabinowitz' result in a different direction: we show uniqueness without requiring smoothness of the boundary.
The proof is a very simple illustration of the stable-compact method.
\begin{theorem}
  \label{thm:irregular}
  Let $\Omega$ be a bounded, connected open set with Lipschitz boundary.
  If $f$ is strong-KPP, then \eqref{eq:main} admits at most one positive solution in the class $\m{C}^2(\Omega) \cap \m{C}(\bar{\Omega})$.
\end{theorem}
\noindent
We note that this result can also be derived for more general operators using the methods of~\cite{Berestycki}.
\begin{proof}
  Let $\Lambda \coloneqq \op{Lip} f$.
  By Proposition~1.1 of~\cite{BN91} and Theorem~1.1 of~\cite{BNV}, there exists $\delta > 0$ depending on $d, \Lambda,$ and $\op{diam}\Omega$ such that
  \begin{equation*}
    \lambda(-\Delta, \Omega_+) > \Lambda
  \end{equation*}
  on any open set $\Omega_+ \subset \Omega$ satisfying $\abs{\Omega_+} \leq \delta$.
  (This is a simple consequence of the Alexandrov--Bakelman--Pucci inequality, and was first observed by Bakelman~\cite{Bakelman}.)
  Let $\Omega_+ \subset \Omega$ be such a set satisfying also $\Omega_- \coloneqq \Omega \setminus \bar{\Omega}_+ \Subset \Omega$.
  For example, $\Omega_+$ can be a thin collar around $\partial \Omega$.
  In our framework, $\Omega_+$ is the stable part and $\Omega_-$ the compact part.

  Now suppose $u_1$ and $u_2$ are two positive solutions of \eqref{eq:main}.
  Because $\Omega_- \Subset \Omega$, the strong maximum principle implies that $\inf_{\Omega_-} u_i > 0$ for each $i \in \{1, 2\}$.
  Thus
  \begin{equation*}
    \ubar{\mu} \coloneqq \inf\{\mu > 0 \mid u_1 \leq \mu u_2 \text{ in }\Omega_-\}
  \end{equation*}
  is finite.
  Suppose for the sake of contradiction that $\ubar{\mu} > 1$.
  Let $w \coloneqq \ubar{\mu}u_2 - u_1$, which is nonnegative in $\Omega_-$ by continuity.
  Using the strong-KPP property (and the fact that $u_i > 0$), we have
  \begin{equation*}
    -\Delta w = \ubar{\mu} f(u_2) - f(u_1) > f(\ubar{\mu} u_2) - f(u_1) \eqqcolon q w
  \end{equation*}
  for some difference quotient $q \in L^\infty(\Omega)$ satisfying $\abs{q} \leq \Lambda$.
  It follows from the construction of $\Omega_+$ that
  \begin{equation*}
    \lambda(-\Delta - q, \Omega_+) \geq \lambda(-\Delta - \Lambda, \Omega_+) > 0.
  \end{equation*}
  Now $w \geq 0$ on $\partial \Omega_-$ and $w = 0$ on $\partial \Omega$, so $w \geq 0$ on $\partial\Omega_+ = \partial \Omega \cup \partial \Omega_-$.
  Since $\lambda(-\Delta - q, \Omega_+) > 0$, Theorem~1.1 of~\cite{BNV} (a maximum principle on irregular bounded domains) implies that $w \geq 0$ in $\Omega_+$.
  Hence $w \geq 0$ in the entire domain $\Omega$.

  Because $-\Delta w > q w$, the strong maximum principle implies that in fact $w > 0$ in $\Omega$.
  Because $\Omega_- \Subset \Omega$, $\inf_{\Omega_-} w > 0$.
  It follows that $w \geq \eta u_2$ in $\Omega_-$ for some $\eta > 0$.
  That is, $u_1 \leq (\ubar{\mu} - \eta) u_2$ in $\Omega_-$, which contradicts the definition of $\ubar{\mu}$.

  We conclude that in fact $\ubar{\mu} \leq 1$ and $u_1 \leq u_2$ in $\Omega_-$.
  Applying the maximum principle in $\Omega_+$ as above, we see that $u_1 \leq u_2$ in $\Omega_+$ as well, and hence in the entire domain $\Omega$.
  By symmetry, $u_1 = u_2$.
\end{proof}

\section{Open problems}
\label{sec:open}

To close the paper, we briefly dwell on a number of lingering open problems.
Some are natural extensions of the results presented here, while others explore quite different directions.
It is our hope that the following problems will inspire future work in this rich subject.

\subsection{Extensions}

Several of our main results include conditions that may not be necessary.
First, Theorem~\ref{thm:ES} assumes that $\Omega^\cc$ is compact.
This is a more restrictive form of the hypothesis in Theorem~\ref{thm:ES-full} that $\Omega^\cc$ is convex at infinity.
We likewise always assume that $\Omega$ is \emph{strongly} exterior star.
\begin{question}
  \label{quest:ES}
  If $f$ is positive, does \eqref{eq:main} have exactly one solution on \emph{all} exterior-star domains?
\end{question}
\noindent
In a similar vein, we state Theorem~\ref{thm:flat} for asymptotically flat epigraphs.
More broadly, Theorem~\ref{thm:AUL} treats all asymptotically uniformly Lipschitz epigraphs.
but is a condition of this form necessary?
\begin{question}
  If $f$ is positive, does \eqref{eq:main} have exactly one solution on \emph{all} epigraphs?
\end{question}
\noindent
To these, we add one question posed in our prior work~\cite{BG24} on strong-KPP reactions.
\begin{question}
  \label{quest:strong-KPP}
  If $f$ is strong-KPP in the sense of Definition~\ref{def:KPP}, does \eqref{eq:main} have at most one solution on \emph{all} domains?
\end{question}
It seems likely to us that the answer to Questions~\ref{quest:ES}--\ref{quest:strong-KPP} is \emph{yes}.
We anticipate that the stable-compact method could make further headway on these problems, but new ideas may be required to completely resolve them. 

\subsection{Robin boundary}
In our study~\cite{BG24} of strong-KPP reactions, we were able to treat Dirichlet and Robin boundary conditions in a unified manner.
(Neumann conditions are much simpler, for the unique solution is the constant $1$; see~\cite{BHN,Rossi}.)
This is due to the fact that the deformation we employed on the compact part scaled the solution, and thus did not alter the domain.

In contrast, several of the main results of the present paper \emph{slide} the domain to treat the compact part.
This sliding leads to natural monotonicity when the boundary condition is Dirichlet.
However, this is not the case under Robin conditions.

For example, suppose $u$ and $v$ solve \eqref{eq:main} with Robin conditions on a uniformly Lipschitz epigraph.
There is no \emph{a priori} reason for $u$ and $v$ to be ordered on the boundary.
This issue persists in the interior.
If we slide $v$ vertically by a distance $y$, there is no reason that the translated solution $v_y$ should lie below $u$ on $\partial \Omega + y \tbf{e}_d.$
As a result, our proofs of uniqueness on epigraphs and exterior-star domains break down under Robin boundary conditions.
Indeed, the same challenge afflicts the proof of uniqueness on uniformly Lipschitz epigraphs in~\cite{BCN97b}.
\begin{question}
  Consider \eqref{eq:main} with Robin rather than Dirichlet boundary conditions: $u + \al \partial_\nu u = 0$ on $\partial \Omega$, where $\al \in \R_+$ and $\nu$ denotes the outward unit normal vector.
  Does uniqueness hold on the epigraphs of uniformly Lipschitz functions?
  On the complements of compact, convex sets?
\end{question}
\noindent
We pose this question in simple, approachable settings.
We are naturally interested in generalizations to all epigraphs and all exterior-star domains.

Curiously, our proof of Theorem~\ref{thm:dilation} (for large dilations) is largely unaffected by Robin data.
This is because the problem is ``purely stable'': as discussed in Section~\ref{sec:stable-compact}, it is an application of the stable-compact method with empty compact part.
Thus no sliding is required, and the proof in Section~\ref{sec:dilation} goes through with minor modification.
For example, we require a version of Lieb's eigenvalue inequality adapted to Robin conditions; this is Theorem~2.3 in~\cite{BG24}.
As a consequence, we simply assert:
\begin{proposition}
  Let $f$ be a positive reaction and let $\Omega$ be a uniformly smooth domain with Robin boundary parameter $\al \in \R_+$.
  Then there exists $\ubar{\kappa}(f, \Omega, \al) > 0$ such that for all $\kappa > \ubar{\kappa}$, \eqref{eq:main} with $\al$-Robin boundary has a unique bounded positive solution on the dilated domain $\kappa \Omega$.
\end{proposition}
\noindent
In fact, it seems likely that $\ubar{\kappa}$ can be taken independent of $\al$; this would require a somewhat more sophisticated argument.

\subsection{Drift}
In both \cite{BG24} and the present paper, we exclusively study self-adjoint problems.
This reflects a desire for simplicity, but also the fact that the theory of the generalized principal eigenvalue is rather less developed in the non-self-adjoint setting.
That said,~\cite{BR} establishes a link between the maximum principle and a different formulation $\lambda'$ of the principal eigenvalue, which differs from $\lambda$ when the operator is not self-adjoint.
It thus seems possible that a stable-compact method based on $\lambda'$ may pay dividends in non-self-adjoint problems.
In particular, we are interested in the validity of our main results in the presence of drift.
\begin{question}
  \label{quest:drift}
  Given a vector-field $q \colon \Omega \to \R^d$, suppose we replace $-\Delta$ by $-\Delta + q \cdot \nab$ in \eqref{eq:main}.
  Do variations on Theorems~\ref{thm:ES}--\ref{thm:flat} hold?
  Are additional conditions like $\nab \cdot q = 0$ (divergence-free), $(q \cdot \nu)|_{\partial\Omega} = 0$ (non-penetrating), or $\abs{q} \ll 1$ (small) helpful?
\end{question}
We note that in general, drift can lead to quite different behavior.
For example, on the line $\R$, positive reactions admit traveling wave solutions of all speeds $c \geq c_*$ for some $c_* > 0$ depending on the reaction.
Thus if $q \geq c_*$ is constant, \eqref{eq:main} admits multiple positive bounded solutions: $1$ and the traveling wave.
In this case, $q$ is divergence-free and non-penetrating (the boundary is empty), so these conditions alone do not preserve uniqueness.
On the other hand, uniqueness does hold if $\abs{q} < c_*$ (still assuming constant $q$).
For this reason, the condition $\abs{q} \ll 1$ in Question~\ref{quest:drift} seems particularly promising.

\subsection{Linearly degenerate reactions}
Throughout the paper, we have made essential use of the hypothesis that $f'(0) > 0$ (and, to a lesser extent, that $f'(1) < 0$).
Indeed, this nondegeneracy ensures that $0$ is strictly unstable on sufficiently large domains.
However, a number of applications call for reactions that satisfy $f|_{(0, 1)} > 0$ but vanish to higher order at zero.
We are naturally interested in the validity of our results in this more relaxed setting.

On the whole space $\R^d$, $1$ is the unique positive bounded solution of \eqref{eq:main} in the low-dimensional case $d = 1,2$.
Indeed, solutions of \eqref{eq:main} are superharmonic; in low dimensions, the only bounded, superharmonic functions are constants.
In contrast, when $d \geq 3$ and $f(s) = s^\beta$, there exist bounded positive solutions of \eqref{eq:main} when $\beta \geq \frac{d + 2}{d-2}$~\cite[Theorem~9.1]{QS}.
The critical case $\beta = \frac{d + 2}{d-2}$ is related to the Yamabe problem, and solutions correspond to extremizers of the Sobolev inequality~\mbox{\cite{Aubin,Talenti}}.
While this reaction does not satisfy the hypothesis $f(1) = 0$, we can scale the bounded solution to lie below $\tfrac{1}{2}$ and then modify $f$ above $\tfrac{1}{2}$ to vanish at $1$.
Thus uniqueness in \eqref{eq:main} on the whole space depends on both the dimension and the order of vanishing of $f$ near $0$.

The situation is somewhat different in the half-space.
Suppose only that $f|_{(0, 1)} > 0$ and $f(0) = f(1) = 0$.
In collaboration with Caffarelli and Nirenberg~\cite[Theorem~1.5]{BCN97a}, the first author showed that bounded solutions of \eqref{eq:main} in $\H^d$ are one-dimensional (functions of $x_d$ alone) when $d = 2, 3$.
Then by Lemma~6.1 of \cite{BG22a}, there is a unique positive bounded solution.
In particular, uniqueness holds in general in $\H^3$ but not $\R^3$.
Combining the results of~\cite{EL} with the methods of \cite{BCN97b}, one could likewise show uniqueness in coercive, uniformly Lipschitz epigraphs in dimensions $2$ and $3$.

To our knowledge, little is known when $d \geq 4$.
We note that $0$ is the only bounded nonnegative solution of \eqref{eq:main} on the half-space (of any dimension) in the pure-power case $f(s) = s^\beta$ (for any $\beta \geq 1$) \cite{CLZ}.
However, it is not clear that this forbids multiple positive bounded solutions when $f$ is modified to vanish at $1$.

These observations are by no means comprehensive, but they do indicate the complexity of the problem.
We record some of its facets in the following question.
\begin{question}
  Suppose $f(s) \sim A s^\beta$ as $s \to 0$ for some $A > 0$ and $\beta > 1$.
  Is there a critical exponent $\beta_d > 1$ such that Theorems~\ref{thm:ES}--\ref{thm:flat} hold when $\beta \in [1, \beta_d)$ and fail when $\beta > \beta_d$?
  Does the existence of $\beta_d$ depend on the dimension or domain?
\end{question}

\subsection{Specific domains}
Finally, we are interested in uniqueness on several simple domains that fall outside our main results.
\begin{question}\
  \begin{enumerate}[label = \textup{(\roman*)}, itemsep = 1ex]
  \item Given a smooth bounded cross-section $\omega \subset \R^{d-1}$, what can we say about uniqueness on the half-cylinder $\Omega = \R_+ \times \omega$ or suitable smoothings thereof?
    Can we relate uniqueness on $\Omega$ to uniqueness on $\omega$?

  \item
    Suppose $\Omega$ is the complement of two balls.
    Does \eqref{eq:main} admit a unique positive bounded solution on $\Omega$?
    Does the answer depend on the balls' radii or separation?
    We are grateful to Bassam Fayad for raising this question.
  \end{enumerate}
\end{question}
The stable-compact method provides partial results.
On half-cylinders, if \eqref{eq:main} has a unique positive solution on $\omega$ and it is \emph{strictly stable}, then one can prove stability at infinity in $\Omega$ and use sliding on a compact part to prove uniqueness.
This method breaks down, however, if the cross-sectional solution is marginally stable.
In the other direction, it seems possible that \eqref{eq:main} will admit multiple solutions on $\Omega$ if $\omega$ itself supports multiple positive solutions.

For the ``two-body problem,'' we recall that uniqueness holds outside a \emph{single} ball by Theorem~\ref{thm:ES}.
Moreover, one can adapt the proof of strict stability in Proposition~\ref{prop:uniform-stability} to show that this unique solution is strictly stable.
Using this observation, one can then show that the two-body problem is purely stable (much like the dilation problem in Section~\ref{sec:dilation}) provided the two balls are sufficiently far apart; uniqueness follows.
However, uniqueness is far from clear when the two balls are relatively near one another.

\appendix

\section{Construction of a marginally stable epigraph}
\label{app:marginal}

In this appendix, we use ODE arguments to prove Proposition~\ref{prop:marginal}.
The precise structure of the reaction will play a major role.
As noted in Section~\ref{sec:strong-KPP-Lip}, Rabinowitz showed that strong-KPP reactions (in the sense of Definition~\ref{def:KPP}) admit at most one positive solution on bounded domains~\cite{Rabinowitz}.
In contrast, \cite[Proposition~1.4]{BG24} states that weak-KPP reactions can admit multiple positive solutions.
To construct the reaction in Proposition~\ref{prop:marginal}, we examine a one-parameter family of reactions interpolating between strong- and weak-KPP endpoints.
The first appearance of nonuniqueness will correspond to marginal stability.

The heart of the matter is the study of \eqref{eq:main} on the one-dimensional interval.
\begin{lemma}
  \label{lem:small-large}
  Given $m \in \R_+$, let $\m{F}$ be a $\m{C}^2$-compact set of weak-KPP reactions satisfying $f'(0) = m$ and $f''(0) < 0$ for all $f \in \m{F}$.
  Then there exist $\ubar{L},\bar{L} \in \R_+$ satisfying $\pi m^{-1/2} < \ubar{L} \leq \bar{L}$ such that \eqref{eq:main} admits a unique positive solution on $(0, L)$ for all $L \in (\pi m^{-1/2}, \ubar{L}] \cup [\bar{L}, \infty)$ and all reactions $f \in \m{F}$.
\end{lemma}
As in Section~\ref{sec:unstable-interval}, we employ the shooting method.
Recall the ODE \eqref{eq:shooting} for the solution $\phi_\al$:
\begin{equation*}
  -\phi_\al'' = f(\phi_\al), \quad \phi_\al(0) = 0, \enspace \phi_\al'(0) = \al.
\end{equation*}
As noted in Section~\ref{sec:cylinder}, if $\varphi$ denotes the unique positive bounded solution of \eqref{eq:main} on $\R_+$ and $\al^* \coloneqq \varphi'(0)$, then $\phi_\al$ has a first zero $L_\al \in \R_+$ for all $\al \in (0, \al^*)$.
Then $\phi_\al$ solves \eqref{eq:main} on $(0, L_\al)$.
\begin{proof}
  Recall the notation $s_\al = \phi_\al(L_\al/2)$ from Section~\ref{sec:unstable-interval}, which satisfies
  \begin{equation}
    \label{eq:shooting-max-2}
    2F(s_\al) = \al^2.
  \end{equation}
  Using the substitution $s \mapsto s_\al - s$ in \eqref{eq:shooting-length}, we obtain
  \begin{equation}
    \label{eq:shooting-length-desing}
    L_\al = 2 \int_0^{s_\al} \frac{\dn z}{\sqrt{\al^2 - 2F(s_\al - z)}}\,.
  \end{equation}
  In \eqref{eq:shooting-length}, \eqref{eq:shooting-max-2} implies that the integrand is singular at the moving upper endpoint.
  We have changed variables in \eqref{eq:shooting-length-desing} so that the moving endpoint does not coincide with a singularity of the integrand.

  Because $f''(0) < 0$, $f(s) < m s$ near $0$.
  Since $f(s) \leq m s$ everywhere by hypothesis, $F(s) < \tfrac{1}{2}m s^2$.
  Using this in \eqref{eq:shooting-length-desing}, we see that $L_\al > L_*$ for all $\al > 0$.
  
  Now, ODE stability implies that $\phi_\al \to \varphi$ locally uniformly as $\al \nearrow \al^*$.
  So $s_\al \to 1$ and $L_\al \to \infty$ in this limit.
  We wish to show that there exist thresholds $\ubar{\al}$ and $\ubar{\al}$ independent of $f$ such that $0 < \ubar{\al} \leq \bar{\al} < \al^*$ and $L_\al$ is strictly increasing on $(0, \ubar{\al}]$ and $[\bar{\al}, \al^*)$.
  Letting
  \begin{equation*}
    \ubar{L} \coloneqq \inf_{\al \in [\ubar{\al}, \bar{\al}]} L_\al \And \bar{L} \coloneqq \sup_{\al \in [\ubar{\al}, \bar{\al}]} L_\al,
  \end{equation*}
  the lemma will follow.
  Indeed, any length $L \not \in [\ubar{L}, \bar{L}]$ will correspond to a region in which $\al \mapsto L_\al$ is injective, so exactly one $\phi_\al$ solves \eqref{eq:main} on $(0, L)$.
  
  In the following, we use the notation $\dot{g}$ to denote $\tfrac{\dn g}{\dn \al}$.
  Differentiating \eqref{eq:shooting-max-2} in $\al$, we find
  \begin{equation}
    \label{eq:max-deriv-2}
    \dot{s}_\al = \frac{\al}{f(s_\al)} > 0,
  \end{equation}
  so the maximum $s_\al$ is strictly increasing in $\al$.
  Differentiating the second integral in \eqref{eq:shooting-length-desing} and using \eqref{eq:max-deriv-2}, we find
  \begin{equation}
    \label{eq:length-deriv-1}
    \dot{L}_\al = \frac{2}{f(s_\al)} + 2 \al \int_0^{s_\al} \frac{f(s_\al - z)/f(s_\al) - 1}{[\al^2 - 2F(s_\al - z)]^{3/2}} \d z.
  \end{equation}
  Recall that each reaction $f$ is decreasing near $1$.
  By compactness, there exists a single $\delta \in (0, 1)$ such that $f'|_{(1 - \delta, 1]} < 0$ for all $f \in \m{F}$.  
  As $\al \nearrow \al^*$, $s_\al \to 1$, so $s_\al > 1 - \delta/2$ for $\al$ near $\al^*$.
  Hence
  \begin{equation*}
    \int_0^{\delta/2} \frac{f(s_\al - z)/f(s_\al) - 1}{[\al^2 - 2F(s_\al - z)]^{3/2}} \d z > 0
  \end{equation*}
  because the numerator is positive.
  Thus \eqref{eq:length-deriv-1} yields
  \begin{equation*}
    \dot{K}_\al > \frac{2}{f(s_\al)} - 2 \al \int_{\delta/2}^{s_\al} \frac{\dn z}{[\al^2 - 2F(s_\al - z)]^{3/2}} = \frac{1}{f(s_\al)} - \m{O}(1).
  \end{equation*}
  Since $f(s_\al) \to f(1) = 0$, we see that $\dot{K}_\al \to \infty$ as $\al \nearrow \al^*$.
  In particular, there exists $\bar{\al}(f) \in (0, \al^*)$ such that $\dot{K}_\al > 0$ on $(\bar{\al}(f), \al^*)$.
  Using the compactness of $\m{F}$, one can make $\bar{\al}$ uniform over (and independent of) $f$.

  We next consider $\al \searrow 0$.
  In the proof of Lemma~\ref{lem:simple-unstable}, we showed that the assumption $f''(0) > 0$ implies that $\dot{L} < 0$ near $\al = 0$.
  Here, we assume the opposite sign: $f''(0) < 0$.
  Thus the same calculations imply that $\dot{L} > 0$ when $\al \in (0, \ubar{\al}(f)]$ for some $\ubar{\al}(f) \in (0, \bar{\al}]$.
  Using the compactness of $\m{F}$, we can readily show that $\ubar{\al}$ can be taken independent of $f$.
  This completes the proof.
\end{proof}
We next show that uniqueness implies (weak) stability:
\begin{lemma}
  \label{lem:unique-stable}
  If \eqref{eq:main} admits a unique positive solution $u$ on a domain $\Omega$, then
  \begin{equation}
    \label{eq:weak-stability}
    \lambda(-\Delta - f'(u), \Omega) \geq 0.
  \end{equation}
\end{lemma}
\begin{corollary}
  \label{cor:unstable}
  If $\dot{L}_\beta < 0$ for some $\beta \in (0, \al^*)$, then \eqref{eq:main} admits multiple positive solutions on $(0, L_\beta)$.
\end{corollary}
\begin{proof}
  This follows from Lemmas~\ref{lem:instability} and \ref{lem:unique-stable}.
\end{proof}
\begin{proof}[Proof of Lemma~\textup{\ref{lem:unique-stable}}]
  Recall the parabolic semigroup $\m{P}$ from \eqref{eq:parabolic}.
  Because $1$ is a supersolution of \eqref{eq:main}, $\m{P}_t 1$ is nonincreasing in $t$.
  It follows that the limit $\m{P}_\infty 1$ exists and solves \eqref{eq:main}.
  Since $u < 1$ on $\Omega$, comparison implies that $\m{P}_\infty 1 \geq u$.
  By hypothesis, $u$ is the only positive solution of \eqref{eq:main}, so $\m{P}_\infty 1 = u$.
  Now suppose $u \leq v \leq 1$.
  By comparison, $\m{P}_\infty v = u$.
  That is, $u$ is dynamically stable from above; this implies \eqref{eq:weak-stability}.
\end{proof}
We require one further technical lemma.
\begin{lemma}
  \label{lem:analytic}
  If $f$ is analytic and the map $\al \mapsto L_\al$ is not injective, then there exists $\beta \in (0, \al^*)$ such that $\dot{L}_\beta < 0$.
\end{lemma}
\begin{proof}
  Suppose $f$ is analytic and $L_\al$ is not injective.
  Because $L_\al$ is differentiable, there is either a slope $\beta$ of the desired form or there is a nonempty open interval $A \Subset (0, \al^*)$ on which $L_\al$ is constant.
  Suppose for the sake of contradiction that the latter holds.
  By \eqref{eq:max-deriv}, $s_\al$ is strictly increasing in $\al$.
  Combining \eqref{eq:shooting-max} and \eqref{eq:shooting-length}, we thus see that the function
  \begin{equation}
    \label{eq:I-def}
    \m{I}(s) \coloneqq \int_0^s \frac{\dn z}{\sqrt{F(s) - F(s - z)}}
  \end{equation}
  is likewise constant on some nonempty open interval $Y \Subset (0, 1)$.
  Because $f$ is analytic, so is $F$.
  It follows that $\m{I}$ is itself analytic on $(0, 1)$.
  After all, analyticity allows us to write
  \begin{equation*}
    \frac{F(s) - F(s - z)}{z} = f(s) + z g(s, z)
  \end{equation*}
  for some analytic function $g$.
  By the compactness of $Y$, there exists $\delta > 0$ such that $\abs{zg(s, z)} \leq \tfrac{1}{2} f(s)$ for all $s \in Y$ and $z \in (0, \delta)$.
  Hence we can invert and take a square root:
  \begin{equation*}
   \left(\frac{F(s) - F(s - z)}{z}\right)^{-1/2} = \frac{1}{\sqrt{f(s)}} + z h(s, z)
  \end{equation*}
  for some analytic $h$.
  Then \eqref{eq:I-def} becomes
  \begin{equation*}
    \m{I}(s) = \frac{1}{\sqrt{f(s)}} \int_0^\delta \frac{\dn z}{\sqrt{z}} + \int_0^\delta \sqrt{z} h(s, z) \d z + \int_{\delta}^s \frac{\dn z}{\sqrt{F(s) - F(s - z)}}.
  \end{equation*}
  Each term is analytic in $s$, as desired.
  Since $\m{I}(s) \to \infty$ as $s \to 1$, $\m{I}$ is not constant.
  Thus $\m{I}$ cannot coincide with a constant on any nonempty open set.
\end{proof}
We now construct the reaction in Proposition~\ref{prop:marginal}.
\begin{proposition}
  \label{prop:marginal-interval}
  There exists a weak-KPP reaction $f$ and a length $L > 0$ such that $f'(0) > \pi^2L^{-2}$, \eqref{eq:main} admits exactly one positive solution $\phi$ on $(0, L)$, and
  \begin{equation*}
    \lambda\big(\!-\partial_x^2 - f'(\phi), (0, L)\big) = 0.
  \end{equation*}
\end{proposition}
\noindent
By scaling $f$, one can in fact arrange $L = 1$; we will not use this freedom.
We also note that this proposition implies that Lemma~\ref{lem:unique-stable} cannot be improved to \emph{strict} stability: there exist domains with uniqueness but merely marginal stability.
\begin{proof}
  By Proposition~1.4 of~\cite{BG24}, there exists a weak-KPP reaction $\ti f_1$ such that \eqref{eq:main} has multiple positive solutions on the unit interval $(0, 1)$.
  A brief examination of the proof of \cite[Proposition~1.4]{BG24} shows that one can arrange $\ti f_1'(0) > \pi$, $\ti f_1 \in \m{C}^2$, and $\ti f_1''(0) < 0$.
  Approximating $\ti{f}_1''$ in $L^\infty$ by a polynomial and integrating twice, we can find a nearby weak-KPP polynomial $f_1$ with $f_1'(0) > \pi$ and $f_1''(0) < 0$ such that \eqref{eq:main} has multiple positive solutions on the unit interval $(0, 1)$ with reaction $f_1$.
  
  Let $m \coloneqq f_1'(0) > \pi$ and define $f_0(s) \coloneqq ms(1 - s)$.
  Then $f_0$ is a strong-KPP reaction and by Theorem~1.5 of~\cite{BG24}, \eqref{eq:main} admits a positive solution on $(0, 1)$ with reaction $f_0$.
  Rabinowitz showed that this solution is unique~\cite{Rabinowitz}.

  Next, given $\tau \in [0, 1]$, let $f_\tau \coloneqq (1 - \tau) f_0 + \tau f_1$, so that $f_\tau$ interpolates between the reactions $f_0$ and $f_1$.
  The weak-KPP condition is convex, so $f_\tau$ is weak-KPP for all $\tau$.
  Moreover, because $f_i'(0) = m$ and $f_i''(0) < 0$ for each $i \in \{0, 1\}$, we have $f_\tau'(0) = m$ and $f_\tau''(0) < 0$ for all $\tau \in [0, 1]$.
  The family $\m{F} \coloneqq \{f_\tau\}_{\tau \in [0, 1]}$ is clearly compact in $\m{C}^2$, so it satisfies the hypotheses on $\m{F}$ in Lemma~\ref{lem:small-large}.
  Let $L_* \coloneqq \pi m^{-1/2} < 1$.
  Then Lemma~\ref{lem:small-large} provides $L_* < \ubar{L} \leq \bar{L} < \infty$ such that \eqref{eq:main} has a unique positive solution on $(0, L)$ whenever $L \in (L_*, \ubar{L}] \cup [\bar{L}, \infty)$ and $f \in \m{F}$.
  Note that by the choice of $f_1$, $\ubar{L} < 1 < \bar{L}$.

  Let $\m{T} \subset [0, 1]$ denote the set of $\tau$ for which there exists $L(\tau)$ such that \eqref{eq:main} admits multiple positive solutions on $\big(0, L(\tau)\big)$ with reaction $f_\tau$.
  We claim that $\m{T}$ is open.
  To see this, take $\tau \in \m{T}$ and note that $L_\al$ is not injective at the value $L(\tau)$.
  As the convex combination of two polynomials, $f_*$ is a polynomial and hence analytic.
  It follows from Lemma~\ref{lem:analytic} that
  \begin{equation}
    \label{eq:length-decreasing}
    \inf \dot L_\al < 0 \ForAll \tau \in \m{T}.
  \end{equation}
  Since the family $f_\tau$ is smooth in $\tau$, the is an open neighborhood $U \ni \tau$ such that for all $\sigma \in U$, $\dot{L}_\beta(f_\sigma) < 0$ (where we make the dependence on $f$ explicit for clarity).
  By Corollary~\ref{cor:unstable}, $U \subset \m{T}$.
  That is, $\m{T}$ is open.
  
  Now define
  \begin{equation*}
    \tau_* \coloneqq \inf \m{T}
  \end{equation*}
  and $f_* \coloneqq f_{\tau_*}$.
  We claim that $f_*$ is the desired reaction.
  Noting that $0 \not \in \m{T}$ by construction, $\tau_*$ lies on the boundary of $\m{T}$, and in particular $\tau_* \not\in \m{T}$.
  It follows that \eqref{eq:main} admits precisely one solution with reaction $f_*$ on every interval $(0, L)$ with $L \in (L_*, \infty)$.

  By Lemma~\ref{lem:small-large}, for all $\tau \in [0, 1]$ we have uniqueness on lengths outside $[\ubar{L}, \bar{L}]$.
  Let $A \subset (0, \al^*)$ be a compact interval whose image under $\al \mapsto L_\al$ contains $[\ubar{L}, \bar{L}]$.
  Then
  \begin{equation*}
    \inf_{A^\cc} \dot{L}_\al \geq 0 \ForAll \tau \in [0, 1].
  \end{equation*}
  Hence \eqref{eq:length-decreasing} implies that
  \begin{equation*}
    \inf_A \dot{L}_\al < 0 \ForAll \tau \in \m{T}.
  \end{equation*}
  By Lemma~\ref{lem:instability},
  \begin{equation*}
    \inf_A \lambda\big(\!-\partial_x^2 - f_\tau'(\phi_\al), (0, L_\al)\big) < 0 \ForAll \tau \in \m{T}.
  \end{equation*}
  As noted in the proof of Lemma~\ref{lem:simple-unstable}, \cite[Proposition~2.3(vii)]{BR} implies that $\lambda$ is continuous in the potential and the length.
  Approaching $\tau_*$ from within $\m{T},$ it follows that there exists $\al \in A$ such that
  \begin{equation*}
    \lambda\big(\!-\partial_x^2 - f_*'(\phi_\al), (0, L_\al)\big) \leq 0.
  \end{equation*}
  On the other hand, we have uniqueness on $(0, L_\al)$, so by Lemma~\ref{lem:unique-stable},
  \begin{equation*}
    \lambda\big(\!-\partial_x^2 - f_*'(\phi_\al), (0, L_\al)\big) \geq 0.
  \end{equation*}
  Therefore
  \begin{equation*}
    \lambda\big(\!-\partial_x^2 - f_*'(\phi_\al), (0, L_\al)\big) = 0,
  \end{equation*}
  as desired.
  Additionally, $f'(0) = m = \pi^2L_*^{-2} > \pi^{-2} L_\al^{-2}$ because $L_\al \geq \ubar{L} > L_*$.
\end{proof}
We can finally construct an \emph{epigraph} with at best marginal stability.
\begin{proof}[Proof of Proposition~\textup{\ref{prop:marginal}}]
  Let $f$ and $L$ be as in Proposition~\ref{prop:marginal-interval}, and let $\phi$ denote the unique positive solution of \eqref{eq:main} on $(0, L)$, which is marginally stable.
  Let $\Omega \subset \R^2$ have the form in Figure~\ref{fig:counter}(b), so that $\Omega$ includes a sequence of ever deeper wells of limiting width $L$.
  Then the cylinder $\Gamma \coloneqq (0, L) \times \R$ is a local limit of $\Omega$.
  Let $u$ be a positive bounded solution of \eqref{eq:main} on $\Omega$.
  By Lemma~\ref{lem:product},
  \begin{equation}
    \label{eq:cylinder-unstable}
    \lambda(-\Delta, \Gamma) = \lambda\big(\!-\Delta, (0, L)\big) = \frac{\pi^2}{L^2} < f'(0).
  \end{equation}
  Hence by Lemma~4.5 of~\cite{BG24}, $u$ does not vanish locally uniformly in the limit to $\Gamma$.
  Let $u^* > 0$ be a subsequential limit of $u$ on the limit domain $\Gamma$, so $u^*$ solves \eqref{eq:main} on $\Gamma$.
  Write coordinates on $\Gamma$ as $(x', y) \in (0, L) \times \R$.
  We claim that $u^*(x', y) = \phi(x')$.
  To see this, observe that $u^* \leq 1$, so recalling the parabolic semigroup $\m{P}$ from \eqref{eq:parabolic}, the comparison principle yields $u^* \leq \m{P}_\infty 1$.
  Because $1$ is independent of $y$, so is $\m{P}_\infty 1$.
  Thus $\m{P}_\infty 1$ solves \eqref{eq:main} on $(0, L)$, and hence is $\phi$ by the reasoning from the proof of Lemma~\ref{lem:unique-stable}.

  For a lower bound, we observe that by \eqref{eq:cylinder-unstable} and Proposition~2.3(iv), there exist $H \in \R_+$ and $\delta \in (0, L/2)$ such that
  \begin{equation*}
    \lambda\big(-\Delta, (\delta, L - \delta) \times (0, H)\big) < f'(0).
  \end{equation*}
  Let $\psi$ be the principal eigenfunction on the rectangle $R \coloneqq (\delta, L - \delta) \times (0, H)$.
  Because $f'$ is continuous, there exists $\bar \eps \in (0, 1)$ such that $\eps \psi$ is a subsolution of \eqref{eq:main} on $\Gamma$ for all $\eps \in [0, \bar{\eps}]$.
  Because $u^* > 0$, there exists $\eps \in (0, \bar{\eps}]$ such that $u^* \geq \eps \psi$.
  Raising $\eps$, the strong maximum principle implies that $u^* \geq \bar{\eps} \psi$ (because $u^*$ cannot touch any $\eps \psi$).
  Sliding $R$ in $y$, the strong maximum principle further implies that
  \begin{equation*}
    u^*(x', y) \geq \theta(x') \coloneqq \sup_{y \in (0, H)} \psi(x', y).
  \end{equation*}
  Note that $\theta \geq 0$ and $\theta \not\equiv 0$.
  As the supremum of subsolutions, $\theta$ is itself a subsolution of \eqref{eq:main}.
  Hence the parabolic limit $\m{P}_\infty \theta$ exists is a positive solution of \eqref{eq:main} on $(0, L)$.
  By uniqueness, $\m{P}_\infty \theta = \phi$.
  Then comparison yields $u^* \geq \m{P}_\infty \theta = \phi$.
  So indeed $u^* = \phi$.

  By Lemma~2.4 of \cite{BG24}, the principal eigenvalue can only increase along limits:
  \begin{equation*}
    \lambda(-\Delta - f'(u), \Omega) \leq \lambda(-\Delta - f'(u^*), \Gamma) = \lambda(-\Delta - f'(\phi), \Gamma).
  \end{equation*}
  (The lemma is stated only for the operator $-\Delta$, but the proof applies to sequences of operators with potentials as well.)
  By Lemma~\ref{lem:product} and Proposition~\ref{prop:marginal-interval}, we find
  \begin{equation*}
    \lambda(-\Delta - f'(u), \Omega) \leq \lambda(-\Delta - f'(\phi), \Gamma) = \lambda\big(\!-\Delta - f'(\phi), (0, L)\big) = 0.\qedhere
  \end{equation*}
\end{proof}

\printbibliography
\end{document}